\documentclass[11pt]{amsart}
\usepackage{mathrsfs}
\usepackage{geometry}
\usepackage{graphicx}
\usepackage{amssymb}
\usepackage{epstopdf}
\usepackage{amsmath,amscd}
\usepackage{amsthm}
\usepackage{enumerate}
\usepackage{url,verbatim}
\usepackage{subfig}

\usepackage[normalem]{ulem}
\usepackage{fixme}

\RequirePackage[colorlinks,citecolor=blue,urlcolor=blue]{hyperref}
\usepackage{breakurl}
\theoremstyle{plain}
\newtheorem{theorem}{Theorem}
\newtheorem{definition}[theorem]{Definition}
\newtheorem{lemma}[theorem]{Lemma}
\newtheorem{proposition}[theorem]{Proposition}

\newtheorem{example}[theorem]{Example}
\newtheorem{assumption}[theorem]{Assumption}
\newtheorem{remark}[theorem]{Remark}

\newtheorem{construction}[theorem]{Construction}

\newcommand\ol{\overline}

\newcommand\EE{{\mathbb E}}

\newcommand\RR{{\mathbb R}}
\newcommand\ZZ{{\mathbb Z}}
\newcommand\NN{{\mathbb N}}

\newcommand\si{\sigma}

\newcommand\q{\quad}
\newcommand\resp{respectively}

\newcommand\Si{\Sigma}

\renewcommand\ell{l}

\newcommand\GT{\mathbb{G}\mathbb{T}}
\newcommand\CC{\mathbb{C}}
\newcommand\bm{\mathbf{m}}
\newcommand\SH{\mathrm{SH}}

\newcounter{mycount}

\numberwithin{equation}{section}
\numberwithin{theorem}{section}
\numberwithin{figure}{section}

\title[Schur function and limit shape of dimers]{Schur function at  general points and limit shape of perfect matchings on contracting square hexagon lattices with piecewise boundary conditions}

\date{}

\author{Zhongyang Li}
\address{Department of Mathematics,
University of Connecticut,
Storrs, Connecticut 06269-3009, USA}
\email{zhongyang.li@uconn.edu}
\urladdr{\url{https://mathzhongyangli.wordpress.com}}

\begin{document}
\maketitle

\begin{abstract}
We obtain a new formula to relate the value of a Schur polynomial with variables $(x_1,\ldots,x_N)$ with values of Schur polynomials at $(1,\ldots,1)$. This allows to study the limit shape of perfect matchings on a square hexagon lattice with periodic weights and piecewise boundary conditions. In particular, when the edge weights satisfy certain conditions, asymptotics of the Schur function imply that the liquid region of the model in the scaling limit has multiple connected components, while the frozen boundary consists of disjoint cloud curves.
\end{abstract}

\section{Introduction}

Schur polynomials, named after Issai Schur, are a class of symmetric polynomials indexed by decreasing sequences of non-negative integers, which form a linear basis for the space of all symmetric polynomials; see \cite{IGM15}. Besides their applications in representation theory, Schur polynomials also play an important role in the study of integrable lattice models in statistical mechanics (see \cite{AB07,AB11}). One example of such a model is the dimer model, or equivalently, random tiling model.

A \emph{dimer configuration}, or a \emph{perfect matching}, is a subset of the set of edges of a graph in which each vertex is incident to exactly one edge. A \emph{two-dimensional dimer model} is a probability measure on dimer configurations of a plane graph. Two-dimensional dimer models are exactly solvable models, in the sense that one can exactly compute the number of configurations and the local statistics by algebraic methods. Such a property and the connection of this model with several other models in statistical mechanics, including the Ising model (\cite{Fi66,ZLsp,ZL12}) and the 1-2 model (\cite{ZL14,ZL141,GL17,GL15}) put the dimer model at the intersection of several branches of mathematics (probability, combinatorics, representation theory, algebraic geometry), as well as statistical physics and computer science.

The weighted dimer model has been studied extensively by developing the techniques initiated by Temperley, Fisher and Kasteley (\cite{TF61,Ka61}) and analyzing the weighted adjacency matrix of the underlying graph, and spectacular results were obtained including the phase transition (\cite{KOS06,KO06}), conformal invariance (\cite{RK00,RK01,RK14,Li13}), and the limit shape (\cite{OR01,KO07}).
Recently the uniform dimer models on the hexagonal lattice or the square grid were studied by analyzing Schur polynomials. As a determinantal process, the correlation kernel for the uniform dimer model can be computed explicitly as a double integral (see \cite{AO01,LP142}) - this implies the limit shape result (law of large numbers; see \cite{LP14}) and the convergence of height fluctuations to a Gaussian free field (central limit theorem; see \cite{LP142}) in the scaling limit. The asymptotics of Schur polynomials in a neighborhood of $(1,\ldots,1)$ were studied in (\cite{GP15,bg,bg16}), and the limit shape and height fluctuations were obtained for the uniform dimer model on the hexagonal lattice, the uniform dimer model on the square grid (\cite{bk}), and certain periodic dimer model on the square-hexagon lattice with period of edge weights $1\times n$ (\cite{BL17,Li182,ZL20}).

The main aim of this paper is to study questions concerning limit shapes of two-dimensional dimer models. More precisely, to each random perfect matching we associate a height function - a function that assigns an integer to each face of the plane graph. When the plane graphs become larger and larger, we rescale these graphs such that the rescaled graphs approximate a certain simply-connected domain in the plane, then evidence has been amassed that the rescaled height functions are governed by ``laws of large numbers'', and converge to some naturally defined shapes. These questions have origins from the observations that the uniform random domino tilings of a large Aztec diamond (a subgraph of the 2D square grid consisting of all squares whose centers $(x,y)$ satisfy $|x|+|y|\leq n$) tends to be non-random outside a circle tangent to the boundary of the graph. This circle is called the ``arctic circle'', which is an example of a frozen boundary.

This paper is a continuation of \cite{BL17}. In \cite{BL17}, we studied the $1\times n$ periodic dimer model on the square-hexagon lattice where the boundary condition is also periodic in the sense that each remaining vertex on the boundary is followed by $(m-1)$ removed vertices, where $m\geq 1$ is a fixed positive integer. One difference between the uniform and the $1\times n$ periodic dimer model is when computing their partition function (weighted sum of all the configurations), the former can computed by the value of Schur functions at $(1,\ldots,1)$, and the later can be computed by the value of Schur function at a point depending on edge weights. When the boundary condition satisfies the condition that  each remaining vertex on the boundary is followed by $(m-1)$ removed vertices, there is an explicit formula to compute the corresponding Schur function at a generic point. The dimer model on similar graphs were also studied in \cite{BF15,bbccr}.

In this paper, we study the dimer model on a contracting square-hexagon lattice with piecewise boundary conditions. More precisely, the boundary can be divided into finitely many segments; each segment consists of either only remaining vertices or only removed vertices; the segments consisting of only remaining vertices and the segments consisting of only removed vertices are alternate; the length of each segment grows linearly as the size of the graph grows. The main tool used to study such a model is a formula we obtained to relate the value of a Schur function at a generic point to the values of Schur functions at $(1,\ldots,1)$, which gives the asymptotics of the Schur function at a generic point when the boundary condition is piecewise and the edge weights are periodic by finding a leading term in the formula. When the edge weights satisfy certain conditions, from the asymptotics of the Schur function, we obtain the surprising results that the liquid region of the model in the scaling limit has multiple connected components, whose boundary consists of disjoint cloud curves.

The organization of the paper is as follows.  In Section \ref{mr}, we review the definitions and summarize the main results proved in the paper. In Section \ref{fs}, we prove a combinatorial formula which relate the value of a Schur function at a generic point to the values of Schur functions at $(1,\ldots,1)$. In Section \ref{s4}, we study the asymptotics of Schur polynomials at a generic point by analyzing the combinatorial formula proved in Section \ref{fs}. In Section \ref{s5}, we obtain an explicit integral formula for the moments of the limit of counting measure for the $1\times n$ periodic dimer model on a contracting square-hexagon lattice with piecewise boundary conditions. In Section \ref{s7}, we discuss the existence of the frozen region, which is the region where each type of edges has either probability 0 or probability 1 to occur. For certain special cases, we find explicitly the parametric equation of frozen boundary (which is the boundary of the frozen region), and show that the frozen boundary is a union of $n$ disjoint cloud curves, where $n$ is the size of a period. In Section \ref{s8}, we give concrete examples to illustrate combinatorial formulas to compute Schur functions proved in Section \ref{fs}.

\section{Main Results}\label{mr}

In this section, we define the main objects to be studied in this paper, including the Schur function, the square-hexagon lattice and the perfect matching. Then we state the main results of this paper.

\subsection{Partitions, counting measure and Schur functions}

Let $l$ be a positive integer. Throughout this paper, we shall use the following notation:
\begin{eqnarray*}
[l]=\{1,2,\ldots,l\}
\end{eqnarray*}

\begin{definition}
  A \emph{partition} of length $N$ is a sequence of nonincreasing, nonnegative integers
  $\mu=(\mu_1\geq \mu_2\geq \ldots \geq\mu_N\geq 0)$. Each $\mu_k$ is a
  \emph{component} of the partition $\mu$. The length $N$ of the partition $\mu$
  is denoted by $\ell(\mu)$. The \emph{size} of a partition $\mu$ is
  \begin{equation*}
    |\mu| = \sum_{i=1}^N \mu_i.
  \end{equation*}

  We denote by $\GT^+_N$ the
  subset of length-N partitions.
\end{definition}

A graphic way to represent a partition $\mu$ is through its
\emph{Young diagram} $Y_\mu$, a collection of $|\mu|$ boxes arranged on
non-increasing rows aligned on the left: with
$\mu_1$ boxes on the first row, $\mu_2$ boxes on the second row,\dots $\mu_N$
boxes on the $N$th row. Some rows may be empty if the corresponding $\mu_k$ is
equal to 0. The correspondence between partitions of length $N$ and
Young diagrams with $N$ (possibly empty) rows is a bijection.

\begin{definition}
 Let $Y,W$ be two Young diagrams. We say that $Y\subset W$ \emph{differ by a
 horizontal strip} if the
  collection of boxes in $Z=W\setminus Y$ contains at most one box in every
  column. We say that they \emph{differ by a vertical strip} if $Z$ contains at
  most one box in every row.

  We say that two non-negative signatures $\lambda$ and $\mu$ \emph{interlace}, and
  write $\lambda \prec \mu$ if $Y_\lambda\subset Y_\mu$ differ by a horizontal
  strip. We say they \emph{cointerlace} and write $\lambda\prec'\mu$ if
  $Y_\lambda\subset Y_\mu$ differ by a vertical strip.
\end{definition}

\begin{definition}
  Let $\lambda\in \GT_N^+$. The \emph{rational Schur function} $s_\lambda$
  associated to $\lambda$ is the homogeneous symmetric function of degree
  $|\lambda|$ in $N$ variables defined as follows
  \begin{enumerate}
  \item If $N=1$, and $\lambda=(\lambda_1)$ then
  \begin{eqnarray*}
  s_{\lambda}(u_1)=u_1^{\lambda_1}.
  \end{eqnarray*}
  \item If $N\geq 2$, and $\lambda=(\lambda_1\geq \lambda_2\geq\ldots\geq \lambda_N)$, then
\begin{equation}
  s_{\lambda}(u_1,\ldots,u_N)=
  \frac{\det_{i,j=1,\ldots,N}(u_i^{\lambda_j+N-j})}{\prod_{1\leq
  i<j\leq N}(u_i-u_j)}.\label{slb}
\end{equation}
\end{enumerate}
\end{definition}

The Schur function defined by (\ref{slb}) is a symmetric function because the numerator and denominator are both alternating, and a polynomial since all alternating polynomials are divisible by the Vandermonde determinant.

Let $\lambda\in\GT_N^+$ be a partition of length $N$. We define the
\emph{counting measure} $m(\lambda)$ corresponding to $\lambda$ as follows.

\begin{equation}
m(\lambda)=\frac{1}{N}\sum_{i=1}^{N}\delta\left(\frac{\lambda_i+N-i}{N}\right).\label{ml}
\end{equation}

Let $\lambda(N)\in \GT_N^+$. Let $\Sigma_N$ be the permutation group of $N$ elements and let $\sigma\in \Sigma_N$. Let 
\begin{eqnarray*}
X=(x_1,\ldots,x_N).
\end{eqnarray*}
Assume that there exists $n$ between 1 and $N$ such that $x_1,...,x_n$ are pairwise distinct and $\{x_1,...,x_n\}=\{x_1,...,x_N\}$. For $j\in[N]$, let
\begin{eqnarray}
\eta_j^{\sigma}(N)=|\{k:k>j,x_{\sigma(k)}\neq x_{\sigma(j)}\}|.\label{et}
\end{eqnarray}
For $1\leq i\leq n$, let
\begin{eqnarray}
\Phi^{(i,\sigma)}(N)=\{\lambda_j(N)+\eta_j^{\sigma}(N):x_{\sigma(j)}=x_i\}\label{pis}
\end{eqnarray}
and let $\phi^{(i,\sigma)}(N)$ be the partition with length $|\{1\leq j\leq N: x_j=x_i\}|$ obtained by decreasingly ordering all the elements in $\Phi^{(i,\sigma)}(N)$. Let $\Sigma_N^{X}$ be the subgroup $\Sigma_N$ that preserves the value of $X$; more precisely
\begin{eqnarray*}
\Sigma_N^{X}=\{\sigma\in \Sigma_N: x_{\sigma(i)}=x_i,\ \mathrm{for}\ i\in[N]\}.
\end{eqnarray*}
Let $[\Sigma/\Sigma_N^X]^r$ be the collection of all the right cosets of $\Sigma_N^X$ in $\Sigma_N$. More precisely,
\begin{eqnarray*}
[\Sigma/\Sigma_N^X]^r=\{\Sigma_N^X\sigma:\sigma\in \Sigma_N\},
\end{eqnarray*}
where for each $\sigma\in \Sigma_N$
\begin{eqnarray*}
\Sigma_N^X\sigma=\{\xi\sigma:\xi\in \Sigma_N^X\}
\end{eqnarray*}
and $\xi\sigma\in \Sigma_N$ is defined by
\begin{eqnarray*}
\xi\sigma(k)=\xi(\sigma(k)),\ \mathrm{for}\ k\in[N].
\end{eqnarray*}

Below is a combinatorial formula which relates the value of a Schur function at a general point to the values of Schur functions at $(1,\ldots,1)$. The formula will be used to study limit shape of perfect matchings on a square-hexagon lattice, and moreover, the formula may also be of independent interest.

\begin{theorem}\label{p421}Under the assumptions above,
 the Schur function can be computed by the following formula
\begin{eqnarray}
\label{ses}s_{\lambda}(x_1,\ldots,x_N)&=&\sum_{\ol{\sigma}\in[\Sigma_N/\Sigma_N^X]^r} \left(\prod_{i=1}^{n}x_i^{|\phi^{(i,\sigma)}(N)|}\right)\left(\prod_{i=1}^{n}s_{\phi^{(i,\sigma)}(N)}(1,\ldots,1)\right)\\
&&\times\left(\prod_{i<j,x_{\sigma(i)}\neq x_{\sigma(j)}}\frac{1}{x_{\sigma(i)}-x_{\sigma(j)}}\right)\notag
\end{eqnarray}
where $\sigma\in \ol{\sigma}\cap \Sigma_N$ is a representative. 
\end{theorem}

Theorem \ref{p421} will be proved in Section \ref{fs}; it can also lead to asymptotic results for Schur functions at a general point $(x_1,\ldots,x_N)$; see Section \ref{s4}. In the appendix, we give concrete examples to verify Theorem \ref{p421}. It is straightforward to check that when $N\geq 2$ and $x_1,\ldots,x_N$ are pairwise distinct, the righthand side of (\ref{ses}) recovers (\ref{slb}).

For simplicity, we make the following assumptions.

\begin{assumption}\label{ap423}Let $(x_1,\ldots,x_N)$ be an $N$-tuple of real numbers at which we evaluate the Schur polynomial.
\begin{itemize}
\item $x_1>x_2>\ldots>x_n$; and
\item $N$ is an integral multiple of $n$; and.
\item $\{x_i\}_{i=1}^{N}$ are periodic with period $n$, i.e., $x_{i}=x_{j}$ for $1\leq i,j\leq N$ and $[i\mod n]=[j\mod n]$.
\end{itemize}
\end{assumption}

Let $\ol{\sigma}_0\in [\Si_N/\Si_N^X]^r$ be the unique element in $[\Si_N/\Si_N^X]^r$ satisfying the condition that for any representative $\sigma_0\in\ol{\sigma}_0$, we have
\begin{eqnarray}
x_{\si_0(1)}\geq x_{\si_0(2)}\geq\ldots\geq x_{\si_0(N)}.\label{sz}
\end{eqnarray}

\begin{assumption}\label{ap428}Assume $x_1,\ldots,x_N$ satisfy Assumption \ref{ap423}.

Let $s\in[N]$.
Assume there exists positive integers $K_1,K_2,\ldots K_s$, such that 
\begin{enumerate}
\item $\sum_{t=1}^s K_t=N$;
\item
\begin{eqnarray}
\mu_1>\ldots>\mu_s\label{mi}
\end{eqnarray}
are all the distinct elements in $\{\lambda_1,\lambda_2,\ldots,\lambda_N\}$.
\item
\begin{eqnarray*}
&&\lambda_1=\lambda_2=\ldots=\lambda_{K_s}=\mu_1;\\
&&\lambda_{K_s+1}=\lambda_{K_s+2}=\ldots=\lambda_{K_s+K_{s-1}}=\mu_2;\\
&&\ldots\\
&&\lambda_{\sum_{t=2}^{s}K_t}=\lambda_{1+\sum_{t=2}^{s}K_t}=\ldots=\lambda_{\sum_{t=1}^{s}K_t}=\mu_s;
\end{eqnarray*}
 \item Let 
\begin{eqnarray}
J_i=\{t\in[s]:\exists p\in[n],\ \mathrm{s.t.}\ x_{\sigma_0(p)}=x_i,\mathrm{and}\ \lambda_p=\mu_t\}\label{ji}
\end{eqnarray}
\begin{enumerate}
\item If $1\leq i<j\leq n$, $\ell\in J_i$, and $t\in J_j$, then $\ell<t$.
\item For any $p,q$ satisfying $1\leq p\leq s$ and $1\leq q\leq s$, and $q>p$
\begin{eqnarray*}
C_1N \leq \mu_p-\mu_q\leq C_2N
\end{eqnarray*}
where $C_1$, $C_2$ are constants independent of $N$.
\item $s$ and $n$ are fixed as $N\rightarrow\infty$.
\end{enumerate}
\end{enumerate}
\end{assumption}

Assumption \ref{ap428}(4)(a) may also be interpreted as follows. First of all, we note the following elementary lemma:
\begin{lemma}\label{l424}Let $J_i$ be defined as in (\ref{ji}).
Then for any $1\leq i<j\leq n$, $\ell\in J_i$ and $t\in J_j$, we have $\ell\leq t$.
\end{lemma}
\begin{proof}By Assumption \ref{ap423}, if $i<j$, then $x_i>x_j$. For any $\ell\in J_i$ and $t\in J_j$, there exists $1\leq p\leq N$ and $1\leq q\leq N$, such that
\begin{eqnarray*}
x_{\si_0}(p)=x_i>x_j=x_{\si_0}(q);
\end{eqnarray*}
and
\begin{eqnarray*}
\lambda_p=\mu_{\ell};\qquad \lambda_q=\mu_{t}
\end{eqnarray*}
 By the definition of $\si_0$ in (\ref{sz}), we obtain $p<q$. By the definition of partition we have $\lambda_{p}\geq \lambda_{q}$, and therefore $\mu_{\ell}\geq \mu_t$. By (\ref{mi}) $\ell\leq t$.
\end{proof}

Assumption \ref{ap428}(4)(a) actually assumes the strict inequality $l<t$ when $l,t$ satisfy the conditions of Lemma \ref{l424}.
We write $\lambda_1,\ldots,\lambda_N$ in decreasing order, and $x_{\sigma_0(1)},\ldots,x_{\sigma_0(N)}$ in decreasing order, and obtain a $2\times N$ array as follows:
\begin{eqnarray}
\label{crr}\begin{array}{cccc}\lambda_1&\lambda_2&\ldots&\lambda_N\\ x_{\sigma_0(1)}& x_{\sigma_0(2)}&\ldots& x_{\sigma_0(N)} \end{array}.
\end{eqnarray}
In the 1st row, there are exactly $s$ distinct values; while in the 2nd row, there are exactly $n$ distinct values. From (\ref{sz}) and Assumption \ref{ap423} we can see that the 2nd row of (\ref{crr}) is the same as 
\begin{eqnarray*}
x_1,\ldots,x_1,x_2,\ldots,x_2,\ldots,x_n\ldots,x_n
\end{eqnarray*}
where the 1st $\frac{N}{n}$ entries are $x_1$'s, the next $\frac{N}{n}$ entries are $x_2$'s, and so on. Similarly, by Assumption (\ref{ap428}), we can see that the 1st row of (\ref{crr}) is the same as
\begin{eqnarray*}
\mu_1,\ldots,\mu_1,\mu_2,\ldots,\mu_2\ldots,\mu_s,\ldots,\mu_s,
\end{eqnarray*}
where the 1st $K_s$ entries are $\mu_1$'s, the next $K_{s-1}$ entries are $\mu_2$'s, etc.
Then $J_i$ consists of all the indices $t\in[s]$ such that there exists a value $\mu_t$ in the 1st row of (\ref{crr}), which is in the same column as a value $x_i$ in the second row. Then Assumption \ref{ap428}(4)(a) says that for any $j\in[N-1]$, if $x_{\sigma_0(j)}>x_{\sigma_0(j+1)}$, then $\lambda_j>\lambda_{j+1}$. In other words, no index $t\in[s]$ can appear in more than one $J_i$'s for $i\in[n]$; or the collection of sets $\{J_i\}_{i\in[n]}$ are pairwise disjoint. In particular, this implies that $s\geq n$.

As we shall see later in Lemma \ref{cm}, under Assumptions \ref{ap423} and \ref{ap428}, as $N\rightarrow\infty$, the counting measures of  $\phi^{(i,\si_0)}(N)$ converges weakly to a limit measure $\bm_i$.

Let
\begin{eqnarray}
H_{\mathbf{m}_i}(u)=\int_{0}^{\ln(u)}R_{\mathbf{m}_i}(t)dt+\ln\left(\frac{\ln(u)}{u-1}\right)\label{hmi}
\end{eqnarray}
and $\mathbf{R}_{\mathbf{m}_i}$ is the Voiculescu R-transform of $\mathbf{m}_i$ given by
\begin{eqnarray*}
R_{\bm_i}=\frac{1}{S_{\bm_i}^{(-1)}(z)}-\frac{1}{z};
\end{eqnarray*}
Where $S_{\bm_i}$ is the moment generating function for $\bm_i$ given by
\begin{eqnarray}
S_{\bm_i}(z)=z+M_1(\bm_i)z^2+M_2(\bm_i)z^3+\ldots;\label{smi}
\end{eqnarray}
$M_k(\bm_i)=\int_{\RR}x^k\bm_i(dx)$; and $S_{\bm_i}^{-1}(z)$ is the inverse series of $S_{\bm_i}(z)$.
See also Section 2.2 of \cite{bg} for details.

We may further make the assumptions below

\begin{assumption}\label{ap32}Assume $x_{1,N}=x_1>0$ and $(x_{2,N},\ldots,x_{n,N})$ changes with $N$. Assume that for each fixed $N$,  $(x_{1,N},\ldots,x_{n,N})$ satisfy Assumption \ref{ap423}. Suppose that Assumption \ref{ap428} holds. Moreover, assume that
 \begin{eqnarray*}
\liminf_{N\rightarrow\infty} \frac{\log\left(\min_{1\leq i<j\leq n}\frac{x_{i,N}}{x_{j,N}}\right)}{N}\geq \alpha>0,
\end{eqnarray*}
where $\alpha$ is a sufficiently large positive constant independent of $N$.
\end{assumption}

\begin{theorem}\label{tm2}Under Assumptions \ref{ap32} and \ref{ap428}, for each given $\{a_i,b_i\}_{i=1}^{n}$, when $\alpha$ in Assumption \ref{ap32} is sufficiently large, $u_1,\ldots,u_k$ are in an open complex neighborhood of $1$, we have
\begin{eqnarray}
&&\lim_{N\rightarrow\infty}\frac{1}{N}\log \frac{s_{\lambda(N)}(u_1x_{1,N},\ldots,u_kx_{k,N},x_{k+1,N},\ldots,x_{N,N})}{s_{\lambda(N)}(x_{1,N},\ldots,x_{N,N})}=\sum_{i=1}^{k}[Q_i(u_i)]\label{fc}
\end{eqnarray}
where for $1\leq i\leq k$, 
\begin{enumerate}
\item if $[i\mod n]\neq 0$,
\begin{eqnarray*}
Q_i(u)=\frac{H_{\mathbf{m}_{i\mod n}}(u)}{n}-\frac{(n-[i\mod n])\log(u)}{n}.
\end{eqnarray*}
\item if $[i\mod n]=0$,
\begin{eqnarray*}
Q_i(u)=\frac{H_{\mathbf{m}_n}(u)}{n}.
\end{eqnarray*}
\end{enumerate}
Moreover, the convergence of (\ref{fc}) is uniform when $u_1,\ldots,u_k$ are in an open complex neighborhood of $1$.
\end{theorem}

Theorem \ref{tm2} is proved in Section \ref{s4}.

\subsection{Square-hexagon lattice}

Consider a doubly-infinite binary sequence indexed by integers
$\ZZ=\{\ldots,-2,-1,0,1,2,\ldots\}$.
\begin{equation}
  \check{a}=(\ldots,a_{-2},a_{-1},a_0,a_1,a_2,\ldots)\in\{0,1\}^{\ZZ}.\label{ca}
\end{equation}

We now define a bipartite plane graph $\mathrm{SH}(\check{a})$, called \emph{whole-plane square-hexagon lattice} associated with the  sequence
$\check{a}$. The vertex set of $\mathrm{SH}(\check{a})$ is a subset of $\frac{\ZZ}{2}\times \frac{\ZZ}{2} $.
Each vertex of $\mathrm{SH}(\check{a})$ is either black or white, and we
identify the vertices with points on the plane. 
 For $m\in \ZZ $, the black vertices have $y$-coordinate $m$; while the white
 vertices have $y$-coordinate $m-\frac{1}{2}$. We will label all the vertices
 with $y$-coordinate $t$ $\left(t\in\frac{\ZZ}{2}\right)$ as vertices in the $(2t)$th row. We further
 require that for each $m\in\ZZ$,
 \begin{itemize}
\item  each black vertex on the $(2m)$th row is adjacent to two white vertices in the $(2m+1)$th row; and
\item if $a_m=1$, each white vertex on the $(2m-1)$th row is adjacent to exactly one black vertex in the $(2m)$th row; 
 if $a_m=0$, each white vertex on the $(2m-1)$th row is adjacent to two black vertices in the $(2m)$th row. 

  \end{itemize}
 See Figure \ref{lcc}.

The square-hexagon lattice defined above is related to the
rail-yard graph; see \cite{bbccr}.

\begin{figure}
\subfloat[Structure of $\mathrm{SH}(\check{a})$ between the $(2m)$th row and the $(2m+1)$th row]{\includegraphics[width=.6\textwidth]{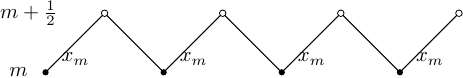}}\\
\subfloat[Structure of $\mathrm{SH}(\check{a})$ between the $(2m-1)$th row and the $(2m)$th row when $a_m=0$]{\includegraphics[width = .6\textwidth]{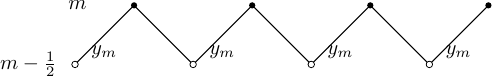}}\\
\subfloat[Structure of $\mathrm{SH}(\check{a})$ between the $(2m-1)$th row and the $(2m)$th row when $a_m=1$]{\includegraphics[width = .55\textwidth]{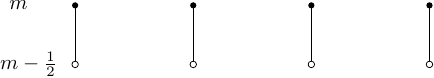}}
\caption{Graph structures of the square-hexagon lattice on the $(2m-1)$th, $(2m)$th, and $(2m+1)$th rows depend on the values of $(a_m)$. Black vertices are along the $(2m)$th row, while white vertices are along the $(2m-1)$th and $(2m+1)$th row.}
\label{lcc}
\end{figure}

We shall assign edge weights to the whole-plane square-hexagon lattice
$\SH(\check{a})$ in the following way.

\begin{assumption}\label{apew}
  For $m\in\ZZ$, we assign weight $x_m>0$ to each NE-SW edge joining the
  $(2m)$th row  to the $(2m+1)$th row of $\mathrm{SH}(\check{a})$. We assign
  weight $y_m>0$ to each NE-SW edge joining the $(2m-1)$th row to the $(2m)$th
  row of $\mathrm{SH}(\check{a})$, if such an edge exists. We assign weight $1$
  to all the other edges. 
\end{assumption}

It is straightforward to check that in the graph $\SH(\check{a})$, either all the faces on a row are hexagons, or all the faces on a row are squares, depending on the corresponding entry of $\check{a}$. A \emph{contracting square-hexagon lattice} is built from a whole-plane
square-hexagon lattice as follows:

\begin{definition}\label{dfr}
  Let $N\in \NN$. Let $\Omega=(\Omega_1,\ldots,\Omega_N)$ be an $N$-tuple of
  positive integers, such that  $1=\Omega_1<\Omega_2<\cdots<\Omega_{N}$. Set
  $m=\Omega_N-N$.
  The contracting square-hexagon lattice $\mathcal{R}(\Omega,\check{a})$ is a
  subgraph of $\mathrm{SH}(\check{a})$ with $2N$ or $2N+1$ rows of vertices.
  We shall now enumerate the rows of $\mathcal{R}(\Omega,\check{a})$
  inductively, starting from the bottom as follows:
  \begin{itemize}
    \item The first row consists of vertices $(i,j)$ with $i=\Omega_1-\frac{1}{2},\ldots,\Omega_N-\frac{1}{2}$ and $j=\frac{1}{2}$. We call this row the boundary row of $\mathcal{R}(\Omega,\check{a})$.
    \item When $k=2s$, for $s=1,\ldots N$,  the $k$th row consists of vertices $(i,j)$ with $j=\frac{k}{2}$ and incident to at least one vertex in the $(2s-1)th$ row of the whole-plane square-hexagon lattice $\mathrm{SH}(\check{a})$ lying between the leftmost vertex and rightmost vertex of the $(2s-1)$th row of $\mathcal{R}(\Omega,\check{a})$
    \item When $k=2s+1$, for $s=1,\ldots N$,  the $k$th row consists of vertices $(i,j)$ with $j=\frac{k}{2}$ and incident to two vertices in the $(2s)$th row of  of $\mathcal{R}(\Omega,\check{a})$.
  \end{itemize}
\end{definition}

  The transition from an odd row to the next even row in a contracting
  square-hexagon lattice can be of two kinds
  depending on whether vertices are connected to one or two vertices of the row
  above them.

  \begin{definition}
    \label{defI1I2}
  Let $I_1$ (resp.\@ $I_2$) be the set of indices $j$ such that vertices
  of the $(2j-1)$th row are connected to one vertex (resp.\@ two vertices) of
  the $(2j)$th row.
  In terms of the sequence $\check{a}$,

  \begin{equation*}
    I_1=\{k\in \{1,\dots,N\}\ |\ a_k=1\},\quad
    I_2=\{k\in \{1,\dots,N\}\ |\ a_k=0\}.
  \end{equation*}
\end{definition}

  The sets $I_1$ and $I_2$ form a partition of $\{1,\dots,N\}$, and we have
  $|I_1|=N-|I_2|$.

\begin{figure}
  \includegraphics{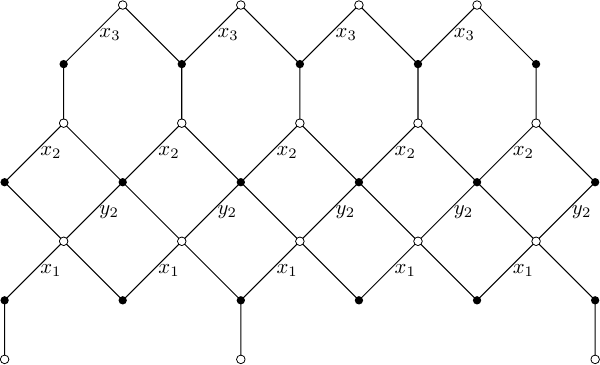}
\caption{Contracting square-hexagon lattice with $N=3$, $m=3$, $\Omega=(1,3,6), (a_1,a_2,a_3)=(1,0,1)$.}
\label{fig:SH}
\end{figure}

\subsection{Dimer model}

\begin{definition}\label{dfvl}A dimer configuration, or a perfect matching $M$ of a contracting square-hexagon lattice $\mathcal{R}(\Omega,\check{a})$ is a set of edges $((i_1,j_1),(i_2,j_2))$, such that each vertex of $\mathcal{R}(\Omega,\check{a})$ belongs to an unique edge in $M$.
 The set of perfect matchings of $\mathcal{R}(\Omega,\check{a})$ is denoted by
  $\mathcal{M}(\Omega,\check{a})$.
\end{definition}

\begin{definition} The partition function of the dimer model of a finite graph
$G$ with edge weights $(w_e)_{e\in E(G)}$ is given by
\begin{equation*}
Z=\sum_{M\in \mathcal{M}}\prod_{e\in M}w_e,
\end{equation*}
where $\mathcal{M}$ is the set of all perfect matchings of $G$. The
Boltzmann dimer probability measure on $M$ induced by the weights $w$ is
thus defined by declaring that probability of a perfect matching is equal to
\begin{equation*}
  \frac{1}{Z}\prod_{e\in M} w_e.
\end{equation*}
\end{definition}

\begin{definition}
  Let $M\in \mathcal{M}(\Omega,\check{a})$ be a perfect matching of
  $\mathcal{R}(\Omega,\check{a})$. We call an edge $e=((i_1,j_1),(i_2,j_2))\in
  M$ a \emph{$V$-edge} if $\max\{j_1,j_2\}\in\NN$ (i.e.\@ if its higher
  extremity is black) and we call it a \emph{$\Lambda$-edge}
  otherwise. In other words, the edges going upwards starting from an odd row
  are $V$-edges and those ones starting from an even row are $\Lambda$-edges. We
  also call the corresponding vertices-$(i_1,j_1)$ and $(i_2,j_2)$ $V$-vertices
  and $\Lambda$-vertices accordingly.
\end{definition}

We shall associate to each perfect matching in $\mathcal{M}(\Omega,\check{a})$
a sequence of partitions, one for each row of the graph.

\begin{construction}\label{ct}To the boundary row $\Omega=(\Omega_1<\cdots<\Omega_N)$ of a contracting
square-hexagon lattice is naturally associated a partition $\omega$
of length $N$ by:
\begin{equation*}
  \omega=(\Omega_N-N,\dotsc,\Omega_1-1).
\end{equation*}

  Let $j\in\{2,\dots,2N+1\}$. Assume that the $j$th row of
  $\mathcal{R}(\Omega,\check{a})$ has $n_j$ V-vertices and $m_j$
  $\Lambda$-vertices. The a dimer configuration at the $j$th row of $\mathcal{R}(\Omega,\check{a})$
  corresponds to a partition $\mu\in \GT_{n_j}^+$, such that 
  \begin{itemize}
  \item $\mu=(\mu_1,\ldots,\mu_{n_j})$;
  \item We label all the $V$-vertices on the $j$th row by the 1st $V$-vertex, the 2nd $V$-vertex, \ldots, the $n_j$th $V$-vertex,
  such that the $1$st $V$-vertex is the rightmost $V$-vertex on the $j$th row.
  for $1\leq k\leq n_j$, $\mu_k$ is the number of $\Lambda$-vertices to the left of the $k$th $V$-vertex.
  \end{itemize}
\end{construction}

Then we have 
\begin{theorem}[\cite{BL17} Theorem 2.13]
  For given $\Omega$, $\check{a}$, let $\omega$ be the partition associated to
  $\Omega$. Then the construction~\ref{ct} defines a
  bijection between the set of perfect matchings
  $\mathcal{M}(\Omega,\check{a})$ and the set $S(\omega,\check{a})$ of
  sequences of partitions
  \begin{equation*}
    \{(\mu^{(N)},\nu^{(N)},\dots,\mu^{(1)}, \nu^{(1)}, \mu^{(0)}\}
  \end{equation*}
  where the partitions satisfy the following properties:
  \begin{itemize}
    \item All the parts of $\mu^{(0)}$ are equal to 0;
    \item The partition $\mu^{(N)}$ is equal to $\omega$;
    \item For $0\leq i\leq N$, $\mu^{(i)}\in \GT_i^{+}$.
    \item The signatures satisfy the following (co)interlacement relations:
      \begin{equation*}
        \mu^{(N)} \prec' \nu^{(N)} \succ \mu^{(N-1)} \prec' \cdots
        \mu^{(1)} \prec' \nu^{(1)} \succ \mu^{(0)}.
      \end{equation*}
  \end{itemize}
  Moreover, if $a_m=1$, then $\mu^{(N+1-k)}=\nu^{(N+1-k)}$.
  \label{myb}
\end{theorem}

For $N\geq 1$, let $\lambda(N)\in \GT_N^+$ be the boundary partition satisfying Assumption \ref{ap428}.  Let
\begin{eqnarray*}
\Omega=(\Omega_1<\Omega_2<\ldots<\Omega_N)=(\lambda_N(N)+1,\lambda_{N-1}(N)+2,\ldots,\lambda_1(N)+N)
\end{eqnarray*}
Indeed, $\Omega_1,\ldots,\Omega_N$ are the locations of the $N$ remaining vertices on the bottom boundary of the contracting square-hexagon lattice.
Under Assumption \ref{ap428}, we may assume
\begin{eqnarray}
\Omega&=&(A_1,A_1+1,\ldots.B_1-1,B_1,\label{abt}\\
&&A_2,A_2+1,\ldots,B_2-1,B_2,\ldots,A_s,A_s+1,\ldots,B_s-1,B_s).\notag
\end{eqnarray}
where 
\begin{eqnarray*}
B_i-A_i+1=K_i.
\end{eqnarray*}
and
\begin{eqnarray*}
\sum_{i=1}^{s}(B_i-A_i+1)=N.
\end{eqnarray*}
Suppose as $N\rightarrow\infty$,
\begin{eqnarray}
A_i(N)=a_iN+o(N),\qquad B_i(N)=b_iN+o(N),\qquad \mathrm{for}\ i\in[s],\label{abi}
\end{eqnarray}
and $a_1<b_1<\ldots<a_s<b_s$ are fixed parameters independent of $N$ and satisfy $\sum_{i=1}^{s}(b_i-a_i)=1$. Under Assumption \ref{ap428}, it is straightforward to check that for $i\in[s]$
\begin{eqnarray}
b_i&=&\lim_{N\rightarrow\infty}\frac{\mu_{s-i+1}+\sum_{t=1}^{i}K_t}{N}\label{dbi}\\
a_i&=&\lim_{N\rightarrow\infty}\frac{\mu_{s-i+1}+\sum_{t=1}^{i-1}K_t}{N}\label{dai}
\end{eqnarray}

Here are the main theorems concerning the limit counting measures of partitions corresponding to dimer configurations on all the horizontal levels of a contracting square-hexagon lattice. In Theorem \ref{tm5} we give explicit integral formulas for all the moments of limit counting measures at all horizontal levels, from which we can see that the limit counting measure at each horizontal level is deterministic.

\begin{theorem}\label{tm5}Suppose Assumptions \ref{ap423}, \ref{ap32} and \ref{ap428} hold. Let $\kappa\in(0,1)$ be a positive number. Let $\rho_{\lfloor (1-\kappa)N\rfloor}$ be a probability measure on $\GT_{\lfloor (1-\kappa)N\rfloor}^+$, which is the probability measure for partitions corresponding to the random $V$-edges incident to the $\lfloor(1-\kappa)N\rfloor$th row (counting from the top) of white vertices in a dimer configuration of a contracting square-hexagon lattice $\mathcal{R}(\Omega,\check{a})$. Let $\bm[\rho_{\lfloor (1-\kappa)N\rfloor}]$ be the corresponding random counting measure. Then as $N\rightarrow\infty$, $\bm[\rho_{\lfloor (1-\kappa)N\rfloor}]$ converge in probability, in the sense of moments to a deterministic measure $\bm^{\kappa}$, whose moments are given by
\begin{eqnarray*}
\int_{\RR}x^{p}\textbf{m}^{\kappa}(dx)=
%=\sum_{i=1}^{n}\frac{1}{2(p+1)\pi \mathbf{i}}\oint_{x_{t+i}}\frac{dz}{z}\left(zQ_{\kappa}'(z)+\sum_{j=1}^{n}\frac{z}{n(z-x_{j})}\right)^{p+1}
\frac{1}{2(p+1)\pi \mathbf{i}}\sum_{i=1}^{n}\oint_{C_{1}}\frac{dz}{z}\left(zQ_{i,\kappa}'(z)+\frac{n-i}{n}+\frac{z}{n(z-1)}\right)^{p+1}
\end{eqnarray*}
where for $i\in[n]$
\begin{eqnarray*}
Q_{i,\kappa}(z)=\left\{\begin{array}{cc}\frac{1}{(1-\kappa)n}\left[
H_{\bm_i}(z)-(n-i)\log z+\kappa\sum_{l\in[n]\cap I_2}\log\frac{1+y_lzx_1}{1+y_lx_1}\right]&\mathrm{if}\ i=1\\\frac{1}{(1-\kappa)n}\left[H_{\bm_i}(z)-(n-i)\log z\right]
&\mathrm{otherwise}\end{array}\right.
\end{eqnarray*}
and for $i\geq n+1$,
\begin{eqnarray*}
Q_{i,\kappa}(z)=\left\{\begin{array}{cc}Q_{(i\mod n),\kappa}(z),&\mathrm{if}\ (i\mod n)\neq 0\\Q_{n,\kappa}(z),&\mathrm{if}\ (i\mod n)= 0 \end{array}\right.
\end{eqnarray*}
\end{theorem}

Theorem \ref{tm5} is proved in Section \ref{s5}.

The \emph{frozen boundary} of the limit shape is the boundary curve of the region where each type of edge has probability 0 and 1 to occur in the perfect matching (frozen region).
The algebraic curve we obtain for the frozen boundary has special properties,
  that can be read from its dual curve, 
  as described in the definition and the theorem below:

\begin{definition}[\cite{KO07}]
  \label{df43}
  A degree $d$ real algebraic curve $C\subset \RR P^2$ is \emph{winding} if: 
  \begin{enumerate}
    \item it intersects every line $L\subset \RR P^2$ in at least $d-2$ points
      counting multiplicity; and
    \item there exists a point $p_0\in \RR P^2$ called center, such that every
      line through $p_0$ intersects $C$ in $d$ points.
  \end{enumerate}
  The dual curve of a winding curve is called a \emph{cloud curve}.
\end{definition}

\begin{theorem}\label{tm3}Suppose Assumptions \ref{ap423}, \ref{ap32} and \ref{ap428} hold. If $|I_2\cap[n]|\in \{0,1\}$, then the frozen boundary consists of $n$ disjoint cloud curves.
\end{theorem}

Theorem \ref{tm3} is proved in Section \ref{s7}.

\section{Schur Polynomial at $(x_1,\ldots,x_N)$ and Schur polynomials at $(1,\ldots,1)$: combinatorial results} \label{fs}

In this section, we prove the combinatorial formula to compute the Schur function at $(x_1,\ldots,x_N)$ by the values of Schur functions at $(1,\ldots,1)$, as stated in Theorem \ref{p421}. We shall prove a general formula to compute the Schur polynomial (Proposition \ref{p437}) at $(w_1,\ldots,w_N)\in \CC^N$, where the variable $(w_1,\ldots,w_N)$ differs from $(x_1,\ldots,x_N)$ by at most $k$ components, and then obtain Theorem \ref{p421} as a special case when $(w_1,\ldots,w_N)=(x_1,\ldots,x_N)$. Proposition \ref{p437} will also be used to obtain asymptotical results of Schur polynomials (Theorem \ref{tm2}). We start with the following lemma.

\begin{lemma}\label{lee}For any $\xi\in \Sigma_N^X$, $\si\in \Si_N$ and $1\leq j\le N$, we have
\begin{eqnarray*}
\eta_j^{\xi\si}(N)=\eta_j^{\si}(N).
\end{eqnarray*}
\end{lemma}
\begin{proof}Since $\xi\in \Sigma_N^X$, we have
\begin{eqnarray*}
\eta_j^{\xi\si}(N)=|\{k:k>j,x_{\xi\sigma(k)}\neq x_{\xi\si(j)}\}|=|\{k:k>j,x_{\sigma(k)}\neq x_{\si(j)}\}|=\eta_j^{\si}(N).
\end{eqnarray*}
\end{proof}

\begin{lemma}\label{pee}For any $\xi\in \Sigma_N^X$, $\si\in \Si_N$ and $1\leq i\leq n$, we have
\begin{eqnarray*}
\phi^{(i,\sigma)}(N)=\phi^{(i,\xi\si)}(N),
\end{eqnarray*}
as elements in $\GT_{|\{j\in[N]:x_j=x_i\}|}$.
\end{lemma}
\begin{proof}By Lemma \ref{lee}, we have
\begin{eqnarray*}
\Phi^{(i,\sigma)}(N)=\{\lambda_j(N)+\eta_j^{\sigma}(N):x_{\sigma(j)}=x_i\}=\{\lambda_j(N)+\eta_j^{\xi\sigma}(N):x_{\xi\sigma(j)}=x_i\}=\Phi^{(i,\xi\sigma)}(N).
\end{eqnarray*}
Then the lemma follows from the fact that $\phi^{(i,\sigma)}(N)$ (\resp.\ $\phi^{(i,\xi\sigma)}(N)$) is the partition obtained by decreasingly ordering all the elements in $\Phi^{(i,\sigma)}(N)$ (\resp.\ $\Phi^{(i,\xi\sigma)}(N)$)
\end{proof}

Let $k\in[N]$. Let
\begin{eqnarray}
\label{dw}w_{i}=\left\{\begin{array}{cc}u_i&\mathrm{if}\ 1\leq i\leq k\\x_i&\mathrm{if}\ k+1\leq i\leq N\end{array}\right.
\end{eqnarray}

Assume 
\begin{eqnarray}
\label{dk}k=qn+r,\qquad \mathrm{where}\ r<n,
\end{eqnarray}
and $q,r$ are positive integers.

For $i\in[n]$, let
\begin{eqnarray}
T_i=\{j:j\in[N], x_j=x_i\}.\label{dti}
\end{eqnarray}

Let $\mathcal{S}_{T_i}$ be the subgroup of $\Sigma_N$ consisting of all  the permutations within $I_i$ while preserving all the elements outside $I_i$. Note that
\begin{eqnarray*}
\mathcal{S}_{T_i}\subset \Sigma_N^X,\qquad \forall\ i\in[n]\\
\mathcal{S}_{T_i}\cap\mathcal{S}_{T_j}=\{id\}, \qquad \forall\ i\neq j,
\end{eqnarray*}
where $id$ is the identity in $\Si_N$. Indeed, $\Sigma_N^X$ is the direct product of $\mathcal{S}_{T_n}$, $\mathcal{S}_{T_{n-1}}$,\ldots, $\mathcal{S}_{T_1}$. 

\begin{lemma}\label{l437}Let $\ol{\sigma}\in[\Sigma_N/\Sigma_N^X]^r$. Assume that $\si_1,\si_2\in\ol{\si}$ be two distinct representatives. Then we have
\begin{eqnarray}
\prod_{i<j,x_{\sigma_1(i)}\neq x_{\sigma_1(j)}}\frac{1}{w_{\sigma_1(i)}-w_{\sigma_1(j)}}=\prod_{i<j,x_{\sigma_2(i)}\neq x_{\sigma_2(j)}}\frac{1}{w_{\sigma_2(i)}-w_{\sigma_2(j)}}\label{12e}
\end{eqnarray}
\end{lemma}
\begin{proof}Assume $\si_1=\xi\si_2$, where $\xi\in \Si_N^X$. We claim $\xi$ is a product of transpositions in $\Si_N^X$. Indeed, since $\Si_N^X$ is a direct product of $\mathcal{S}_{T_1}$, $\mathcal{S}_{T_2}$,\ldots, $\mathcal{S}_{T_n}$, we have
\begin{eqnarray*}
\xi=\prod_{k=1}^n\xi_k,
\end{eqnarray*}
where $\xi_k\in \mathcal{S}_{T_k}$. For each $k\in[n]$, since $\mathcal{S}_{T_k}$ is isomorphic to $\Si_{|T_k|}$ (the permutation group of $|T_k|$ elements), $\xi_k$ is the product of transpositions in $\mathcal{S}_{T_k}\subset \Sigma_{N}^X$. Therefore for each $\xi\in\Sigma_N^X$, $\xi$ is a product of transpositions in $\Si_N^X$.

Now it suffices to show that for each transposition $\eta\in \Si_N^X$ and $\sigma\in\Si_N$,  we have
\begin{eqnarray}
\prod_{i<j,x_{\eta\sigma(i)}\neq x_{\eta\sigma(j)}}\frac{1}{w_{\eta\sigma(i)}-w_{\eta\sigma(j)}}=\prod_{i<j,x_{\sigma(i)}\neq x_{\sigma(j)}}\frac{1}{w_{\sigma(i)}-w_{\sigma(j)}}\label{ir}
\end{eqnarray}
If $\eta$ is a transposition of two elements not in $\{u_1,u_2,\ldots,u_k\}$, then
\begin{eqnarray*}
\eta\in \Si_N^W:=\{\si\in\Si_N: w_i=w_{\sigma(i)}, \forall\ i\in[N]\}. 
\end{eqnarray*}
In this case (\ref{ir}) obviously holds. 

We now check that if $\eta$ is a transposition involving elements in $\{u_1,\ldots,u_k\}$, the identity (\ref{ir}) still holds. Without loss of generality, assume that $\eta=(u_{\si(a)},w_{\si(b)})$, where $a\in[k]$, $b>a$ and $b\in[N]$. Given that $\eta\in\Si_N^X$, we must have
$x_{\si(a)}=x_{\si(b)}$.
Then
\begin{eqnarray*}
\prod_{i<j,x_{\eta\sigma(i)}\neq x_{\eta\sigma(j)}}\frac{1}{w_{\eta\sigma(i)}-w_{\eta\sigma(j)}}
:=D_1^{(\eta\si)} D_2^{(\eta\si)}\\
\prod_{i<j,x_{\sigma(i)}\neq x_{\sigma(j)}}\frac{1}{w_{\sigma(i)}-w_{\sigma(j)}}
:=D_1^{(\si)} D_2^{(\si)}
\end{eqnarray*}
where
\begin{eqnarray*}
D_1^{(\si)}&:=&\left[\prod_{i<j,\{i,j\}\cap \{a,b\}=\emptyset,x_{\si(i)}\neq x_{\si(j)}}\frac{1}{w_{\si(i)}-w_{\si(j)}}\right]\times \left[
\prod_{i<a,x_{\si(i)}\neq x_{\si(a)}}\frac{1}{w_{\si(i)}-w_{\si(a)}}\right]\\
&&\times\left[
\prod_{i<a,x_{\si(i)}\neq x_{\si(b)}}\frac{1}{w_{\si(i)}-w_{\si(b)}}\right]\times\left[
\prod_{j>b,x_{\si(j)}\neq x_{\si(a)}}\frac{1}{w_{\si(a)}-w_{\si(j)}}\right]\\
&&\times \left[
\prod_{j>b,x_{\si(j)}\neq x_{\si(b)}}\frac{1}{w_{\si(b)}-w_{\si(j)}}\right]=D_1^{(\eta\si)};
\end{eqnarray*}
and
\begin{eqnarray*}
D_2^{(\si)}&:=&\prod_{a<i<b}\left[\frac{1}{(w_{\si(a)}-w_{\si(i)})(w_{\si(i)}-w_{\si(b)})}\right]=D_2^{(\eta\sigma)}
\end{eqnarray*}
Then the lemma follows.
\end{proof}

\begin{proposition}\label{p437}Let $\{w_i\}_{i\in[N]}$ and $k$ be given by (\ref{dw}) and (\ref{dk}), respectively. Then we have the following formula
\begin{eqnarray}
&&\label{sws}s_{\lambda}(w_1,\ldots,w_N)\\
&=&\sum_{\ol{\sigma}\in[\Sigma_N/\Sigma_N^X]^r} \left(\prod_{i=1}^{n}x_i^{|\phi^{(i,\sigma)}(N)|}\right)\left(\prod_{i=1}^{r}s_{\phi^{(i,\sigma)}(N)}\left(\frac{u_i}{x_i},\frac{u_{n+i}}{x_i}\ldots,\frac{u_{qn+i}}{x_i},1,\ldots,1\right)\right)\notag\\
&&\times\left(\prod_{i=r+1}^{n}s_{\phi^{(i,\sigma)}(N)}\left(\frac{u_i}{x_i},\frac{u_{n+i}}{x_i}\ldots,\frac{u_{(q-1)n+i}}{x_i},1,\ldots,1\right)\right)\notag\\
&&\times\left(\prod_{i<j,x_{\sigma(i)}\neq x_{\sigma(j)}}\frac{1}{w_{\sigma(i)}-w_{\sigma(j)}}\right)\notag
\end{eqnarray}
where $\sigma\in \ol{\sigma}\cap \Sigma_N$ is a representative. 
\end{proposition}

\begin{proof}First of all, by Lemmas \ref{pee} and \ref{l437},  the right hand side of (\ref{sws}) is independent of the choice of the representative in $\ol{\si}\in[\Sigma_N/\Sigma_N^X]^r$.

Recall the well-known formula to compute the Schur function
\begin{eqnarray*}
s_{\lambda}(w_1,\ldots,w_N)&=&\sum_{\sigma\in\Si_N}\left(w_{\si(1)}^{\lambda_1}\cdots w_{\si(N)}^{\lambda_N}\prod_{i<j}\frac{w_{\si(i)}}{w_{\si(i)}-w_{\si(j)}}\right).\\
&=&\sum_{\ol{\sigma}\in [\Sigma_N/\Sigma_N^X]^r}\sum_{\sigma\in\ol{\sigma}}\left(w_{\si(1)}^{\lambda_1}\cdots w_{\si(N)}^{\lambda_N}\prod_{i<j}\frac{w_{\si(i)}}{w_{\si(i)}-w_{\si(j)}}\right).
\end{eqnarray*}
If some $w_i=w_j$ for some $i\neq j$, the right hand side is computed by the limit $\lim_{w_i\rightarrow w_j}$.
To prove (\ref{sws}), it suffices to show that for each $\ol{\si}\in[\Si_N/\Si_N^X]^r$, we have
\begin{eqnarray*}
&&\sum_{\sigma\in\ol{\sigma}}\left(w_{\si(1)}^{\lambda_1}\cdots w_{\si(N)}^{\lambda_N}\prod_{i<j}\frac{w_{\si(i)}}{w_{\si(i)}-w_{\si(j)}}\right)\\
&=&\left(\prod_{i=1}^{n}x_i^{|\phi^{(i,\sigma)}(N)|}\right)\left(\prod_{i=1}^{r}s_{\phi^{(i,\sigma)}(N)}\left(\frac{u_i}{x_i},\frac{u_{n+i}}{x_i}\ldots,\frac{u_{qn+i}}{x_i},1,\ldots,1\right)\right)\notag\\
&&\times\left(\prod_{i=r+1}^{n}s_{\phi^{(i,\sigma)}(N)}\left(\frac{u_i}{x_i},\frac{u_{n+i}}{x_i}\ldots,\frac{u_{(q-1)n+i}}{x_i},1,\ldots,1\right)\right)\notag\\
&&\times\left(\prod_{i<j,x_{\sigma(i)}\neq x_{\sigma(j)}}\frac{1}{w_{\sigma(i)}-w_{\sigma(j)}}\right)\notag
\end{eqnarray*}
Note that
\begin{eqnarray*}
\ol{\sigma}=\Sigma_N^X\sigma=\mathcal{S}_{T_n}\mathcal{S}_{T_{n-1}}\cdot\ldots\cdot\mathcal{S}_{T_1}\sigma.
\end{eqnarray*}
Hence we have
\begin{eqnarray*}
&&\sum_{\sigma\in\ol{\sigma}}\left(w_{\si(1)}^{\lambda_1}\cdots w_{\si(N)}^{\lambda_N}\prod_{i<j}\frac{w_{\si(i)}}{w_{\si(i)}-w_{\si(j)}}\right)\\
&=&\sum_{\xi_n\in\mathcal{S}_{T_n}}\cdot\ldots\cdot\sum_{\xi_1\in\mathcal{S}_{T_1}}\left(w_{\xi_n\ldots\xi_1\si(1)}^{\lambda_1}\cdots w_{\xi_n\ldots\xi_1\si(N)}^{\lambda_N}\prod_{i<j}\frac{w_{\xi_n\ldots\xi_1\si(i)}}{w_{\xi_n\ldots\xi_1\si(i)}-w_{\xi_n\ldots\xi_1\si(j)}}\right)\\
&=&\sum_{\xi_n\in\mathcal{S}_{T_n}}\cdot\ldots\cdot\sum_{\xi_1\in\mathcal{S}_{T_1}}\left(w_{\xi_n\ldots\xi_1\si(1)}^{\lambda_1}\cdots w_{\xi_n\ldots\xi_1\si(N)}^{\lambda_N}\prod_{i<j,x_{\sigma(i)}\neq x_{\sigma(j)}}\frac{w_{\xi_n\ldots\xi_1\si(i)}}{w_{\xi_n\ldots\xi_1\si(i)}-w_{\xi_n\ldots\xi_1\si(j)}}\right)\\
&&\times\left(\prod_{h=1}^{n}\prod_{i<j,x_{\sigma(i)}= x_{\sigma(j)}=x_h}\frac{w_{\xi_n\ldots\xi_1\si(i)}}{w_{\xi_n\ldots\xi_1\si(i)}-w_{\xi_n\ldots\xi_1\si(j)}}\right)
\end{eqnarray*}
Note that if $x_{\sigma(g)}=x_h$, for some $h\in[n]$ and $\xi_i\in \mathcal{S}_{T_i}$, for all $i\in[n]$,
\begin{eqnarray*}
\xi_n\ldots\xi_1\sigma(g)=\xi_h(\sigma(g))
\end{eqnarray*}
Then by Lemma \ref{l437}, we obtain
\begin{eqnarray*}
&&\sum_{\sigma\in\ol{\sigma}}\left(w_{\si(1)}^{\lambda_1}\cdots w_{\si(N)}^{\lambda_N}\prod_{i<j}\frac{w_{\si(i)}}{w_{\si(i)}-w_{\si(j)}}\right)
=\prod_{i<j,x_{\sigma(i)}\neq x_{\sigma(j)}}\frac{1}{w_{\si(i)}-w_{\si(j)}}\\
&&\left[\prod_{h=1}^{n}\left(\sum_{\xi_h\in\mathcal{S}_{T_h}}\prod_{g:x_{\si(g)}=x_h}w_{\xi_h\si(g)}^{\lambda_g+\eta_g^{\si}}\prod_{i<j,x_{\sigma(i)}= x_{\sigma(j)}=x_h}\frac{w_{\xi_h\si(i)}}{w_{\xi_h\si(i)}-w_{\xi_h\si(j)}}\right)\right]
\end{eqnarray*}

Note that 
\begin{eqnarray*}
&&\sum_{\xi_h\in\mathcal{S}_{T_h}}\prod_{g:x_{\si(g)}=x_h}w_{\xi_h\si(g)}^{\lambda_g+\eta_g^{\si}}\prod_{i<j,x_{\sigma(i)}= x_{\sigma(j)}=x_h}\frac{w_{\xi_h\si(i)}}{w_{\xi_h\si(i)}-w_{\xi_h\si(j)}}\\
&=&x_h^{|\phi^{(h,\si)}(N)|}s_{\phi^{(h,\sigma)}(N)}\left(\frac{u_h}{x_h},\frac{u_{n+h}}{x_h}\ldots,\frac{u_{tn+h}}{x_h},1,\ldots,1\right).
\end{eqnarray*}
where 
\begin{eqnarray*}
t=\left\{\begin{array}{cc}q&\mathrm{if}\ 1\leq i\leq r\\q-1&\mathrm{if}\ r+1\leq i\leq n\end{array}\right. 
\end{eqnarray*}
Then the proposition follows..
\end{proof}

\noindent{\textbf{Proof of Theorem \ref{p421}.} Theorem \ref{p421} follows from Proposition \ref{p437} by letting $u_j=x_j$ for all $j\in[k]$. $\hfill\Box$

\section{Asymptotics of the Schur polynomial at a general point}\label{s4}

In this section, we use Proposition \ref{p437} to study the asymptotics of Schur functions at a general point. The main goal is to prove Theorem \ref{tm2}.
Note that Proposition \ref{p437} expresses the Schur polynomial at a general point as a sum of Schur polynomials at $(1,\ldots,1)$; when $(x_1,\ldots,x_N)$ are periodic with a fixed finite period $n$ as in Assumption \ref{ap423}, the number of summands to compute the Schur function at a general point in the formula as given in Proposition \ref{p437}, is exponential in $N$. The idea is to find a leading term among all these summands, and then use the leading term to study the asymptotics of the Schur polynomial at a general point.

Recall that evaluating the Schur function $s_{\phi^{(i,\sigma)}(N)}$ at $(1,\ldots,1)$ can be done by using the Weyl character formula
\begin{eqnarray*}
s_{\phi^{(i,\sigma)}(N)}(1,\ldots,1)=\prod_{1\leq j<k\leq |T_i|}\frac{\phi_j^{(i,\sigma)}(N)-\phi_k^{(i,\sigma)}(N)+k-j}{k-j}\
\end{eqnarray*}
where $T_i$ is defined by (\ref{dti}) for $i\in[n]$.

With the help of the Weyl character formula, (\ref{ses}) can also be written as
\begin{eqnarray*}
s_{\lambda}(x_1,\ldots,x_N)&=&\sum_{\ol{\sigma}\in[\Sigma_N/\Sigma_N^X]^r} \left(\prod_{i=1}^{N}x_{\sigma(i)}^{\lambda_i(N)}\right)\left(\prod_{i=1}^{n}\prod_{1\leq j<k\leq |T_i|}\frac{\phi^{(i,\si)}_j-\phi^{(i,\si)}_k+k-j}{k-j}\right)\\
&&\times\left(\prod_{i<j,x_{\sigma(i)}\neq x_{\sigma(j)}}\frac{x_{\sigma(i)}}{x_{\sigma(i)}-x_{\sigma(j)}}\right)
\end{eqnarray*}

Let $\si\in\Si_N$. Under Assumption \ref{ap428} we have
\begin{eqnarray}
&&\left|\frac{x_{\si_0(1)}^{\lambda_{1}}\cdots x_{\si_0(N)}^{\lambda_N}}{x_{\si(1)}^{\lambda_1}\cdots x_{\si(N)}^{\lambda_N}}\right|\notag\\
&=&\left|\frac{\left[\prod_{i=1}^{K_s}x_{\sigma_0(i)}\right]^{\mu_1}\left[\prod_{i=K_s+1}^{K_s+K_{s-1}}x_{\sigma_0(i)}\right]^{\mu_2}\cdot\ldots\cdot\left[\prod_{i=1+\sum_{t=2}^sK_t}^{\sum_{t=1}^{s}K_t}x_{\sigma_0(i)}\right]^{\mu_s}}{\left[\prod_{i=1}^{K_s}x_{\si(i)}\right]^{\mu_1}\left[\prod_{i=K_s+1}^{K_s+K_{s-1}}x_{\si(i)}\right]^{\mu_2}\cdot\ldots\cdot\left[\prod_{i=1+\sum_{t=2}^sK_t}^{\sum_{t=1}^{s}K_t}x_{\si(i)}\right]^{\mu_s}}\right|\label{pq}
\end{eqnarray}

For $i,j\in[n]$, $t\in[s]$, $\si\in \Sigma_N$, define
\begin{eqnarray}
I_{i,j,t}^{\sigma}\label{ijs}
=\{p:p\in[N], x_{\si_0(p)}=x_i, x_{\si(p)}=x_j, \lambda_p=\mu_t\}.
\end{eqnarray}
We may interpret $I_{i,j,t}$ as follows. Consider a $3\times N$ array
\begin{eqnarray*}
B:=\left(\begin{array}{cccc}\lambda_1&\lambda_2&\ldots&\lambda_N\\x_{\si_0(1)}&x_{\si_0(2)}&\ldots&x_{\si_0(N)}\\x_{\si(1)}&x_{\si(2)}&\ldots&x_{\si(N)}\end{array}\right).
\end{eqnarray*}
Then $I_{i,j,t}^{\si}$ consists of all the column indices $p$, such that
\begin{eqnarray}
&&B(1,p)=\mu_t,\ \mathrm{and}\notag\\ &&B(2,p)=x_i,\ \mathrm{and}\label{b2p}\\ &&B(3,p)=x_j\label{b3p}
\end{eqnarray}

We use $\left|I_{i,j,t}^{\sigma}\right|$ to denote the cardinality of the set $I_{i,j,t}^{\sigma}$.

Recall that for $i\in[n]$, $J_i$ is defined in (\ref{ji}).
Let 
\begin{eqnarray}
I_{i,j}^{\sigma}=\cup_{t\in[s]}I_{i,j,t}^{\sigma}=\cup_{t\in J_i}I_{i,j,t}^{\si}\label{ism}.
\end{eqnarray}
That is, $I_{i,j}^{\sigma}$ consists of all the column indices $p\in[N]$ such that 
(\ref{b2p}) and (\ref{b3p}) hold. The last identity follows from the fact that $I_{i,j,t}^{\sigma}=\emptyset$ unless $t\in J_i$.
Since the right hand side of (\ref{ism}) is a disjoint union, we have
\begin{eqnarray}
\left|I_{i,j}^{\sigma}\right|=\sum_{t\in J_i}\left|I_{i,j,t}^{\si}\right|=\sum_{t\in [s]}\left|I_{i,j,t}^{\si}\right|\label{iat}
\end{eqnarray}

Recall also that $T_i$ is defined by (\ref{dti}), we have
\begin{eqnarray}
\sigma_0^{-1}(T_i)=\cup_{j\in[n]}I_{i,j}^{\sigma}=\cup_{j\in[n]}\cup_{t\in J_i}I_{i,j,t}^{\sigma}\label{tri}\\
\si^{-1}(T_j)=\cup_{i\in[n]}I_{i,j}^{\si}=\cup_{i\in[n]}\cup_{t\in J_i}I_{i,j,t}^{\si}\label{trig}
\end{eqnarray}
where for $j\in[n]$ and $\si\in\Si_N$,
\begin{eqnarray*}
\sigma^{-1}(T_j)=\{l\in[N]:x_{\si(l)}=x_j\}.
\end{eqnarray*}

\begin{lemma}\label{l425}Suppose Assumption \ref{ap423} holds.
 Let $\mathcal{P}$ be a nonempty, proper subset of $[n]$, i.e. 
\begin{eqnarray*}
\mathcal{P}\subset[n];\qquad\mathcal{P}\neq \emptyset;\qquad \mathcal{P}\neq[n].
\end{eqnarray*}
Let $\mathcal{P}^c=[n]\setminus \mathcal{P}$ be the complement of $\mathcal{P}$ in $[n]$. Then
\begin{eqnarray}
\sum_{i\in \mathcal{P}}\sum_{j\in \mathcal{P}^c}\left|I_{i,j}^{\si}\right|=\sum_{i\in \mathcal{P}}\sum_{j\in\mathcal{P}^c}\left|I_{j,i}^{\si}\right|.\label{ii}
\end{eqnarray}

\end{lemma}
\begin{proof}

Let
\begin{eqnarray*}
X_{\mathcal{P}}=\{x_i:i\in \mathcal{P}\};\qquad X_{\mathcal{P}^c}=\{x_i:i\in \mathcal{P}^c\}
\end{eqnarray*}
By (\ref{ijs}), the left hand side of (\ref{ii}) is equal to 
\begin{eqnarray}
|\{p:p\in[N], x_{\si_0(p)}\in X_{\mathcal{P}}\}|-|\{p:p\in[N], x_{\si_0(p)}\in X_{\mathcal{P}}, x_{\si(p)}\in X_{\mathcal{P}}\}|\label{lt}
\end{eqnarray}
while the right hand side of (\ref{ii}) is equal to \begin{eqnarray}
 \label{rt}
|\{p:p\in[N], x_{\si(p)}\in X_{\mathcal{P}}\}|-|\{p:p\in[N], x_{\si_0(p)}\in X_{\mathcal{P}}, x_{\si(p)}\in X_{\mathcal{P}}\}|\notag
\end{eqnarray}
Since $\si$ is a bijection from $[N]$ to $[N]$, and under Assumption \ref{ap423}, there are exactly $\frac{N}{n}$ $j$'s in $[N]$ such that $x_j=x_i$ for each fixed $i\in[n]$,  we obtain
\begin{eqnarray*}
|\{p:p\in[N], x_{\si(p)}\in X_{\mathcal{P}}\}|=\frac{N|\mathcal{P}|}{n};
\end{eqnarray*}
for any $\si\in \Si_N$.

Therefore both (\ref{lt}) and (\ref{rt}) are equal to 
\begin{eqnarray*}
\frac{N|\mathcal{P}|}{n}-|\{p:p\in[N], x_{\si_0(p)}\in X_{\mathcal{P}}, x_{\si(p)}\in X_{\mathcal{P}}\}|; 
\end{eqnarray*}
then the lemma follows.
\end{proof}

\begin{lemma}Suppose (\ref{mi}) and Assumption \ref{ap423} hold, and $\sigma\in \Sigma_N\setminus \ol{\si}_0$. For $p\in[n]$, let
\begin{eqnarray*}
m_p=\max\{\ell:\ell\in J_p\}
\end{eqnarray*}
Let
\begin{eqnarray}
L:&=&\sum_{1\leq i<j\leq n}\sum_{t=1}^{s}\mu_{t}\left[\left|I_{i,j,t}^{\si}\right|-\left|I_{j,i,t}^{\si}\right|\right](j-i)\label{dli}
\end{eqnarray}
Then we have
\begin{eqnarray*}
L= \sum_{p=1}^{n}\left[\sum_{i=1}^{p}\sum_{j=p+1}^{n}\sum_{t\in J_i}(\mu_{t}-\mu_{m_p})\left|I^{\si}_{i,j,t}\right|+\sum_{i=p+1}^{n}\sum_{j=1}^{p}\sum_{t\in J_i}(\mu_{m_p}-\mu_{t})\left|I_{i,j,t}^{\si}\right|\right] \notag
\end{eqnarray*}
In particular $L\geq 0$. Moreover, if Assumption \ref{ap428} holds, then
\begin{eqnarray}
L&\geq&\frac{1}{2}\min_{p<q}(\mu_p-\mu_q)\sum_{i=1}^{n}\sum_{j\in[n]\setminus\{i\}}|I_{i,j}^{\si}|
\geq  \frac{C_1}{2} N \sum_{i=1}^{n}\sum_{j\in[n]\setminus\{i\}}|I_{i,j}^{\si}|.\label{lq}
\end{eqnarray}
\end{lemma}

\begin{proof}We rewrite the right hand side of (\ref{dli}) as follows:
\begin{eqnarray}
L=\sum_{t=1}^{s}\mu_{t}\left[\sum_{1\leq i<j\leq n}\left|I_{i,j,t}^{\si}\right|(j-i)-\sum_{1\leq j<i\leq n}\left|I_{i,j,t}^{\si}\right|(i-j)\right],\label{ity0}
\end{eqnarray}
which is also equal to
\begin{eqnarray}
L&=&\sum_{p=1}^{n}\left[\sum_{i=1}^{p}\sum_{j=p+1}^{n}\sum_{t\in J_i}\mu_{t}\left|I^{\si}_{i,j,t}\right|-\sum_{i=p+1}^{n}\sum_{j=1}^{p}\sum_{t\in J_i}\mu_{t}\left|I_{i,j,t}^{\si}\right|\right]\label{ity1}
\end{eqnarray}
To see why (\ref{ity1}) follows from (\ref{ity0}), note that for each given $(a,b)\in[n]^2$, $a<b$, in $\sum_{p=1}^{n}\sum_{i=1}^{p}\sum_{j=p+1}^{n}\sum_{t\in J_i}\mu_{t}\left|I^{\si}_{i,j,t}\right|$, the monomial $\mu_{t}\left|I^{\si}_{a,b,t}\right|$ is added when $p=a,a+1,\ldots,b-1$, which is exactly $(b-a)$ times.

By Lemma \ref{l424}, we have
\begin{eqnarray*}
L\geq \sum_{p=1}^{n}\mu_{m_p}\left[\sum_{i=1}^{p}\sum_{j=p+1}^{n}\sum_{t\in J_i}\left|I^{\si}_{i,j,t}\right|-\sum_{i=p+1}^{n}\sum_{j=1}^{p}\sum_{t\in J_i}\left|I_{i,j,t}^{\si}\right|\right]
\end{eqnarray*}

Let
\begin{eqnarray*}
\mathcal{P}=[p]\subseteq[n].
\end{eqnarray*}
By Lemma \ref{l425} and (\ref{iat}), we obtain
\begin{eqnarray*}
\sum_{i=1}^{p}\sum_{j=p+1}^{n}\sum_{t\in J_i}\left|I^{\si}_{i,j,t}\right|=\sum_{i=p+1}^{n}\sum_{j=1}^{p}\sum_{t\in J_i}\left|I_{i,j,t}^{\si}\right|
\end{eqnarray*}
Therefore $L\geq 0$.

Moreover, we have 
\begin{eqnarray}
L&=& L-0\label{lg}\\
&= & \sum_{p=1}^{n}\left[\sum_{i=1}^{p}\sum_{j=p+1}^{n}\sum_{t=1}^{s}\mu_{t}\left|I^{\si}_{i,j,t}\right|-\sum_{i=p+1}^{n}\sum_{j=1}^{p}\sum_{t=1}^{s}\mu_{t}\left|I_{i,j,t}^{\si}\right|\right]\notag  \\
&&-\sum_{p=1}^{n}\mu_{m_p}\left[\sum_{i=1}^{p}\sum_{j=p+1}^{n}\sum_{t=1}^{s}\left|I^{\si}_{i,j,t}\right|-\sum_{i=p+1}^{n}\sum_{j=1}^{p}\sum_{t=1}^{s}\left|I_{i,j,t}^{\si}\right|\right]\notag\\
&=& \sum_{p=1}^{n}\left[\sum_{i=1}^{p}\sum_{j=p+1}^{n}\sum_{t\in J_i}(\mu_t-\mu_{m_p})\left|I^{\si}_{i,j,t}\right|+\sum_{i=p+1}^{n}\sum_{j=1}^{p}\sum_{t\in J_i}(\mu_{m_p}-\mu_t)\left|I_{i,j,t}^{\si}\right|\right]\notag
\end{eqnarray}

Under the assumption that $\si\notin\ol{\si}_0$, there exist $i,j\in [n]$, $i\neq j$ and $t\in[s]$, such that $\left|I_{i,j,t}^{\si}\right|>0$. The following cases might occur
\begin{itemize}
\item There exist $i<j$ such that $\sum_{t=1}^{s}\left|I_{i,j,t}^{\si}\right|>0$. By (\ref{ii}) we have for any positive integer $p$ satisfying $i\leq p<j$,
\begin{eqnarray*}
\sum_{a=1}^{p}\sum_{b=p+1}^{n}\sum_{t=1}^{s}\left|I_{a,b,t}^{\sigma}\right|=\sum_{a'=p+1}^{n}\sum_{b'=1}^{p}\sum_{t=1}^{s}\left| I_{a',b',t}^{\si}\right|
\end{eqnarray*}
Then there exist $i'>j'$, such that $\sum_{t=1}^{s}\left|I_{i',j',t}^{\si}\right|>0$.
\item There exist $i>j$ such that $\sum_{t=1}^{s}\left|I_{i,j,t}^{\si}\right|>0$, then by Assumption \ref{ap428} for any $i<j$, $\ell\in J_i$, $t\in J_j$, we have $\ell<t$, then by (\ref{lg}), we obtain
\begin{eqnarray*}
L&\geq& \min_{p<q}(\mu_{p}-\mu_{q})\sum_{r=1}^{n}\sum_{i=r+1}^{n}\sum_{j=1}^{r}\sum_{t\in J_i}\left|I_{i,j,t}^{\si}\right|
=\min_{p<q}(\mu_{p}-\mu_{q})\sum_{r=1}^{n}\sum_{i=r+1}^{n}\sum_{j=1}^{r}\left|I_{i,j}^{\si}\right|
\end{eqnarray*}
By Lemma \ref{l425}, we have
\begin{eqnarray*}
L&\geq&\min_{p<q}(\mu_{p}-\mu_{q})\sum_{r=1}^{n}\sum_{i=1}^{r}\sum_{j=r+1}^{n}\left|I_{j,i}^{\si}\right|\\
&=&\frac{\min_{p<q}(\mu_{p}-\mu_{q})}{2}\left[\sum_{1\leq j<i\leq n}\left|I_{i,j}^{\si}\right|(i-j)+\sum_{1\leq i<j\leq n}\left|I_{i,j}^{\si}\right|(j-i)\right]\\
&\geq &\frac{\min_{p<q}(\mu_{p}-\mu_{q})}{2}\sum_{i=1}^{n}\sum_{j\in[n]\setminus\{i\}}|I_{i,j}^{\si}|\\
&\geq &\frac{C_1 N}{2}\sum_{i=1}^{n}\sum_{j\in[n]\setminus\{i\}}|I_{i,j}^{\si}|
\end{eqnarray*}
where the last inequality follows from Assumption \ref{ap428}. Then the Lemma follows.
\end{itemize}
\end{proof}

Then
\begin{eqnarray}
&&\label{xg}\left|\frac{x_{\si_0(1)}^{\lambda_{1}}\cdots x_{\si_0(N)}^{\lambda_N}}{x_{\si(1)}^{\lambda_1}\cdots x_{\si(N)}^{\lambda_N}}\right|=\prod_{1\leq i<j\leq n}\prod_{t=1}^{s}\left(\frac{x_i}{x_j}\right)^{\mu_{t}\left[\left|I_{i,j,t}^{\si}\right|-\left|I_{j,i,t}^{\si}\right|\right]}\\
&=&\prod_{1\leq i<j\leq n}\left(\frac{x_{i}}{x_{i+1}}\cdot\frac{x_{i+1}}{x_{i+2}}\cdot\ldots\cdot\frac{x_{j-1}}{x_j}\right)^{\sum_{t=1}^{s}\mu_{t}\left[\left|I_{i,j,t}^{\si}\right|-\left|I_{j,i,t}^{\si}\right|\right]}\notag\\
&\geq& \left(\min_{1\leq i<j\leq n}\frac{x_i}{x_j}\right)^{\sum_{1\leq i<j\leq n}\sum_{t=1}^{s}\mu_{t}\left[\left|I_{(i,j,t)}^{\si}\right|-\left|I_{(j,i,t)}^{\si}\right|\right](j-i)}\notag
\end{eqnarray}

If Assumption \ref{ap428} holds, then by (\ref{lq}), we have
\begin{eqnarray*}
&&\left|\frac{x_{\si_0(1)}^{\lambda_{1}}\cdots x_{\si_0(N)}^{\lambda_N}}{x_{\si(1)}^{\lambda_1}\cdots x_{\si(N)}^{\lambda_N}}\right|\geq \left(\min_{1\leq i<j\leq n}\frac{x_i}{x_j}\right)^{\min_{p<q}(\mu_p-\mu_q)}
\end{eqnarray*}

\begin{lemma}\label{ll43}Let $\si_0$ be defined as in (\ref{sz}) and $\si\in \Si_N$. Then
\begin{eqnarray}
&&\label{xq}\left|\frac{\prod_{i<j,x_{\sigma_0(i)}\neq x_{\sigma_0(j)}}\frac{x_{\sigma_0(i)}}{x_{\sigma_0(i)}-x_{\sigma_0(j)}}}{\prod_{i<j,x_{\sigma(i)}\neq x_{\sigma(j)}}\frac{x_{\sigma(i)}}{x_{\sigma(i)}-x_{\sigma(j)}}}\right|=
\left|\prod_{x_i\neq x_j,\si_0^{-1}(i)>\si_0^{-1}(j),\si^{-1}(i)<\si^{-1}(j)}\frac{x_j}{x_i}\right|\geq 1\notag
\end{eqnarray}
\end{lemma}
\begin{proof}Note that
\begin{eqnarray*}
&&\left|\frac{\prod_{i<j,x_{\sigma_0(i)}\neq x_{\sigma_0(j)}}\frac{x_{\sigma_0(i)}}{x_{\sigma_0(i)}-x_{\sigma_0(j)}}}{\prod_{i<j,x_{\sigma(i)}\neq x_{\sigma(j)}}\frac{x_{\sigma(i)}}{x_{\sigma(i)}-x_{\sigma(j)}}}\right|
=\left|\frac{\prod_{x_i\neq x_j,\si_0^{-1}(i)<\si_0^{-1}(j)}\left(1-\frac{x_j}{x_i}\right)^{-1}}{\prod_{x_i\neq x_j,\si^{-1}(i)<\si^{-1}(j)}\left(1-\frac{x_j}{x_i}\right)^{-1}}\right|\\
&=&\left|\frac{\prod_{x_i\neq x_j,\si_0^{-1}(i)>\si_0^{-1}(j),\si^{-1}(i)<\si^{-1}(j)}\left(1-\frac{x_j}{x_i}\right)}{\prod_{x_i\neq x_j,\si_0^{-1}(i)<\si_0^{-1}(j),\si^{-1}(i)>\si^{-1}(j)}\left(1-\frac{x_j}{x_i}\right)}\right|=\left|\prod_{x_i\neq x_j,\si_0^{-1}(i)>\si_0^{-1}(j),\si^{-1}(i)<\si^{-1}(j)}\frac{x_j}{x_i}\right|\\
&\geq& 1
\end{eqnarray*}
where the last inequality holds because $\si_0^{-1}(i)>\si_0^{-1}(j)$, by (\ref{sz}), we obtain $x_j\geq x_i$
\end{proof}

It is straightforward to see that 
\begin{eqnarray}
&&\left|\frac{\prod_{i=1}^{n}\prod_{1\leq j<k\leq |T_i|}\frac{\phi^{(i,\si_0)}_j-\phi^{(i,\si_0)}_k+k-j}{k-j}}{\prod_{i=1}^{n}\prod_{1\leq j<k\leq |T_i|}\frac{\phi^{(i,\si)}_j-\phi^{(i,\si)}_k+k-j}{k-j}}\right|\geq \prod_{i=1}^{n}\prod_{1\leq j<k\leq |T_i|}\frac{k-j}{\phi^{(i,\si)}_j-\phi^{(i,\si)}_k+k-j}\label{pbq}
\end{eqnarray}

We have the following lemma
 \begin{lemma}\label{ll46}Under Assumptions \ref{ap423} and \ref{ap428}, we have
 \begin{eqnarray}
 &&\prod_{i=1}^{n}\prod_{1\leq j<k\leq |T_i|}\frac{k-j}{\phi^{(i,\si)}_j-\phi^{(i,\si)}_k+k-j}\geq e^{-C_3N^2}\label{46iq}
 \end{eqnarray}
 where $C_3>0$ is a constant independent of $N$.
 \end{lemma}
 
 \begin{proof}Lemma \ref{ll46} can also be obtained as follows. Note that the left hand side of (\ref{46iq}) is exactly the reciprocal of $\prod_{i=1}^{n}s_{\phi^{(i,\si)}}(1,\ldots,1)$. The Schur polynomial $s_{\phi^{(i,\si)}}(1,\ldots,1)$ counts the number of perfect matchings on a contracting hexagon lattice with boundary partition given by $\phi^{(i,\si)}$. Under Assumption $\ref{ap428}$, all the boundaries of the contracting hexagon lattice grow linearly in $N$, and therefore the number of vertices in the contraction hexagon lattice is $O(N^2)$. Hence the total number of perfect matchings is bounded above by $e^{O(N^2)}$.
\end{proof}

\begin{proposition}\label{p436}Suppose Assumptions \ref{ap32} and \ref{ap428} hold, and let $\alpha$ be given as in Assumption \ref{ap32}. For each given $\{a_i,b_i\}_{i=1}^{n}$, when $\alpha$ is sufficiently large, for any $\sigma\notin \ol{\si}_0$ we have
\begin{eqnarray}
\label{gep}&&\left|\frac{\left(\prod_{i=1}^{n}x_{i,N}^{|\phi^{(i,\sigma_0)}(N)|}\right)\left(\prod_{i=1}^{n}s_{\phi^{(i,\sigma_0)}(N)}(1,\ldots,1)\right)}{\left(\prod_{i=1}^{n}x_{i,N}^{|\phi^{(i,\sigma)}(N)|}\right)\left(\prod_{i=1}^{n}s_{\phi^{(i,\sigma)}(N)}(1,\ldots,1)\right)
}\right|\\
&&\times\left|\frac{\left(\prod_{i<j,x_{\sigma_0(i),N}\neq x_{\sigma_0(j),N}}\frac{1}{x_{\sigma_0(i),N}-x_{\sigma_0(j),N}}\right)}{\left(\prod_{i<j,x_{\sigma(i),N}\neq x_{\sigma(j),N}}\frac{1}{x_{\sigma(i),N}-x_{\sigma(j),N}}\right)}\right|\notag\geq e^{CN^2}\notag
\end{eqnarray}
where $C>0$ is a constant independent of $N$ and $\si$, and increases as $\alpha$ increases. Indeed, we have
\begin{eqnarray*}
\lim_{\alpha\rightarrow\infty} C=\infty.
\end{eqnarray*}
\end{proposition}

\begin{proof}Let $\mathcal{R}$ denote the left hand side of (\ref{gep}). By Lemma \ref{ll43}, we have
\begin{eqnarray}
\mathcal{R}
\geq\frac{\left(\prod_{i=1}^{N}x_{\si_0(i),N}^{\lambda_i}\right)\left(\prod_{i=1}^{n}s_{\phi^{(i,\sigma_0)}(N)}(1,\ldots,1)\right)
}{\left(\prod_{i=1}^{N}x_{\si(i),N}^{\lambda_i}\right)\left(\prod_{i=1}^{n}s_{\phi^{(i,\sigma)}(N)}(1,\ldots,1)\right)
}\label{r1}
\end{eqnarray}
By Lemma \ref{ll46}, we have
\begin{eqnarray}
&&\frac{\prod_{i=1}^{n}s_{\phi^{(i,\sigma_0)}(N)}(1,\ldots,1)\notag
}{\prod_{i=1}^{n}s_{\phi^{(i,\sigma)}(N)}(1,\ldots,1)}\geq e^{-C_3N^2} \label{r2}
\end{eqnarray}
where $C_{3}>0$ is a constant independent of $N$.
Under Assumption \ref{ap32} and Assumption \ref{ap428}, we have
\begin{eqnarray}
\left|\frac{x_{\si_0(1),N}^{\lambda_{1}}\cdots x_{\si_0(N),N}^{\lambda_N}}{x_{\si(1),N}^{\lambda_1}\cdots x_{\si(N),N}^{\lambda_N}}\right|\geq
 \left(\min_{1\leq i<j\leq n}\frac{x_{i,N}}{x_{j,N}}\right)^{\frac{C_1N}{2}\left(\sum_{1\leq i\leq n}\sum_{1\leq j\leq n,j\neq i}|I_{ij}^{\si}|\right)}\label{r3}
\end{eqnarray}
If $\alpha$ in Assumption \ref{ap32} satisfies
\begin{eqnarray*}
\alpha>\frac{2 C_{3}}{C_1},
\end{eqnarray*}
we obtain
\begin{eqnarray*}
\mathcal{R}\geq e^{-\left(\frac{C_1\alpha}{2}-C_3\right)N^2}
\end{eqnarray*}
Choose $C=\frac{C_1\alpha}{2}-C_{3}$, then the lemma follows.
\end{proof}

\begin{lemma}\label{cm}For $i\in[n]$, let $J_i$ be defined as in (\ref{ji}).
Under Assumptions \ref{ap423} and \ref{ap428}, assume that 
\begin{eqnarray*}
J_i=\left\{\begin{array}{cc}\{d_i,d_i+1,\ldots,d_{i+1}-1\}& \mathrm{if}\ 1\leq i\leq n-1\\ \{d_n, d_n+1,\ldots,s\}& \mathrm{if}\ i=n \end{array}\right.
\end{eqnarray*}
where $d_1,\ldots,d_n$ are positive integers satisfying
\begin{eqnarray*}
1=d_1<d_2<\ldots <d_n\leq s
\end{eqnarray*}
Let
\begin{eqnarray*}
d_{n+1}:=s+1.
\end{eqnarray*}
For $j\in[s]$, let $a_j,b_j$ be given by (\ref{abi}). 
If $i\in[n]$, for $0\leq k\leq d_{i+1}-d_i-1$, let
\begin{eqnarray}
\beta_{i,k}&=&n\left(a_1+\sum_{l=2}^{s-d_i-k+1}(a_l-b_{l-1})\right)+n-i+1-n\left(\sum_{l=s-d_i-k+1}^{s-d_i+1}(b_l-a_l)\right)\label{btik}\\
\gamma_{i,k}&=&n\left(a_1+\sum_{l=2}^{s-d_i-k+1}(a_l-b_{l-1})\right)+n-i+1-n\left(\sum_{l=s-d_i-k+2}^{s-d_i+1}(b_l-a_l)\right).\label{cmik}
\end{eqnarray}

Then the counting measures of  $\phi^{(i,\si_0)}(N)$ converge weakly to a limit measure $\bm_i$ as $N\rightarrow\infty$. Moreover,
if $i\in[n]$, for $0\leq k\leq d_{i+1}-d_i-1$, the limit counting measure $\mathbf{m}_i$ is a probability measure on $[\beta_{i,1},\gamma_{i,d_{i+1}-d_i-1}]$ with density given by
\begin{eqnarray*}
\frac{d\mathbf{m}_i}{dx}=\left\{\begin{array}{cc}1,& \mathrm{if}\ \beta_{i,k}< x< \gamma_{i,k};\\0,& \mathrm{if}\ \gamma_{i,k}\leq x\leq \beta_{i,k+1}. \end{array}\right.
\end{eqnarray*}

\end{lemma}
\begin{proof}First of all, it is straightforward to check that under Assumptions \ref{ap423} and \ref{ap428}, for $i\in[n]$, the limiting measure $\bm_i$ has constant densities of 0's and 1's on finitely many alternating intervals. It suffices to determine the endpoints of these intervals on which $\bm_i$ has constant densities.

Note that under Assumption \ref{ap423}, we have
\begin{eqnarray*}
x_{\sigma_0\left(\frac{lN}{n}+1\right)}=\ldots=x_{\sigma_0\left(\frac{(l+1)N}{n}\right)}=x_{l+1},\ \forall l\in\{0,\ldots n-1\}
\end{eqnarray*}
Then for $i\in[n]$, $j\in[N]$, such that
\begin{eqnarray}
j=(i-1)\frac{N}{n}+p\label{jrp}
\end{eqnarray}
for some $p\in\left[\frac{N}{n}\right]$,
we have
\begin{eqnarray*}
\eta_j^{\sigma_0}(N)=\frac{(n-i)N}{n}.
\end{eqnarray*}
Then
\begin{eqnarray}
\lambda_j=\mu_t,\ t\in\{d_i,d_i+1,\ldots,d_{i+1}-1\}\label{ljmt}
\end{eqnarray}

Under Assumption \ref{ap428}, we obtain
\begin{eqnarray*}
\phi^{(i,\sigma_0)}_{p}(N)=\frac{(n-i)N}{n}+\mu_{t}.
\end{eqnarray*}
Let $k=t-d_i$, then $k\in\left\{0,1,\ldots,d_{i+1}-d_i-1\right\}$. From (\ref{abi}), we obtain
\begin{eqnarray*}
\lim_{N\rightarrow\infty}\frac{\mu_t}{N}=a_1+\sum_{l=2}^{s-t+1}(a_l-b_{l-1})
\end{eqnarray*}
Moreover,
\begin{eqnarray*}
1-n\left(\sum_{l=s-t+2}^{s-d_i+1}(b_l-a_l)\right)\geq\lim_{N\rightarrow\infty}\frac{\frac{N}{n}-p}{\frac{N}{n}}\geq 1-n\left(\sum_{l=s-t+1}^{s-d_i+1}(b_l-a_l)\right)
\end{eqnarray*}
Hence for each fixed $t\in J_i$, for all $p\in\left[\frac{N}{n}\right]$ satisfying (\ref{ljmt}) and (\ref{jrp}),  we have
\begin{eqnarray*}
\lim_{N\rightarrow\infty}\frac{\phi_p^{(i,\sigma_0)}(N)+\frac{N}{n}-p}{\frac{N}{n}}\in\left[\beta_{i,k},\gamma_{i,k}\right]
\end{eqnarray*}
One can check that $\bm_i$ has density 1 in $(\beta_{i,k},\gamma_{i,k})$ for $k\in\left\{0,1,\ldots,d_{i+1}-d_i-1\right\}$, and density 0 everywhere else. 
Then the lemma follows.
\end{proof}

\begin{lemma}\label{l440}Let $\si_0$ satisfy (\ref{sz}), and let $\ol{\si}_0\in [\Si_N/ \Si_N^X]^r$. For $1\leq i\leq k$, assume $\frac{u_i}{x_i}$ is in an open complex neighborhood of $1$.  For any $\si\in \Si_N$, let
\begin{eqnarray*}
G_{\si}&=& \left(\prod_{i=1}^{n}x_i^{|\phi^{(i,\sigma)}(N)|}\right)\left(\prod_{i=1}^{r}s_{\phi^{(i,\sigma)}(N)}\left(\frac{u_i}{x_i},\frac{u_{n+i}}{x_i}\ldots,\frac{u_{qn+i}}{x_i},1,\ldots,1\right)\right)\notag\\
&&\times\left(\prod_{i=r+1}^{n}s_{\phi^{(i,\sigma)}(N)}\left(\frac{u_i}{x_i},\frac{u_{n+i}}{x_i}\ldots,\frac{u_{(q-1)n+i}}{x_i},1,\ldots,1\right)\right)\left(\prod_{i<j,x_{\sigma(i)}\neq x_{\sigma(j)}}\frac{1}{w_{\sigma(i)}-w_{\sigma(j)}}\right)
\end{eqnarray*}
Suppose that Assumption \ref{ap32} holds. When $\alpha$ in Assumption \ref{ap32} is sufficiently large, we have
\begin{eqnarray*}
\left|\frac{G_{\si_0}}{G_{\si}}\right|\geq e^{CN^2}
\end{eqnarray*}
where $C>0$ is a constant independent of $\si$, $N$ and $(u_1,\ldots,u_k)$.
\end{lemma}

\begin{proof}First of all note that
\begin{eqnarray*}
\left|\frac{\prod_{i<j,x_{\sigma(i)}\neq x_{\sigma(j)}}\frac{1}{w_{\sigma(i)}-w_{\sigma(j)}}}{\prod_{i<j,x_{\sigma_0(i)}\neq x_{\sigma_0(j)}}\frac{1}{w_{\sigma_0(i)}-w_{\sigma_0(j)}}}\right|=1
\end{eqnarray*}
We can express the quotient of two Schur polynomials as an HCIZ integral as follows
\begin{eqnarray*}
&&\frac{s_{\phi^{(i,\sigma)}(N)}\left(\frac{u_i}{x_i},\frac{u_{n+i}}{x_i}\ldots,\frac{u_{qn+i}}{x_i},1,\ldots,1\right)}{s_{\phi^{(i,\sigma)}(N)}\left(1,\ldots,1\right)}\\
&=&\left[\prod_{1\leq j< l\leq q+1}\frac{\log\left(\frac{u_{(j-1)n+i}}{x_i}\right)-\log\left(\frac{u_{(l-1)n+i}}{x_i}\right)}{\frac{u_{(j-1)n+i}}{x_i}-\frac{u_{(l-1)n+i}}{x_i}}\right]\left[\prod_{1\leq t\leq q+1}\prod_{q+2\leq j\leq N}\frac{\log\left(\frac{u_{(t-1)n+i}}{x_i}\right)}{\frac{u_{(t-1)n+i}}{x_i}-1}\right]\\
&&\times\int_{U(N)}e^{\mathrm{tr}(U^*A_NUB_N)}dU
\end{eqnarray*}
where
\begin{eqnarray*}
A_N&=&\mathrm{diag}\left[\log\left(\frac{u_i}{x_i}\right),\log\left(\frac{u_{n+i}}{x_i}\right),\ldots,\log\left(\frac{u_{qn+i}}{x_i}\right),0,\ldots,0\right]\\
B_N&=&\mathrm{diag}\left[\phi^{(i,\sigma)}_1(N)+N-1,\phi^{(i,\sigma)}_2(N)+N-2,\ldots,\phi^{(i,\sigma)}_N(N)\right]
\end{eqnarray*}
Under Assumptions \ref{ap423} and \ref{ap428}, we have
\begin{eqnarray*}
&&|\phi^{(i,\sigma)}_j(N)+N-j|\leq CN\\
&&\left|\log\left(\frac{u_{sn+i}}{x_i}\right)\right|\leq C,\ \mathrm{for}\ 0\leq s\leq q
\end{eqnarray*}
for all $j\in[N]$ and $i\in[n]$, where $C>0$ is a constant independent of $i,j$ and $N$.
We obtain
\begin{eqnarray*}
\left|\int_{U(N)}e^{\mathrm{tr}(U^*A_NUB_N)}dU\right|&=&\left|\int_{U(N)} e^{\sum_{1\leq i\leq q+1,1\leq j\leq N}A_N(i,i)U(i,j)B_N(j,j)\ol{U(i,j)}}dU\right|\\
&\leq & e^{(q+1)C^2N}
\end{eqnarray*}
Then there exists a constant $C>0$ (which might be different from the $C$ above, we abuse the notation here, and similar below), such that
\begin{eqnarray}
\left|\frac{s_{\phi^{(i,\sigma)}(N)}\left(\frac{u_i}{x_i},\frac{u_{n+i}}{x_i}\ldots,\frac{u_{qn+i}}{x_i},1,\ldots,1\right)}{s_{\phi^{(i,\sigma)}(N)}\left(1,\ldots,1\right)}\right|\leq e^{CN}.\label{ecn}
\end{eqnarray}
Note also that for each $i\in[N]$, as $N\rightarrow \infty$, by Lemma \ref{cm} the counting measure for $\phi^{(i,\si_0)}(N)$ converges to a measure $\mathbf{m}_i$. By Theorem 4.2 of \cite{bg}, we have
\begin{eqnarray*}
&&\lim_{N\rightarrow\infty}\frac{1}{N}\log\frac{s_{\phi^{(i,\sigma_0)}(N)}\left(\frac{u_i}{x_i},\frac{u_{n+i}}{x_i}\ldots,\frac{u_{qn+i}}{x_i},1,\ldots,1\right)}{s_{\phi^{(i,\sigma_0)}(N)}\left(1,\ldots,1\right)}\\
&=&\frac{1}{n}\left[H_{\mathbf{m}_i}\left(\frac{u_i}{x_i}\right)+H_{\mathbf{m}_i}\left(\frac{u_{n+i}}{x_i}\right)+\ldots+H_{\mathbf{m}_i}\left(\frac{u_{qn+i}}{x_i}\right)\right]
\end{eqnarray*}
where $H_{\mathbf{m}_i}$ is a function defined by (\ref{hmi}).
 In particular, $H_{\mathbf{m}_i}(u)$ is an holomorphic function of $u$ when $u$ is in an open complex neighborhood of $1$. Therefore there exists constant $C>0$, such that
\begin{eqnarray}
e^{-CN}\leq \left|\frac{s_{\phi^{(i,\sigma_0)}(N)}\left(\frac{u_i}{x_i},\frac{u_{n+i}}{x_i}\ldots,\frac{u_{qn+i}}{x_i},1,\ldots,1\right)}{s_{\phi^{(i,\sigma_0)}(N)}\left(1,\ldots,1\right)}\right|\leq e^{CN}\label{pmcn}
\end{eqnarray}
when $\frac{u_i}{x_i},\frac{u_{n+i}}{x_i},\ldots,\frac{u_{qn+i}}{x_i}$ are in an open complex neighborhood of $1$ and when $N$ is sufficiently large. Then
\begin{eqnarray*}
\left|\frac{s_{\phi^{(i,\sigma_0)}(N)}\left(\frac{u_i}{x_i},\frac{u_{n+i}}{x_i}\ldots,\frac{u_{qn+i}}{x_i},1,\ldots,1\right)}{s_{\phi^{(i,\sigma)}(N)}\left(\frac{u_i}{x_i},\frac{u_{n+i}}{x_i}\ldots,\frac{u_{qn+i}}{x_i},1,\ldots,1\right)}\right|=D\cdot E\cdot F
\end{eqnarray*}
where
\begin{eqnarray*}
D&=& \left|\frac{s_{\phi^{(i,\sigma_0)}(N)}\left(\frac{u_i}{x_i},\frac{u_{n+i}}{x_i}\ldots,\frac{u_{qn+i}}{x_i},1,\ldots,1\right)}{s_{\phi^{(i,\sigma_0)}(N)}\left(1,\ldots,1\right)}\right|\\
E&=&\left|\frac{s_{\phi^{(i,\sigma_0)}(N)}\left(1,\ldots,1\right)}{s_{\phi^{(i,\sigma)}(N)}\left(1,\ldots,1\right)}\right|\\
F&=&\left|\frac{s_{\phi^{(i,\sigma)}(N)}\left(1,\ldots,1\right)}{s_{\phi^{(i,\sigma)}(N)}\left(\frac{u_i}{x_i},\frac{u_{n+i}}{x_i}\ldots,\frac{u_{qn+i}}{x_i},1,\ldots,1\right)}\right|
\end{eqnarray*}
For some constant $C>0$, we have $D\geq e^{-CN}$ by (\ref{pmcn}), $F\geq e^{-CN}$ by (\ref{ecn}). Then the lemma follows from Proposition \ref{p436}.
\end{proof}

Recall that $J_i$ was defined as in (\ref{ji}). 

\begin{lemma}\label{l319}For $i\in[n]$, by Lemma \ref{cm} let $\mathbf{m}_i$ is the limit of counting measures for the partition $\phi^{(i,\si_0)}$ as $N\rightarrow\infty$, where the counting measure for a partition is defined by (\ref{ml}). Let $H_{\mathbf{m}_i}$ be defined by (\ref{hmi}). Under Assumptions \ref{ap423}, \ref{ap32} and \ref{ap428}, the moment generating function for $\bm_i$, as defined by (\ref{smi}), is given by
\begin{eqnarray*}
S_{\mathbf{m}_i}(z)=\log\left[\prod_{j=0}^{d_{i+1}-d_i-1}\frac{1-\beta_{i,j}z}{1-\gamma_{i,j}z}\right].
\end{eqnarray*}

\end{lemma}
\begin{proof}Under Assumption \ref{ap423}, \ref{ap32} and \ref{ap428}, the components in $\phi^{(i,\si_0)}$ takes finitely many values for all $N$. Note that if $\si_0(j)\in T_i$, then
\begin{eqnarray*}
\eta_j^{\si_0}=N-\frac{i N}{n}.
\end{eqnarray*}
Then
\begin{itemize}
\item if $\si_0(j)\in T_i$, $i\in[n]$, then there exists an integer $a$, such that $0\leq a\leq d_{i+1}-d_i-1$ and
\begin{eqnarray*}
&&\lambda_j=\mu_{d_i+a}=\sum_{l=2}^{s-d_i-a+1}(A_l-B_{l-1}-1)
\end{eqnarray*}
\end{itemize}
By Lemma \ref{cm}, for each $i\in[n]$ the $k$-th moment $M_k(\mathbf{m}_i)$ of $\mathbf{m}_i$ is 
\begin{eqnarray*}
M_k(\mathbf{m}_i)=\sum_{j=0}^{d_{i+1}-d_i-1}\frac{\gamma_{i,j}^{k+1}-\beta_{i,j}^{k+1}}{k+1}
\end{eqnarray*}
Then the Stieljes transformation for $\mathbf{m}_i$ is
\begin{eqnarray*}
\mathrm{St}_{\mathbf{m}_i}(t)&=&\frac{1}{t}+\frac{M_1(\mathbf{m}_i)}{t^2}+\frac{M_2(\mathbf{m}_i)}{t^3}+\ldots
=\sum_{j=0}^{d_{i+1}-d_i-1}\log\frac{t-\beta_{i,j}}{t-\gamma_{i,j}}.
\end{eqnarray*}
Then the lemma follows from the fact that $S_{\mathbf{m}_i}(z)=\mathrm{St}_{\mathbf{m}_i}\left(\frac{1}{z}\right)$.
\end{proof}

By (\ref{hmi}), we can also compute
\begin{eqnarray*}
H'_{\mathbf{m}_i}(u)=\frac{1}{u S_{\mathbf{m}_i}^{(-1)}(\log u)}-\frac{1}{u-1}
\end{eqnarray*}
By Lemma \ref{l319}, we obtain
\begin{eqnarray*}
H'_{\mathbf{m}_i}(u)=\frac{1}{u z}-\frac{1}{u-1};
\end{eqnarray*}
where $z$ and $u$ satisfy the following condition:
\begin{eqnarray*}
u=\prod_{j=1}^{d_{i+1}-d_i-1}\frac{1-\beta_{i,j}z}{1-\gamma_{i,j}z},\qquad \forall i\in[n].
\end{eqnarray*}

\noindent\textbf{Proof of Theorem \ref{tm2}.} By Proposition \ref{p421}, both $s_{\lambda(N)}(u_1x_{1,N},\ldots,u_kx_{1,N},x_{k+1,N},\ldots,x_{N,N})$ and $s_{\lambda(N)}(x_{1,N},\ldots,x_{N,N})$ can be expressed as a sum of $\left|[\Si_N/\Si_N^X]^r\right|$ terms. Under Assumption \ref{ap32}, we have
\begin{eqnarray*}
\left|[\Si_N/\Si_N^X]^r\right|=\frac{N!}{\left[\left(\frac{N}{n}\right)!\right]^n}
\end{eqnarray*}
By Stirling's formula we obtain
\begin{eqnarray}
\lim_{N\rightarrow\infty}\left|[\Si_N/\Si_N^X]^r\right|^{\frac{1}{N}}=n.\label{Nn}
\end{eqnarray}
By Proposition \ref{p436} and (\ref{Nn}),  when $\alpha$ in Assumption \ref{ap32} is sufficiently large,
\begin{eqnarray*}
&&s_{\lambda(N)}(x_1,\ldots,x_N)\\
&=&\left(\prod_{i=1}^{n}x_i^{|\phi^{(i,\sigma_0)}(N)|}\right)\left(\prod_{i=1}^{n}s_{\phi^{(i,\sigma_0)}(N)}(1,\ldots,1)\right)\left(\prod_{i<j,x_{\sigma_0(i)}\neq x_{\sigma_0(j)}}\frac{1}{x_{\sigma_0(i)}-x_{\sigma_0(j)}}\right)\left(1+e^{-CN^2}\right)
\end{eqnarray*} 
For $1\leq i\leq N$, let
\begin{eqnarray*}
w_{i,N}=\left\{\begin{array}{cc}u_ix_{i,N}&\mathrm{for}\ 1\leq i\leq k\\x_{i,N}&\mathrm{for}\ k+1\leq i\leq N \end{array}\right.
\end{eqnarray*}

When each one of $u_1,\ldots,u_k$ is in an open complex neighborhood of $1$, respectively, by Proposition \ref{p437} and Lemma \ref{l440}, we have
\begin{eqnarray*}
&&s_{\lambda(N)}(w_{1,N},\ldots,w_{N,N})\\
&=&\left(\prod_{i=1}^{n}x_{i,N}^{|\phi^{(i,\sigma_0)}(N)|}\right)\left(\prod_{i=1}^{r}s_{\phi^{(i,\sigma_0)}(N)}\left(u_i,u_{n+i},\ldots,u_{qn+i},1,\ldots,1\right)\right)\notag\\
&&\times\left(\prod_{i=r+1}^{n}s_{\phi^{(i,\sigma_0)}(N)}\left(u_i,u_{n+i},\ldots,u_{(q-1)n+i},1,\ldots,1\right)\right)\notag\\
&&\times\left(\prod_{i<j,x_{\sigma_0(i)}\neq x_{\sigma_0(j)}}\frac{1}{w_{\sigma_0(i),N}-w_{\sigma_0(j),N}}\right)\left(1+e^{-CN^2}\right)
\end{eqnarray*}
where $w_1,\ldots,w_N$ is defined as in Lemma \ref{l440}, for some constant $C>0$ independent of $N$.

Therefore,
\begin{eqnarray*}
\lim_{N\rightarrow\infty}\frac{1}{N}\log \frac{s_{\lambda(N)}(u_1x_{1,N},\ldots,u_kx_{k,N},x_{k+1,N},\ldots,x_{N,N})}{s_{\lambda(N)}(x_{1,N},\ldots,x_{N,N})}
=S_1+S_2,
\end{eqnarray*}
where
\begin{eqnarray*}
S_1
&=&\lim_{N\rightarrow\infty}\frac{1}{N}\log\\
&&\left[\frac{\prod_{i=1}^{r}s_{\phi^{(i,\sigma_0)}(N)}\left(u_i, u_{n+i}\ldots,u_{qn+i},1,\ldots,1\right)\prod_{i=r+1}^{n}s_{\phi^{(i,\sigma_0)}(N)}\left(u_i,u_{n+i}\ldots,u_{(q-1)n+i},1,\ldots,1\right)}{\prod_{i=1}^{n}s_{\phi^{(i,\sigma_0)}(N)}(1,\ldots,1)}\right]\\
S_2&=&\lim_{N\rightarrow\infty}\frac{1}{N}\log\left[\prod_{i<j,x_{\sigma_0(i),N}\neq x_{\sigma_0(j),N}}\frac{x_{\sigma_0(i),N}-x_{\sigma_0(j),N}}{w_{\sigma_0(i),N}-w_{\sigma_0(j),N}}\right]
\end{eqnarray*}
Note that for each $1\leq i\leq n$, $\phi^{(i,\si_0)}(N)\in\GT^+_{\frac{N}{n}}$.
By Theorem 4.2 of \cite{bg}, for each $1\leq i\leq r$, we obtain
\begin{eqnarray*}
\lim_{N\rightarrow\infty}\frac{1}{N}\log \frac{s_{\phi^{(i,\sigma_0)}(N)}\left(u_i,u_{n+i}\ldots,u_{qn+i},1,\ldots,1\right)}{s_{\phi^{(i,\sigma_0)}(N)}(1,\ldots,1)}=\frac{1}{n}\left[\sum_{t=0}^{q}H_{\mathbf{m}_i}\left(u_{i+tn}\right)\right]
\end{eqnarray*}
For each $r+1\leq i\leq n$, we obtain
\begin{eqnarray*}
\lim_{N\rightarrow\infty}\frac{1}{N}\log \frac{s_{\phi^{(i,\sigma_0)}(N)}\left(u_i,u_{n+i}\ldots,u_{(q-1)n+i},1,\ldots,1\right)}{s_{\phi^{(i,\sigma_0)}(N)}(1,\ldots,1)}=\frac{1}{n}\left[\sum_{t=0}^{q-1}H_{\mathbf{m}_i}\left(u_{i+tn}\right)\right]
\end{eqnarray*}
Therefore we have
\begin{eqnarray*}
S_1=\frac{1}{n}\left[\sum_{1\leq i\leq k, [i\mod n]\neq 0}H_{\mathbf{m}_{[i\mod n]}}\left(u_i\right)+\sum_{1\leq i\leq k, [i\mod n]= 0}H_{\mathbf{m}_n}\left(u_i\right)\right]
\end{eqnarray*}
Moreover,
\begin{eqnarray*}
S_2&=&\lim_{N\rightarrow\infty}\frac{1}{N}\log\left[\left(\prod_{i=1}^{r}\prod_{t=0}^{q}\prod_{j=i+1}^{n}\frac{x_{i,N}-x_{j,N}}{u_{i+tn}x_{i,N}-x_{j,N}}\right)\left(\prod_{i=r+1}^{n-1}\prod_{t=0}^{q-1}\prod_{j=i+1}^{n}\frac{x_{i,N}-x_{j,N}}{u_{i+tn}x_{i,N}-x_{j,N}}\right)\right]^{\frac{N}{n}}\\
&=&\lim_{N\rightarrow\infty}\frac{1}{n}\sum_{\{1\leq i\leq k,[i\mod n]\neq 0.\}}\sum_{\{j=[i\mod n]+1\}}^{n}\log\frac{x_{[i\mod n],N}-x_{j,N}}{u_{i}x_{[i\mod n],N}-x_{j,N}}\\
&=&-\frac{1}{n}\sum_{\{1\leq i\leq k,[i\mod n]\neq 0.\}}\sum_{\{j=[i\mod n]+1\}}^{n}\log\left(u_i\right)
\end{eqnarray*}
where the last identity is obtained from Assumption \ref{ap32}. Then the theorem follows.
$\hfill\Box$

\section{Periodic dimer model on contracting square-hexagon lattice with piecewise boundary conditions: limit of the moments of the counting measure}\label{s5}

In this section, we study the periodic dimer model on contracting square-hexagon lattice with edge-weight period $1\times n$ and piecewise boundary conditions by analyzing the Schur function at a general point using the formula in Theorem \ref{p421} and Corollary \ref{p437}. The main goal is to prove Theorem \ref{tm5}. The idea is to define a Schur generating function (see Definition \ref{df33}), such that the moments of the counting measure can be computed by the derivatives of the Schur generating function. The Schur generating function for the uniform perfect matchings on the hexagon lattice was defined and analyzed in \cite{bg16}; for the periodic perfect matchings on the square-hexagon lattice with periodic boundary conditions was defined and analyzed in \cite{BL17}. Here we consider the case that the edge weights are periodic and the boundary condition is piecewise.  We first recall a few lemmas proved in \cite{BL17}.

Recall that to the boundary row $\Omega=(\Omega_1<\cdots<\Omega_N)$ of a contracting
square-hexagon lattice is naturally associated a partition $\omega\in \GT_N^+$
of length $N$ by:
\begin{equation*}
  \omega=(\Omega_N-N,\dotsc,\Omega_1-1).
\end{equation*}

\begin{proposition}(\cite{BL17})\label{p16}Let $\mathcal{R}(\Omega,\check{a})$ be a contracting square-hexagon lattice, with edge weights assigned as in Assumption \ref{apew}. Then the partition function for perfect matchings on $\mathcal{R}(\Omega,\check{a})$ is given by
\begin{eqnarray*}
Z=\left[\prod_{i\in I_2}\Gamma_i\right] s_{\omega}(x_{1},\ldots,x_{N})
\end{eqnarray*}
where $\omega\in \GT_N$ describes the bottom boundary condition  of $\mathcal{R}(\Omega,\check{a})$, and for $i\in I_2$, $\Gamma_i$ is defined by
\begin{eqnarray}
\Gamma_i=\prod_{t=i+1}^{N}\left(1+y_{i}x_{t}\right).\label{gi}
\end{eqnarray}.
\end{proposition}

\begin{definition}\label{df33}Let
\begin{eqnarray}
X=(x_1,x_2,\ldots,x_N)\in\RR^N;\label{xn}
\end{eqnarray}
and $(u_1,\ldots,u_N)\in\CC^N$. 
Let $\rho_N$ be a probability measure on $\GT_N$. Then the \emph{Schur generating function} with respect to $\rho_N$, $X$ is given by
\begin{eqnarray*}
S_{\rho_N,X}(u_1,\ldots,u_N)=\sum_{\lambda\in \GT_N}\rho_N(\lambda)\frac{s_{\lambda}(u_1,\ldots,u_N)}{s_{\lambda}(x_1,\ldots,x_N)}
\end{eqnarray*}
\end{definition}

For a positive integer $s$, write $\ol{s}=[s\mod n]$.

\begin{lemma}\label{lmm212} Suppose that Assumption \ref{ap423} holds.
Let 
\begin{eqnarray*}
X^{(N-t)}&=&(x_{\ol{t+1}},\ldots,x_{\ol{N}}),\\
Y^{(t)}&=&(x_{\ol{1}},\ldots,x_{\ol{t}}) 
\end{eqnarray*}
for each integer $t$ satisfying $0\leq t\leq N-1$. Let $k\in\{2t+1,2t+2\}$. Let $\omega$ be the partition corresponding to the configuration on the boundary row,  let $\rho^k$ be the probability measure on $\GT_{N-t}^+$ which is the distribution of partitions corresponding to the dimer configuration on the $k$th row of vertices of $\mathcal{R}(\Omega,\check{a})$, counting from the bottom. Then the Schur generating function, as defined in Definition \ref{df33}, can be computed by
\begin{enumerate}
\item If $t=2k+1$, then
\begin{eqnarray*}
S_{\rho^k,X^{(N-t)}}(u_1,\ldots,u_{N-t})&=&\frac{s_{\omega}\left(u_1,\ldots,u_{N-t},Y^{(t)}\right)}{s_{\omega}(X^{(N)})}\prod_{i\in\{1,\ldots,t\}\cap I_2}\prod_{j=1}^{N-t}\left(\frac{1+y_{\ol{i}}u_j}{1+y_{\ol{i}}x_{\ol{t+j}}}\right).
\end{eqnarray*}

\item If $t=2k+2$, then
\begin{eqnarray*}
S_{\rho^k,X^{(N-t)}}(u_1,\ldots,u_{N-t})&=&\frac{s_{\omega}\left(u_1,\ldots,u_{N-t},Y^{(t)}\right)}{s_{\omega}(X^{(N)})}\prod_{i\in\{1,\ldots,t+1\}\cap I_2}\prod_{j=1}^{N-t}\left(\frac{1+y_{\ol{i}}u_j}{1+y_{\ol{i}}x_{\ol{t+j}}}\right),
\end{eqnarray*}
for $k=2t+2,\ t=0,1,\ldots,N-1$.
\end{enumerate}
where $I_2$ is defined in Definition \ref{defI1I2}.
\end{lemma}
\begin{proof}See Lemma 3.17 of \cite{BL17}.
\end{proof}

Let 
\begin{eqnarray*}
 X^{(N-t)}_N&=&(x_{t+1,N},\ldots,x_{N,N})\\
 Y^{(t)}_N&=&(x_{1,N},\ldots,x_{t,N}).
\end{eqnarray*}

By Lemma~\ref{lmm212}, we have
\begin{eqnarray*}
&&S_{\rho^k,X_N^{(N-t)}}(u_1x_{t+1,N},\ldots,u_{N-t}x_{N,N})\\
&=&\frac{s_{\lambda(N)}(u_1x_{t+1,N},\ldots,u_{N-t}x_{N,N},x_{1,N},\ldots,x_{t,N})}{s_{\lambda(N)}(x_{1,N},\ldots,x_{N,N})}
\prod_{i\in\{1,2,\ldots,t/t+1\}\cap I_2}
\prod_{j=1}^{N-t}\left(\frac{1+y_iu_jx_{t+j,N}}{1+y_ix_{t+j,N}}\right).
\end{eqnarray*}
for $k=2t+1$ or $k=2t+2$. Let
\begin{eqnarray}
j&=&\left\{\begin{array}{cc}(i+t)\mod n&\mathrm{if}\ 1\leq [(i+t)\mod n]\leq n-1\\ n&\mathrm{if}\ [(i+t)\mod n]=0\end{array}\right..\label{j}
\end{eqnarray}
When the edge weights are assigned periodically and satisfy Assumptions \ref{ap32}, letting $N\to\infty$, $\frac{t}{N}\to \kappa\in[0,1)$, by Theorem \ref{tm2} we have
\begin{multline*}
\lim_{(1-\kappa)N\rightarrow\infty}
\frac{1}{(1-\kappa)N} \log S_{\rho^k,X_N^{(N-t)}}(u_1x_{t+1,N},\ldots,u_{\ell}x_{t+\ell,N},x_{t+1+\ell,N},\ldots,x_{N-t,N})\\
=\sum_{1\leq i\leq \ell}\left[Q_{j,\kappa}(u_i)\right],
\end{multline*}
%where $\kappa=\frac{t}{N}$, and $k=2t+1$ or $k=2t+2$. Let
where $j$ and $s$ are given by (\ref{j}), respectively; and
\begin{equation*}
Q_{j,\kappa}(u)=\left\{\begin{array}{cc}\frac{1}{1-\kappa}\left[
Q_{j}(u)+\frac{\kappa}{n}\sum_{r\in\{1,2,\ldots,n\}\cap I_2}\log\frac{1+y_rux_1}{1+y_rx_1}\right]&\mathrm{if}\ [j\mod n]=1\\\frac{Q_{j}(u)}{1-\kappa}
&\mathrm{otherwise}\end{array}\right.
\end{equation*}

Let $p\geq 1$ be a positive integer. Let $\rho_{\lfloor (1-\kappa )N\rfloor}:=\rho^{2(N-\lfloor (1-\kappa )N)\rfloor)+1}$ be a probability measure on $\GT_{\lfloor (1-\kappa )N\rfloor}^+$(Indeed, we will obtain exactly the same result in the limit as $N\rightarrow\infty$ if we define $\rho_{\lfloor (1-\kappa )N\rfloor}:=\rho^{2(N-\lfloor (1-\kappa )N)\rfloor)+2}$), and let $\mathbf{m}_{\rho_{\lfloor
(1-\kappa) N\rfloor}}$ be the corresponding random counting measure. Let $\mathcal{N}=\lfloor (1-\kappa) N \rfloor$. Let 
\begin{eqnarray*}
U&=&(u_1,\ldots,u_N); \qquad
X_N=(x_{1,N},\ldots,x_{N,N});\\
U_{X,N}&=&(u_1x_{1,N},\ldots,u_N x_{N,N}).\qquad
U_{X,N}^{(N-t)}=(u_1x_{t+1,N},\ldots,u_{N-t} x_{N,N}).
\end{eqnarray*}
Suppose that $X_N$ satisfies Assumption \ref{ap32}.

For a positive integer $p$, define an operator
\begin{eqnarray*}
\mathcal{D}_{p,N}=\frac{1}{\prod_{1\leq i<j\leq N}(u_ix_{i,N}-u_jx_{j,N})}\circ\left(\sum_{i=1}^{N}\left(u_i\frac{\partial}{\partial u_i}\right)^p\right)\circ\prod_{1\leq i<j\leq N}(u_ix_{i,N}-u_jx_{j,N})
\end{eqnarray*}
where $\circ$ denotes composition of operators; the left operator and the right operator above are multiplication operators, and the middle operator above is a differential operator.

Let $\lambda\in \GT_N^+$ be a length-$N$ partition. Explicit computations show that
\begin{eqnarray}
\mathcal{D}_{p,N} s_{\lambda}(U_{X,N})=\sum_{i=1}^{N}(\lambda_i+N-i)^{p} s_{\lambda}(U_{X,N})\label{ev}
\end{eqnarray}

 We write the Schur generating function as defined by Definition \ref{df33} as 
\begin{eqnarray*}
S_{\rho_{\lfloor (1-\kappa )N\rfloor},X_N^{(N-t)}}\left(U_{X,N}^{(N-t)}\right)=\mathrm{exp}\left(\sum_{i=1}^{N}\mathcal{N}Q_{j,\kappa}(u_i)\right)T_{N,\kappa}\left(U_{X,N}^{(N-t)}\right)
\end{eqnarray*}
Since ${S}_{\rho_{\lfloor (1-\kappa )N\rfloor},X_N^{(N-t)}}\left(X_N^{(N-t)}\right)=1$, the definition of $Q_{j,\kappa}$ implies that $Q_{j,\kappa}(1)=0$ for all $j\in[N]$ and $0<\kappa<1$. Therefore $T_{N,\kappa}\left(X_N^{(N-t)}\right)=1$ and
\begin{eqnarray*}
\lim_{N\rightarrow\infty}\frac{1}{N}\log\left[T_{N,\kappa}\left(u_1x_{t+1,N},\ldots,u_lx_{t+l,N},x_{t+l+1,N},\ldots,x_{N,N}\right)\right]=0.
\end{eqnarray*}
for any fixed positive integer $\ell$. Moreover, the convergence is uniform when $(u_1,\ldots,u_{\ell})$ is in an open complex neighborhood of $(1,\ldots,1)$. Therefore any partial derivative of $T_{N,\kappa}\left(u_1x_{t+1,N},\ldots,u_lx_{t+l,N},x_{t+l+1,N},\ldots,x_{N,N}\right)$ with respect to $(u_1,\ldots,u_{\ell})$, divided by $NT_{N,\kappa}\left(u_1x_{t+1,N},\ldots,u_lx_{t+l,N},x_{t+l+1,N},\ldots,x_{N,N}\right)$ tends to 0 uniformly when $(u_1,\ldots,u_l)$ is in a certain complex neighborhood of $(1,\ldots,1)$.

We write
\begin{eqnarray*}
&&\mathcal{D}_{p,\lfloor (1-\kappa)N\rfloor}\\
&=&\frac{1}{\prod_{1\leq i<j\leq N-t}(u_ix_{i+t,N}-u_jx_{j+t,N})}\circ\left(\sum_{i=1}^{N-t}\left(u_i\frac{\partial}{\partial u_i}\right)^p\right)\circ\prod_{1\leq i<j\leq N-t}(u_ix_{i+t,N}-u_jx_{j+t,N})
\end{eqnarray*}

 By (\ref{ev}) we have

\begin{eqnarray}
&&\EE\left(\int_{\RR}x^p\mathbf{m}_{\rho_{\lfloor(1-\kappa)N\rfloor}}(dx)\right)^m\label{emc}\\
&=&\frac{1}{\mathcal{N}^{m(p+1)}}\lim_{(u_1,\ldots,u_N)\rightarrow(1,\ldots,1)}\left(\mathcal{D}_{p,\lfloor (1-\kappa)N\rfloor}\right)^m {S}_{\rho_{\lfloor (1-\kappa )N\rfloor},X_N^{(N-t)}}\left(U_{X,N}^{(N-t)}\right)\notag
\end{eqnarray}

Using the Leibnitz rule to expand $\left(\mathcal{D}_{p,\lfloor (1-\kappa)N\rfloor}\right)^m{S}_{\rho_{\lfloor (1-\kappa )N\rfloor},X_N^{(N-t)}}\left(U_{X,N}^{(N-t)}\right)$, we obtain a linear combination of terms of the following form 
\begin{eqnarray}
&&(u_{g_1}\cdots u_{g_{\gamma}})\cdot\left(\frac{\frac{\partial}{\partial u_{t_1}}\cdots\frac{\partial}{\partial u_{t_{\tau}}}\prod_{1\leq i<j\leq N-t}(u_ix_{i+t,N}-u_jx_{j+t,N})}{\prod_{1\leq i<j\leq N-t}(u_ix_{i+t,N}-u_jx_{j+t,N})}\right)\label{ep}\\
&&\times\left(\frac{\partial }{\partial u_{a_1}}\cdots\frac{\partial}{\partial u_{a_{\alpha}}}\exp\left(\sum_{i=1}^{N-t}\mathcal{N} Q_{j,\kappa}(u_i)\right)\right)\cdot\left(\frac{\partial}{\partial u_{b_1}}\cdots\frac{\partial}{\partial u_{b_{\beta}}}T_{N,\kappa}\left(U_{X,N}^{(N-t)}\right)\right),\notag
\end{eqnarray}
where $\gamma\leq mp$ and $\alpha+\beta+\tau\leq mp$. 

Further expanding the second term in (\ref{ep}) we obtain a linear combination of terms with the following form
\begin{eqnarray}
&&(u_{g_1}\cdots u_{g_{\gamma}})(x_{t_1+t}\cdots x_{t_{\tau}+t})\cdot\left(\prod_{(a,b)\in P}\frac{1}{u_ax_{a+t}-u_bx_{b+t}}\right)\label{evr}\\
&&\times\left(\frac{\partial }{\partial u_{a_1}}\cdots\frac{\partial}{\partial u_{a_{\alpha}}}\exp\left(\sum_{i=1}^{N}\mathcal{N} Q_{j,\kappa}(u_i)\right)\right)\cdot\left(\frac{\partial}{\partial u_{b_1}}\cdots\frac{\partial}{\partial u_{b_{\beta}}}T_{N,\kappa}\left(U_{X,N}^{(N-t)}\right)\right),\notag
\end{eqnarray}
where $P\subset\{(a,b)|1\leq a<b\leq N-t\}$.

Note that in each derivation of the exponent in (\ref{evr}), a multiple of $N$ appears. Also recall that for any $\beta\geq 1$,
\begin{eqnarray*}
\frac{\partial}{\partial u_{b_1}}\cdots\frac{\partial}{\partial u_{b_{\beta}}}T_{N,\kappa}\left(U_{X,N}^{(N-t)}\right)=o(N).
\end{eqnarray*}
Therefore when $m=1$ and $N$ large, the leading term for
\begin{equation*}
\mathcal{N}^{p+1}\EE\int_{\RR}x^p\mathbf{m}_{\rho_{\lfloor(1-\kappa) N\rfloor}}(dx)
\end{equation*}
is the same as that of
\begin{eqnarray*}
&&\left(\frac{1}{\prod_{1\leq i<r\leq \mathcal{N}}(u_ix_{i+t,N}-u_rx_{r+t,N})}\right)\\
&&\times\left(\left.\sum_{a=1}^{\mathcal{N}}u_a^p\frac{\partial^p}{\partial u_a^p}\left[\mathrm{exp}\left(\mathcal{N}\left(\sum_{i=1}^{\mathcal{N}}Q_{j,\kappa}(u_i)\right)\right)\prod_{b<r}(u_bx_{b+t,N}-u_rx_{r+t,N})\right]\right)\right|_{(u_1,\ldots,u_N)=(1,\ldots,1)}.
\end{eqnarray*}
The latter has the same leading term as 
\begin{eqnarray*}
\mathcal{M}_{p,\mathcal{N}}&:=&\left.\sum_{\ell=0}^{p}\sum_{i=1}^{\mathcal{N}}\left(\begin{array}{c} p\\ \ell\end{array}\right)\mathcal{N}^{p-\ell}u_i^{p}\frac{\frac{\partial^{l}}{\partial u_i^l}\prod_{1\leq b<r\leq \mathcal{N} }(u_bx_{b+t,N}-u_rx_{r+t,N})}{\prod_{1\leq b<r\leq \mathcal{N}} (u_bx_{b+t,N}-u_rx_{r+t,N})}Q'_{j,\kappa}(u_i)^{p-\ell}\right|_{(u_1,\ldots,u_N)=(1,\ldots,1)}\\
&=&\left.\sum_{\ell=0}^{p}\sum_{i=1}^{\mathcal{N}}\left(\begin{array}{c} p\\ \ell\end{array}\right)\mathcal{N}^{p-\ell}u_i^{p}\right.\\
&&\times\left.\left[\sum_{1\leq j_1<\ldots<j_{\ell}\leq \mathcal{N}, j_{s}\neq i (1\leq s\leq \ell)}\frac{x_{i+t,N}^{\ell}}{\prod_{s=1}^{\ell}(u_ix_{i+t,N}-u_{j_{s}}x_{j_{s}+t,N})}\right]Q'_{j,\kappa}(u_i)^{p-\ell}\right|_{(u_1,\ldots,u_N)=(1,\ldots,1)}\\
&\approx&\left.\sum_{i=1}^{\mathcal{N}}u_i^p\left(\mathcal{N}Q'_{j,\kappa}(u_i)+\sum_{r\in\{1,2,\ldots,\mathcal{N}\},r\neq i}\frac{x_{j,N}}{u_ix_{j,N}-u_rx_{r+t,N}}\right)^p\right|_{(u_1,\ldots,u_{N})=(1,\ldots,1)}.
\end{eqnarray*}
where $A\approx B$ means that $A$ and $B$ have the same leading term as $\mathcal{N}\rightarrow\infty$.
For $i\in[n]$, let
\begin{equation*}
S_{\mathcal{N}}(i)=\{j\in[\mathcal{N}] \}: ((j-i)\mod n)=0\}=
\{an+i; 0\leq a \leq \lfloor \mathcal{N}/n\rfloor\}.
\end{equation*}
Then
\begin{eqnarray*}
\mathcal{M}_{p,\mathcal{N}}&\approx&\lim_{(u_1,\ldots,u_N)\rightarrow (1,\ldots,1)}\sum_{i=1}^{\mathcal{N}}\sum_{l=0}^{p}\mathcal{N}^{p-l}\left(\begin{array}{c}p\\l\end{array}\right)u_i^p[Q'_{j,\kappa}(u_i)]^{p-l}\left[\sum_{k=0}^{l}\left(\begin{array}{c}l\\k\end{array}\right)\right.\\
&&\left.\times\left(\sum_{r\in\{1,2,\ldots,\mathcal{N}\}\setminus S_{\mathcal{N}}(i)}\frac{x_{j,N}}{u_ix_{j,N}-u_rx_{r+t,N}}\right)^{l-k}\left(\sum_{r\in S_{\mathcal{N}}(i)\setminus\{i\}}\frac{x_{j,N}}{u_ix_{j,N}-u_rx_{r+t,N}}\right)^{k}\right]\\
&=&  \lim_{(u_1,\ldots,u_N)\rightarrow (1,\ldots,1)} \sum_{i=1}^{\mathcal{N}}\sum_{k=0}^{p} \sum_{l=k}^{p} \frac{p!}{k!(l-k)!(p-l)!}\\
&&\times\left[\sum_{j_1,\ldots,j_k\in S_{\mathcal{N}}(i)\setminus \{i\}}\frac{u_i^p\left[\mathcal{N}Q'_{j,\kappa}(u_i)\right]^{p-l}\left(\sum_{r\in\{1,2,\ldots,\mathcal{N}\}\setminus S_{\mathcal{N}}(i)}\frac{x_{j,N}}{u_ix_{j,N}-u_rx_{r+t,N}}\right)^{l-k}x_{j,N}^k}{\prod_{r=1}^k(u_ix_{j,N}-u_{j_r}x_{j_r+t,N})}\right]\\
 &=&\lim_{(u_1,\ldots,u_N)\rightarrow (1,\ldots,1)}\sum_{i=1}^{\mathcal{N}}\sum_{k=0}^{p}\frac{p!}{k!(p-k)!}\\
&&\times \left[\sum_{j_1,\ldots,j_k\in S_{\mathcal{N}}(i)\setminus \{i\}}\frac{u_i^p\left[\mathcal{N}Q'_{j,\kappa}(u_i)+\sum_{r\in\{1,2,\ldots,\mathcal{N}\}\setminus S_{\mathcal{N}}(i)}\frac{x_{j,N}}{u_ix_{j,N}-u_rx_{r+t,N}}\right]^{p-k}x_{j,N}^k}{\prod_{s=1}^k(u_ix_{j,N}-u_{j_s}x_{j_s+t,N})}\right]
\end{eqnarray*}

We then apply the following lemma slightly adapted from~\cite{bg} after a change
of variables, to compute the limit as $(u_1,\ldots,u_N)$ approaches $(1,\ldots,1)$:
\begin{lemma}[\cite{bg}, Lemma~5.5]
  \label{l81}Let $\xi\in \CC\setminus\{0\}$ be a nonzero complex number; and $n>0$ be a positive integer. Assume $g(z)$ is analytic in a neighborhood of $\xi$. Then
\begin{equation*}
  \lim_{\forall\ i, z_i\to\xi}\sum_{j=1}^n \frac{g(z_j)}{\prod_{i\neq
  j}(z_j-z_i)} = 
%\lim_{z_i\rightarrow\xi}\left(\frac{g(z_1)}{(z_1-z_2)(z_1-z_3)\cdots(z_1-z_n)}+\frac{g(z_2)}{(z_2-z_1)(z_2-z_3)\cdots(z_2-z_n)}\right.\\
%\left.+\ldots+\frac{g(z_n)}{(z_n-z_1)(z_n-z_2)\cdots(z_n-z_{n-1})}\right)=
  \left.\frac{\partial^{n-1}}{\partial z^{n-1}}\left(\frac{g(z)}{(n-1)!}\right)\right|_{z=\xi}.
\end{equation*}
\end{lemma}

Given $x_{i\mod n}=x_i$, by Lemma~\ref{l81} and Assumption \ref{ap32} we have: 
\begin{eqnarray*}
&&\lim_{N\rightarrow\infty}\frac{\mathcal{M}_{p,\mathcal{N}}}{\mathcal{N}^{p+1}}=\lim_{N\rightarrow\infty}\lim_{(u_1,\ldots,u_N)\rightarrow
    (1,\ldots,1)}\frac{1}{n}\sum_{i=1}^{n}\sum_{k=0}^{p}\frac{p!}{k!(p-k)!}\frac{1}{\mathcal{N}^k(k)!}\\
&&\times\left.\frac{\partial^k\left[u_i^p\left(Q_{j,\kappa}'(u_i)+\frac{1}{n}\sum_{1\leq r\leq n,r\neq i}\frac{x_{j,N}}{u_ix_{j,N}-u_rx_{r+t,N}}\right)^{p-k}\right]}{\partial ^{k}u_i}\right|_{u_i=1}\\
&=&\left.\lim_{(u_1,\ldots,u_N)\rightarrow
(1,\ldots,1)}\frac{1}{n}\sum_{i=1}^{n}
\left[u_i^p\left(Q_{j,\kappa}'(u_i)+\frac{n-j}{n u_i}\right)^{p}\right]\right|_{u_i=1}
\end{eqnarray*}

Using residue we obtain
\begin{eqnarray*}
\EE\left[\int_{\RR}x^{p}\textbf{m}^{\kappa}(dx)\right]=
%=\sum_{i=1}^{n}\frac{1}{2(p+1)\pi \mathbf{i}}\oint_{x_{t+i}}\frac{dz}{z}\left(zQ_{\kappa}'(z)+\sum_{j=1}^{n}\frac{z}{n(z-x_{j})}\right)^{p+1}
\frac{1}{2(p+1)\pi \mathbf{i}}\sum_{i=1}^{n}\oint_{C_{1}}\frac{dz}{z}\left(zQ_{i,\kappa}'(z)+\frac{n-i}{n}+\frac{z}{n(z-1)}\right)^{p+1},
\end{eqnarray*}
where $C_{1}$ is a small counterclockwise contour enclosing $1$ and no other singularities of the integrand. 

Now we need to show that the moments of random measures $\textbf{m}^{\kappa}$ become deterministic as $N\rightarrow\infty$. It suffices to show that
\begin{eqnarray}
\lim_{N\rightarrow\infty}\EE\left[\left(\int_{\RR}x^{p}\textbf{m}^{\kappa}(dx)\right)^2\right]=\lim_{N\rightarrow\infty}\left[\EE\left(\int_{\RR}x^{p}\textbf{m}^{\kappa}(dx)\right)\right]^2\label{sv}
\end{eqnarray}
Let $2^{[\mathcal{N}]}$ be the collection of all the subsets of $[\mathcal{N}]$. For $1\leq i,i'\leq \mathcal{N}$, define sets $A_{i,i'},C_{i,i'}\in 2^{[\mathcal{N}]}\times 2^{[\mathcal{N}]}$ as follows. If $i\neq i'$,then
\begin{eqnarray*}
A_{i,i'}:&=&\left\{(M,M')\in 2^{[\mathcal{N}]}\times 2^{[\mathcal{N}]}:|M|=\ell,|M'|=\ell',i\notin M,i'\notin M'; \right.\\
&&\left.\mathrm{if}\ i\in M', \mathrm{then}\ i'\notin M;\mathrm{if}\ i'\in M, \mathrm{then}\ i\notin M'\right\}.
\end{eqnarray*}
If $i=i'$, then
\begin{eqnarray*}
A_{i,i}&=&\{(M,M')\in 2^{[\mathcal{N}]}\times 2^{[\mathcal{N}]}:|M|=\ell,|M'|=\ell',i\notin M,i\notin M'; M\cap M'=\emptyset\}.
\end{eqnarray*}
and
\begin{eqnarray*}
C_{i,i'}:&=&\{(M,M')\in 2^{[\mathcal{N}]}\times 2^{[\mathcal{N}]}:|M|=\ell,|M'|=\ell',i\notin M,i'\notin M'\}.
\end{eqnarray*}
Let
\begin{eqnarray*}
j'=\left\{\begin{array}{cc}(i'+t)\mod n&\mathrm{if}\ 1\leq [(i'+t)\mod n]\leq n-1\\ n&\mathrm{if}\ [(i'+t)\mod n]=0\end{array}\right..
\end{eqnarray*}
Note that the right hand side of (\ref{sv}) is 
\begin{eqnarray}
&&\left\{\sum_{\ell=0}^{p}\sum_{i=1}^{\mathcal{N}}\left(\begin{array}{c} p\\ \ell\end{array}\right)\mathcal{N}^{-\ell-1}u_i^{p}\times\left[\sum_{1\leq j_1<\ldots<j_{\ell}\leq \mathcal{N}, j_{s}\neq i (1\leq s\leq \ell)}\right.\right.\label{sv1}\\
&&\left.\left.\left.\frac{x_{i+t,N}^{\ell}}{\prod_{s=1}^{\ell}(u_ix_{i+t,N}-u_{j_{s}}x_{j_{s}+t,N})}\right]Q'_{j,\kappa}(u_i)^{p-\ell}\right|_{(u_1,\ldots,u_N)=(1,\ldots,1)}\right\}^2;\notag
\end{eqnarray}
Expanding (\ref{sv1}), we obtain
\begin{eqnarray}
&&\sum_{\ell=0}^{p}\sum_{\ell'=0}^{p}\sum_{i=1}^{\mathcal{N}}\sum_{i'=1}^{\mathcal{N}}\left(\begin{array}{c} p\\ \ell\end{array}\right)\left(\begin{array}{c} p\\ \ell'\end{array}\right)\mathcal{N}^{-\ell-\ell'-2}u_i^{p}u_{i'}^{p}\times\left[\sum_{C_{i,i'}}\right.\frac{x_{i+t,N}^{\ell}}{\prod_{r\in M}(u_ix_{i+t,N}-u_{r}x_{r+t,N})}\label{s1}\\
&&\left.\left.\frac{x_{i'+t,N}^{\ell'}}{\prod_{r'\in M'}(u_{i'}x_{i'+t,N}-u_{r'}x_{r'+t,N})}\right]Q'_{j,\kappa}(u_i)^{p-\ell}Q'_{j',\kappa}(u_{i'})^{p-\ell'}\right|_{(u_1,\ldots,u_N)=(1,\ldots,1)}.\notag
\end{eqnarray}
Let $m=2$ in (\ref{emc}), one can compute that the leading term on the left hand side of (\ref{sv}) is 
\begin{eqnarray}
&&\sum_{\ell=0}^{p}\sum_{\ell'=0}^{p}\sum_{i=1}^{\mathcal{N}}\sum_{i'=1}^{\mathcal{N}}\left(\begin{array}{c} p\\ \ell\end{array}\right)\left(\begin{array}{c} p\\ \ell'\end{array}\right)\mathcal{N}^{-\ell-\ell'-2}u_i^{p}u_{i'}^{p}\times\left[\sum_{A_{i,i'}}\right.\frac{x_{j,N}^{\ell}}{\prod_{r\in M}(u_ix_{j,N}-u_{r}x_{r+t,N})}\label{s2}\\
&&\left.\left.\frac{x_{j',N}^{\ell'}}{\prod_{r'\in M'}(u_{i'}x_{j',N}-u_{r'}x_{r'+t,N})}\right]Q'_{j,\kappa}(u_i)^{p-\ell}Q'_{j',\kappa}(u_{i'})^{p-\ell'}\right|_{(u_1,\ldots,u_N)=(1,\ldots,1)}\notag
\end{eqnarray}
It is not hard to see that
\begin{eqnarray*}
\sum_{i=1}^{\mathcal{N}}\sum_{i'=1}^{\mathcal{N}}\left|A_{i,i'}\triangle C_{i,i'}\right|=o(N^{l+l'+2})
\end{eqnarray*}
Therefore the leading term for (\ref{s1}) and (\ref{s2}) when $N$ is large are the same. Then (\ref{sv}) follows.

\section{Frozen Boundary}\label{s7}

In this section, we prove Theorem \ref{tm3}. This is obtained by analyzing the density of the limit measure obtained in Section \ref{s4}, and find explicitly the region where the density is 0 or 1. This turns out to be related to the real and complex roots of a sequence of equations. In the special case when $I_2\cap[n]=\emptyset$, for which the graph is a hexagonal lattice, or when $|I_2\cap[n]|=1$, we explicitly write down the equation of the frozen boundary, and show that the frozen boundary is a union of $n$ disjoint cloud curves, where $n$ is the size of a period. Similar approaches were used in \cite{bk} to study the frozen region of uniform perfect matchings on the square grid, and in \cite{BL17} to study the frozen region of periodic perfect matchings on the square-hexagon lattice with periodic boundary conditions. Here we shall study the periodic perfect matchings on the square-hexagon lattice with piecewise boundary conditions and prove the surprising result that when the edge weights satisfy certain conditions, the liquid region becomes disconnected.

For $i\in[n]$ and $\kappa\in(0,1)$, let  
\begin{eqnarray*}
F_{i,\kappa}(z)=zQ'_{i,\kappa}(z)+\frac{n-i}{n}+\frac{z}{n(z-1)}.
\end{eqnarray*}

We can compute the Stieltjes transform of the measure $\mathbf{m}^{\kappa}$ when $x$ is in a neighborhood of infinity, by Theorem \ref{tm5} we obtain
\begin{eqnarray*}
\mathrm{St}_{\mathbf{m}^{\kappa}}(x)&=&\sum_{j=0}^{\infty}\frac{\int_{\RR}y^j\mathbf{m}^{\kappa}(dy)}{x^{j+1}}=\sum_{i=1}^{n}\sum_{j=0}^{\infty}\frac{1}{2(j+1)\pi\mathbf{i}}\oint_{C_1}\left(\frac{F_{i,\kappa}(z)}{x}\right)^{j+1}\frac{dz}{z}\\
&=&-\frac{1}{2\pi\mathbf{i}}\sum_{i=1}^{n}\oint_{C_1}\log\left(1-\frac{F_{i,\kappa}(z)}{x}\right)\frac{dz}{z}
\end{eqnarray*}
Integration by parts we have
\begin{eqnarray*}
\mathrm{St}_{\mathbf{m}^{\kappa}}(x)=\frac{1}{2\pi\mathbf{i}}\left(\sum_{i=1}^{n}\left(\oint_{C_1}\log z\frac{\frac{d}{dz}\left(1-\frac{F_{i,\kappa}(z)}{x}\right)}{1-\frac{F_{i,\kappa}(z)}{x}}dz\right)-\oint_{C_{x_1,\ldots,x_n}} d\left(\log z\log \left(1-\frac{F_{i,\kappa}(z)}{x}\right)\right)\right)
\end{eqnarray*}

We claim that when $|x|$ is sufficiently large, $F_{i,\kappa}(z)=x$ has exactly one root in a neighborhood of $1$ for each $i\in[n]$. Indeed, $F_{i,\kappa}(z)$ has a Laurent series expansion in a neighborhood of $1$ given by
\begin{eqnarray*}
F_{i,\kappa}(z)=\frac{1}{n(z-1)}+\sum_{k=0}^{\infty}\alpha_k(z-1)^k.
\end{eqnarray*}
We can find a unique composite inverse Laurent series of $F_{i,\kappa}(z)$ given by
\begin{eqnarray*}
G_{i,\kappa}(w)=1+\sum_{k=1}^{\infty}\frac{\beta_k}{w^k},
\end{eqnarray*}
such that $F_{i,\kappa}(G_{i,\kappa}(w))=w$ when $w$ is in a neighborhood of infinity. Then
\begin{equation}
  \label{eq:rootFG}
z_i(x)=G_{i,\kappa}(x)
\end{equation}
is the unique root of $F_{i,\kappa}(z)=x$ in a neighborhood of $1$.

Since $1-\frac{F_{i,\kappa}(z)}{x}$ has exactly one zero $z_i(x)$ and one pole $1$ in a neighborhood of $1$ when $|x|$ is sufficiently large, we have
\begin{eqnarray*}
\oint_{C_1}d\left(\log z\log \left(1-\frac{F_{i,\kappa}(z)}{x}\right)\right)=0;
\end{eqnarray*}
and therefore
\begin{eqnarray}
\mathrm{St}_{\mathbf{m}^{\kappa}}(x)=\sum_{i=1}^{n}\log(z_i(x))\label{sjl}
\end{eqnarray}
when $x$ is in a neighborhood of infinity. By the complex analyticity of both sides of (\ref{sjl}), we infer that (\ref{sjl}) holds whenever $x$ is outside the support of $\mathbf{m}^{\kappa}$.

Recall that if a measure $\mu$ has a continuous density $f$ with respect to the Lebesgue measure, then 
\begin{eqnarray}
f(x)=-\lim_{\epsilon\rightarrow 0+}\frac{1}{\pi}\Im\left[\mathrm{St}_{\mu}(x+i\epsilon)\right];\label{dsf}
\end{eqnarray}
see Lemma 4.2 of \cite{bk}.

Let $(\chi,\kappa)$ be the continuous coordinates in the limit of rescaled square-hexagon lattice $\frac{1}{N}\mathcal{R}(\Omega,\check{a})$ as $N\rightarrow\infty$. The frozen region is the region consisting of all points $(\chi,\kappa)$ where the density of the counting measure $\mathbf{m}^{\kappa}$ is 0 or 1.

\begin{proposition} If the equation
\begin{eqnarray}
F_{i,\kappa}(z)=\frac{\chi}{1-\kappa}\label{fic}
\end{eqnarray}
only has real roots for all $i\in[n]$; then $(\chi,\kappa)$ is in the frozen region.
\end{proposition}

\begin{proof}The proposition follows directly from (\ref{sjl}), (\ref{dsf}) and the definition of the frozen region.
\end{proof}

\begin{proposition}\label{p72}For any $x>0$, and $i\in[n]$ the equation $F_{i,\kappa}(z)=x$ has at most one pair of complex conjugate roots.
\end{proposition}

\begin{proof}For $2\leq i\leq n$, we can write down the equation $F_{i,\kappa}(z)=x$ explicitly as follows
\begin{eqnarray*}
\frac{zH_{\mathbf{m}_i}'(z)}{n}-\frac{\kappa(n-i)}{n}+\frac{(1-\kappa)z}{n(z-1)}=x(1-\kappa);
\end{eqnarray*}
and we have
\begin{eqnarray*}
\frac{zH_{\mathbf{m}_1}'(z)}{n}+\frac{\kappa z}{n}\sum_{l\in\left([n]\cap I_2\right)}\frac{y_lx_1}{1+y_lx_1 z}-\frac{\kappa(n-1)}{n}+\frac{(1-\kappa)z}{n(z-1)}=x(1-\kappa).
\end{eqnarray*}
where $\mathbf{m}_i$ is a probability measure on an interval  which is divided into finitely many sub-intervals, and each sub-interval has probability density $0$ or $1$. For $i\in[n]$
\begin{eqnarray*}
H_{\mathbf{m}_i}'(z)=\frac{\mathrm{St}_{\mathbf{m}_i}^{(-1)}(\log(z))}{z}-\frac{1}{z-1}.
\end{eqnarray*}
By introducing additional variables $t_i$ such that $\mathrm{St}_{\mathbf{m}_i}(t_i)=\log (z)$ for $1\leq i\leq n$, one can then write for $\kappa\in(0,1)$ , $2\leq i\leq n$:
\begin{equation*}
F_{i,\kappa}(z,t_i):=
\frac{z}{(1-\kappa)n}\left(\frac{t_i}{z}-\frac{1}{z-1}-\frac{n-i}{z}\right)+\frac{z}{n(z-1)}+\frac{n-i}{n}.
\end{equation*}
and 
\begin{equation*}
F_{1,\kappa}(z,t_1):=
\frac{z}{(1-\kappa)n}\left(\frac{t_1}{z}-\frac{1}{z-1}-\frac{n-1}{z}+\kappa\sum_{r\in I_2\cap \{1,2,\ldots,n\}}\frac{y_{r}x_1}{1+y_{r}x_1z}\right)+\frac{z}{n(z-1)}+\frac{n-1}{n}.
\end{equation*}

As a consequence, injecting the expression of the moments of the limiting
measure into the definition of the Stieltjes transform, one gets implicit
equations to be solved:
for any $x\in\mathbb{C}$, finding $(z,t_i)\in
(\mathbb{C}\setminus\mathbb{R}_-)\times
\mathbb{C}\setminus\operatorname{Support}(\bm_i)$ such that

\begin{equation}
  \begin{cases}
    F_{i,\kappa}(z,t_i)=x,\\
    \mathrm{St}_{\mathbf{m}_i}(t_i)=\log(z)
  \end{cases}.
  \label{fi}
\end{equation}

Let $x\mapsto z_i^{\kappa}(x)$ be the composite inverse of $u:z\mapsto
F_{i,\kappa}\left(z,\mathrm{St}^{(-1)}_{\mathbf{m}_i}(\log z)\right)$ as given by (\ref{eq:rootFG}). Note
that $z_i^{\kappa}(x)$ is a uniformly convergent Laurent series in $x$ when $x$ is
in a neighborhood of infinity, and
\begin{eqnarray*}
z_i^{\kappa}\left(F_{i,\kappa}\left(z,\mathrm{St}^{(-1)}_{\mathbf{m}_i}(\log
z)\right)\right)=z;
\end{eqnarray*}
See Section 4.1 of~\cite{bk}.

%\begin{lemma}\label{l31}
 By (\ref{sjl}) The following identity holds when $x$ is in a neighborhood of infinity
\begin{equation*}
\mathrm{St}_{\mathbf{m}^{\kappa}}(x)=\sum_{i=1}^n\log(z_i^{\kappa}(x)),
\end{equation*}

%Let $F_{\kappa}(z,t)$ be given as in (\ref{fk}). From the equation $F_{\kappa}(z,t)=x$ of (\ref{es}), we have

The first equation of the system~\eqref{fi} is
linear in $t$ for given $x$ and $z$, which gives
with $c_r=\frac{1}{y_rx_1}$:
\begin{multline*}
t_1(z,\kappa,x)=
nx(1-\kappa)+\frac{\kappa z}{z-1}+\kappa(n-1)-\kappa z\sum_{r\in \left(I_2\cap[n]\right)}\frac{1}{z+c_r}\\
\label{tzkx}
\end{multline*}
and for $2\leq i\leq n$
\begin{equation*}
t_i(z,\kappa,x)=
nx(1-\kappa)+\kappa(n-i)+\frac{\kappa z}{z-1}.
  \label{eq:defr}
\end{equation*}
For a given value $y\in\mathbb{R}$, and fixed $x$ (and $\kappa$), we investigate
properties of the complex numbers $z$ such that $t(z,\kappa,x)=y$. In
particular, we have the following:
\begin{lemma}
  \label{l32}
  Let $c_r>0$, for $r\in I_2\cap[n]$. Let $\kappa\in (0,1)$, and
  $x,y\in\RR$.
  % Then all the roots for the the following equation in $z$
  Then 
  \begin{itemize}
\item  the the following equation in $z$
\begin{equation}
  t_1(z,\kappa,x)=y
%x(1-\kappa)+\frac{\kappa z}{z-1}-\frac{\kappa z}{n}\sum_{i\in I_2\cap\{1,\ldots,n\}}\frac{1}{c_i+z}-y=0
  \label{fe}
\end{equation}
has $m+1$ roots on the Riemann sphere $\mathbb{C}\cup\{\infty\}$, where $m$ is
the number of distinct values of $c_r$,
and all these roots are real and simple. 
\item For $2\leq i\leq n$, the equation $ t_i(z,\kappa,x)=y$ has exactly one root on the Riemann sphere $\mathbb{C}\cup\{\infty\}$.
\end{itemize}
%\cb{Add possibly now a remark about interlacement of solutions when $y_1<y_2$}
\end{lemma}

\begin{proof}The lemma follows from the same arguments as proof of Lemma 4.1 in \cite{BL17}.
 \end{proof}
 
 Let $0<\gamma_1<\cdots <\gamma_m$ be all the possible distinct values for the $c_i$, and
  $n_1,\dotsc,n_m$ be their respective multiplicities among the $c_i$'s. Let 
  \begin{eqnarray*}
  \ell=|I_2\cap[n]|.
  \end{eqnarray*}
  For $2\leq i\leq n$, let
  \begin{eqnarray*}
  K_i=nx(1-\kappa)+\kappa(n-i+1)-y
  \end{eqnarray*}
  and let
  \begin{eqnarray*}
   K_1=nx(1-\kappa)+\kappa(n-l)-y
  \end{eqnarray*}
  Define
  \begin{equation}
  H_1(z;x,y) = t(z,\kappa,x)-y =
    K_1+\left(
    \frac{\kappa}{z-1} + \kappa\sum_{j=1}^m \frac{n_j \gamma_j}{z+\gamma_j}
     \right)\label{h1}
  \end{equation}
  and for $2\leq i\leq n$, define
  \begin{equation}
    H_i(z;x,y) = t(z,\kappa,x)-y =
    K_i+\frac{\kappa}{z-1}    
    \label{eq:defHzxy}
  \end{equation}

\begin{remark}
  \label{rem:interlaceH}
  Let
  \begin{eqnarray}
  L_1=nx(1-\kappa)+\kappa(n-l).\label{l1}
  \end{eqnarray}
  For $2\leq i\leq n$, let 
   \begin{eqnarray}
  L_i=nx(1-\kappa)+\kappa(n-i+1).\label{l2}
  \end{eqnarray}
  
  Increasing the value of $y$ translates downward the graph of the function
  $z\in\mathbb{R}\mapsto H_1(z;x,y)$. Since $H_1(z;x,y)$ is decreasing with respect to $z$ in any
  interval of definition, the roots present in the bounded intervals decrease.
  The one in $(-\infty,-\gamma_m)\cup(1,+\infty)\cup\{\infty\}$ moves also to
  the left, and if it started in $\mathbb{R}_-$, when it reaches $-\infty$, it
  jumps to the right part of $(1,+\infty)$ and then continues to decrease.
  In particular, it means that if $y<y'$, the respective roots
  $z_1<\cdots<z_{m+1}$ and $z'_1 < \cdots < z'_{m+1}$ are interlaced:
  \begin{itemize}
    \item if $y<y'<L_1$,
      \begin{equation*}
	z'_1<z_1 < -\gamma_m < z'_2< z_2 <-\gamma_{m-1} <\cdots < -\gamma_1
	<z'_{m+1} < z_{m+1} < 1,
      \end{equation*}
    \item if $y<L_1<y'$,
      \begin{equation*}
	z_1 < -\gamma_m < z'_1< z_2 <-\gamma_{m-1} <\cdots < -\gamma_1
	<z'_m < z_{m+1} < 1 < z'_{m+1},
      \end{equation*}
    \item if $L_1<y<y'$,
      \begin{equation*}
	-\gamma_m < z'_1<z_1 < -\gamma_{m-1} < z'_2< z_2 <-\gamma_{m-1} <\cdots
	< 1 <z'_{m+1}< z_{m+1},
      \end{equation*}
  \end{itemize}
  The limiting case when $y$ or $y'$ is equal to $x(1-\kappa)+\frac{\kappa
  r}{n}$ is obtained by sending the corresponding root in
  $(-\infty,-\gamma_m)\cup(1,+\infty)$ to $\infty$.
  
  For $2\leq i\leq n$, we have
  \begin{itemize}
   \item if $y<y'<L_i$, $z'_1<z_1  < 1$.
 \item if $y<L_i<y'$, $z_1<1<z_1'$.
     \item if $L_i<y<y'$, $1<z_1'<z_1$.
     \end{itemize}
  \end{remark}

Rational fractions where zeros of the numerator and denominator interlace have
interesting monotonicity properties, already used for example in~\cite{OR07},
which are straightforwardly checked by induction using the 
decomposition of $R(z)$ into the sum of simple fractions:
\begin{lemma}
  \label{l33}
  \begin{enumerate}
  \item
  Let
  \begin{equation*}
    R(z)=\frac{(z-u_1)(z-u_2)\cdots(z-u_h)}{(z-v_1)(z-v_2)\cdots(z-v_h)},
  \end{equation*}
  where $\{u_i\}$ and $\{v_i\}$ are two sets of real numbers, and $h$ is a positive integer.
  \begin{itemize}
    \item If $\{u_i\}$ and $\{v_i\}$ satisfy
      \begin{equation*}
	v_1<u_1<v_2<u_2<\cdots<v_h<u_h.
      \end{equation*}
      Then $R(z)$ is monotone increasing in each one of the following intervals
      \begin{equation*}
	(-\infty,v_1), (v_1,v_2),\ldots, (v_{h-1},v_h),(v_h,\infty).
      \end{equation*}
    \item If $\{u_i\}$ and $\{v_i\}$ satisfy
      \begin{equation*}
	u_1<v_1<u_2<v_2<\cdots<u_h<v_h.
      \end{equation*}
      Then $R(z)$ is monotone decreasing in each one of the following intervals
      \begin{equation*}
	(-\infty,v_1), (v_1,v_2),\ldots, (v_{h-1},v_h),(v_h,\infty).
      \end{equation*}
  \end{itemize}
\item   Let
  \begin{equation*}
    R(z)=\frac{(z-u_1)\cdots(z-u_{h-1})}{(z-v_1)\cdots(z-v_h)}
    \text{with}\ v_1<u_1<\cdots<u_{h-1}<v_h
  \end{equation*}
  Then $R(z)$ is monotone decreasing in each one of the following intervals
  \begin{equation*}
	(-\infty,v_1), (v_1,v_2),\ldots, (v_{h-1},v_h),(v_h,\infty).
      \end{equation*}
\item Let  
  \begin{equation*}
    R(z)=\frac{(z-u_1)\cdots(z-u_{h+1})}{(z-v_1)\cdots(z-v_h)}
    \text{with}\ u_1<v_1<\cdots<v_{h}<u_{h+1}.
  \end{equation*}
   Then $R(z)$ is monotone increasing in each one of the following intervals
      \begin{equation*}
	(-\infty,v_1), (v_1,v_2),\ldots, (v_{h-1},v_h),(v_h,\infty).
      \end{equation*}
  \end{enumerate}
\end{lemma}
%\begin{lemma}\label{l34}Let $s$ be a positive integer and let $(a_1,a_2,\ldots,a_s)$ and $(b_1,b_2,\ldots,b_s)$ be two $s$-tuples of real numbers such that
%\begin{eqnarray*}
%a_1<b_1<a_2<b_2<\ldots<a_s<b_s\ \mathrm{and}\ \sum_{i=1}^{s}(b_i-a_i)=1.
%\end{eqnarray*}
%Let $\mathbf{m}_{\omega}$ be the uniform measure on the union of the intervals $[a_i,b_i]$. Then the system of equations (\ref{es}) has at most one pair of complex conjugate solutions. Moreover,
%\begin{itemize}
%\item if $x(1-\kappa)-b_i\neq 0$, for all $1\leq i\leq s$, then for each fixed $x\in \RR$, (\ref{es}) has at least $(n+1)s-1$ distinct real roots;
%\item if $x(1-\kappa)-b_i=0$ for some $i$ in $\{1,2,\ldots,s\}$, then for each fixed $x\in \RR$, (\ref{es}) has at least $(n+1)s-2$ distinct real roots.
%\end{itemize}
%\end{lemma}
%
%\cb{I think that in the statement it should be $(m+1)s-1$ (or $-2$), where $m$
%is the number of distinct $c_i$s. A short version of the proof is possible in
%the same line as Lemma~\ref{l32}.}

This is helpful to determine the number of solutions of Equation~\eqref{fic}, as
shown in the following lemma:
%\cb{I think that in the statement below it should be $(m+1)s-1$ (or $-2$), where $m$
%is the number of distinct $c_i$s instead of $n$. A short version of the proof is possible in
%the same line as Lemma~\ref{l32}. Corrected}
\begin{lemma}Let $d_i$, $\beta_{i,k}$, $\gamma_{i,k}$ be defined as in Lemma \ref{l319}. 
  \label{l34}Let
  \begin{eqnarray*}
  D_i=d_{i+1}-d_i-1&\mathrm{if}\ i\in[n]
  \end{eqnarray*}
  Then $\mathbf{m}_i$ is a measure with a density with respect to the
  Lebesgue measure equal to the indicator of a union of intervals
  $\bigcup_{r=0}^{D_i}[\beta_{i,k},\gamma_{i,k}]$, with

  %Let $s$ be a positive integer and let $(a_1,a_2,\ldots,a_s)$ and $(b_1,b_2,\ldots,b_s)$ be two $s$-tuples of real numbers such that
\begin{equation*}
\beta_{i,D_i}<\gamma_{i,D_i}<\beta_{i,D_i-1}<\gamma_{i,D_i-1}<\cdots<\beta_{i,0}<\gamma_{i,0}\quad \text{and}\quad \sum_{k=0}^{D_i}(\gamma_{i,k}-\beta_{i,k})=1;
\end{equation*}

Then the system of equations \eqref{fic} has at most one pair of complex
conjugate solutions. Moreover, let $m$ is the number of distinct $c_j$'s, for $j\in I_2\cap[n]$.
\begin{itemize}
  \item when $i=1$, if for any $0\leq k\leq D_1$, $\gamma_{1,k}\neq L_1$, 
    then for each fixed $x\in \RR$, \eqref{fic} has at least $(m+1)(D_1+1)-1$ distinct
    real roots;
  \item when $i=1$, if $\gamma_{1,k} = L_1$ for some $k\in \{0,1,\ldots,D_1\}$, then for each fixed $x\in \RR$, \eqref{fic} has at least
    $(m+1)(D_1+1)-2$ distinct real roots.
  \item when $2\leq i\leq n$, if for any $0\leq k\leq D_i$, $\gamma_{i,k}\neq L_i$, 
  then for each fixed $x\in \RR$, \eqref{fic} has at least $D_i$ distinct
    real roots;
 \item when $2\leq i\leq n$, if $\gamma_{i,k} = L_i$ for some $k\in \{0,1,\ldots,D_i\}$, 
  then for each fixed $x\in \RR$, \eqref{fic} has at least $(D_i-1)$ distinct
    real roots;
\end{itemize}

\end{lemma}

\begin{proof}
  %\cb{[Alternative short proof instead of the original one below]}
  The Stieltjes transform can be computed explicitly from the definition:
  \begin{equation}
    \mathrm{St}_{\bm_i}(t_i)=\log\prod_{k=0}^{D_i}\frac{t_i-\beta_{i,k}}{t_i-\gamma_{i,k}}.\label{stl}
  \end{equation}
  We use the second expression from \eqref{fi} and (\ref{stl}), we obtain
\begin{eqnarray*} 
z=\prod_{k=0}^{D_i}\frac{t_i-\beta_{i,k}}{t_i-\gamma_{i,k}} 
\end{eqnarray*}  
   By (\ref{h1}) and (\ref{eq:defHzxy}) get (after exponentiation)
  \begin{equation}
    z=\prod_{k=0}^{D_i} \frac{H_i(z;x,\beta_{i,k})}{H_i(z;x,\gamma_{i,k})}.
%G(z,x)=z;
    \label{gzx}
  \end{equation}
Let us suppose that 
\begin{eqnarray*}
\{\beta_{i,0},\ldots,\beta_{i,D_i},\gamma_{i,0},\ldots,\gamma_{i,D_i}\}\cap\left\{L_i\right\}=\emptyset;
\end{eqnarray*}
where $L_i$'s, for $1\leq i\leq n$, are defined by (\ref{l1}) and (\ref{l2}).

The rational fractions $\prod_{k=0}^{D_i} H_i(z;x,\beta_{i,k})$ and $\prod_{k=0}^{D_i} H_i(z;x,\gamma_{i,k})$ have the same
poles $m+1$ poles (of same order $s$) when $i=1$; and they have the same pole $1$ of order $s$ when $2\leq i\leq n$. According to Lemma~\ref{l32}  they have
$s(m+1)$ distinct real roots when $i=1$; and $s$ distinct real roots when $2\leq i\leq n$. These roots interlace.
Therefore, the ratio:
\begin{equation*}
  \prod_{r=0}^{D_i} \frac{H_i(z;x,\beta_{i,k})}{H_i(z;x,\gamma_{i,k})}
\end{equation*}
is a rational fraction of the form described in the hypotheses of
Lemma~\ref{l33}; with $h=m+1$ for $i=1$ and $h=1$ for $2\leq i\leq n$. Therefore, when $i=1$ on each bounded interval between two consecutive
poles, by monotonicity, the graph of the rational fraction will cross the first
diagonal exactly once and there are $(m+1)s-1$ such intervals.

If (no $\gamma_{1,k}$, and exactly) one $\beta_{1,k}$ is equal to $L_1$, the same argument is
applicable. The only difference is that the rational fraction on the right hand
side of Equation~\ref{gzx} has only $(s-1)(m+1)+m=s(m+1)-1$ zeros, but still
$s(m+1)$ poles. Therefore we still get the same number $s(m+1)-1$ of
intersection with the first diagonal, one on each finite interval between two
consecutive poles.

If (no $\beta_{1,k}$ and exactly) one $\gamma_{1,k}$ is equal to $L_1$, then this time the rational fraction has $s(m+1)-1$ finite real poles. Therefore,
there is only $s(m+1)-2$ roots found by this approach between two successive
poles.

Similar arguments apply when $2\leq i\leq n$.
\end{proof}

  Note that when rewriting Equation~\ref{gzx} as a polynomial equation in $z$,
  it has degree
  \begin{itemize} 
  \item When $i=1$,
  \begin{equation*}
    \begin{cases}
      s(m+1)+1 & \text{when no $b_r$ equals $L_1$},\\
      s(m+1) & \text{when a $b_r$ is equal to $L_1$}
    \end{cases}
  \end{equation*}
    \item When $2\leq i\leq n$,
  \begin{equation*}
    \begin{cases}
      s+1 & \text{when no $b_r$ equals $L_1$},\\
      s & \text{when a $b_r$ is equal to $L_1$}
    \end{cases}
  \end{equation*}
  \end{itemize}
  Indeed, in all the cases, the leading coefficients of the numerator and
  denominator of the rational fraction are distinct, thus there is no
  cancellation of the monomials of higher degree when multiplying both sides by
  the denominator. In both case, it is exactly the number of real roots we found
  plus 2. Which means that Equation~(\ref{gzx}), and thus~Equation~(\ref{fic}) has at
  most a pair of complex conjugated roots.
\end{proof}

\begin{proposition}For each $i\in[n]$, the boundary of the region such that (\ref{fic})
has only real roots and the region such that (\ref{fic}) has a pair of complex conjugate roots 
is a rational algebraic curve $C_i$ with an
  explicit parametrization $(\chi_i(t_i),\kappa_i(t_i))$ defined as follows:
  \begin{equation*}
    \chi_i(t_i)=\frac{1}{n}\left[t_i-\frac{J_i(t_i)}{J_i'(t_i)}\right],\quad
    \kappa_i(t_i)=\frac{1}{J_i'(t_i)},
  \end{equation*}
  where
 \begin{eqnarray}
 J_1(t_1)=\frac{1}{\Psi_1(t_1)-1}+(n-l)+\sum_{j=1}^{m}\frac{n_j\gamma_j}{\Psi_1(t_1)+\gamma_j};\label{j1t}
 \end{eqnarray}   
   for $2\leq i\leq n$,
\begin{eqnarray}  
J_i(t_i)=(n-i+1)+\frac{1}{\Psi_i(t_i)-1};\label{jit}
\end{eqnarray}  
and
  \begin{equation}
    \Psi_i(t_i)=\frac{(t_i-\beta_{i,0})(t_i-\beta_{i,1})\cdots(t_i-\beta_{i,D_i})}{(t_i-\gamma_{i,0})(t_i-\gamma_{i,1})\cdots(t_i-\gamma_{i,D_i})}.\label{psi}
  \end{equation}
  \end{proposition}
  
  \begin{proof}
  According to Proposition \ref{p72}, the boundary of the region such that (\ref{fic})
has only real roots and the region such that (\ref{fic}) has a pair of complex conjugate 
is given by the condition that 
  \begin{equation*}
   z
%=\frac{\prod_{i=1}^{s}H_{\chi}(a_i)}{\prod_{i=1}^{s}H_{\chi}(b_i)},
    =\prod_{k=0}^{D_i} \frac{H_i(z;\frac{\chi}{1-\kappa},\beta_{i,k})}{H_i(z;\frac{\chi}{1-\kappa},\gamma_{i,k})}
  \end{equation*}
 has double roots; where $H_i(z;x,y)$ is defined by Equation~\eqref{eq:defHzxy}.
  %\begin{eqnarray*}
  %  H_{\chi}(y)&=&[\chi-y](z-1)\prod_{i\in I_2\cap\{1,2,\ldots,n\}}(z+c_i)+\kappa z\prod_{i\in I_2\cap\{1,2,\ldots,n\}}(z+c_i)\\
  %  &&-\frac{\kappa z(z-1)}{n}\sum_{i\in I_2\cap\{1,2,\ldots,n\}}\prod_{\tiny{\begin{array}{c}j\in I_2\cap\{1,2,\ldots,n\}\\ j\neq i\end{array}}}(c_j+z).
  %\end{eqnarray*}
  We can also rewrite the system of equations \eqref{fi} as follows
  \begin{itemize}
  \item if $2\leq i\leq n$,
  \begin{equation}
      \begin{cases}
	\Psi_i(t_i)=z;\\
	%F_{\kappa}(z)=\frac{z}{1-\kappa}
	%\left(
	%\frac{t}{z}-\frac{1}{z-1}+
	%%\frac{\kappa}{n}\sum_{i\in I_2\cap\{1,2,\ldots,n\}}\frac{1}{c_i+z}
	%\frac{\kappa}{n}\sum_{j=1}^m \frac{n_j}{c_j+z}
	%\right)+\frac{z}{z-1}
	%=\frac{\chi}{1-\kappa}.
	n(1-\kappa)F_{i,\kappa}(z) =
	t_i-\kappa
\left[(n-i+1)+\frac{1}{z-1}\right]
 =n\chi.
      \end{cases}\label{pi}
    \end{equation}
    \item if $i=1$,
  \begin{equation}
      \begin{cases}
	\Psi_1(t_1)=z;\\
	%F_{\kappa}(z)=\frac{z}{1-\kappa}
	%\left(
	%\frac{t}{z}-\frac{1}{z-1}+
	%%\frac{\kappa}{n}\sum_{i\in I_2\cap\{1,2,\ldots,n\}}\frac{1}{c_i+z}
	%\frac{\kappa}{n}\sum_{j=1}^m \frac{n_j}{c_j+z}
	%\right)+\frac{z}{z-1}
	%=\frac{\chi}{1-\kappa}.
	n(1-\kappa)F_{i,\kappa}(z) =
	t_1-\kappa \left[ \frac{1}{z-1}+(n-l)+\sum_{j=1}^{m}\frac{n_j\gamma_j}{z+\gamma_j}\right]
=n\chi.
      \end{cases}\label{p1}
      \end{equation}
    \end{itemize}
 In each one of the two system of equations above,  we plug the expression of $z$ from the first equation into the second
    equation; and  for $1\leq i\leq n$, let $J_i(t_i)$ be defined as in (\ref{j1t}) and (\ref{jit}).
    Note that the condition that the resulting equation has a
    double root is equivalent to the following system of equations (where $1\leq i\leq n$)
    \begin{equation*}
      \begin{cases}
	\chi_i=\frac{t_i-\kappa J_i(t_i)}{n},\\
	1=\kappa J_i'(t_i).
      \end{cases}
    \end{equation*}
Then the parametrization of the
    curve separating the region with a pair of complex conjugate roots and the region with only real roots follows.
  \end{proof}

\begin{proposition}\label{p79}
  The curve $C_1$ (resp.\ $C_i$, for $2\leq i\leq n$) is a cloud curve of class $(m+1)(D_1+1)$ (resp.\ $(D_i+1)$), where $(D_i+1)$ is the
  number of segments in the measure $\bm_i$ for $1\leq i\leq n$, and $m$
  %is given by (\ref{dfm}).
  is the number of distinct values of $c_r=\frac{1}{y_rx_1}$ for $r\in |I_2\cap [n]|$ in one period.
  Moreover, the curve $C_i$ has the following properties
  \begin{enumerate}
  \item it is tangent to the line $\kappa =1$ with a unique tangent point for $i\in[n]$.
  \item it is tangent to the line $\kappa=0$ with $(m+1)(D_1+1)-1$ points of tangency when $i=1$, and with $D_i$ points of tangency when $2\leq i\leq n$.
 \end{enumerate}
  \end{proposition}

\begin{proof}
  We recall that the class of a curve is the degree of its dual curve. So we
  need to show that the dual curve $C_1^{\vee}$ (resp.\ $C_i^{\vee}$, for $2\leq i\leq n$) has degree $(m+1)(D_1+1)$ (resp.\ $(D_i+1)$) and is
  winding.

  We apply the classical formula to obtain from a parametrization $(x(t),
  y(t))$ of the curve $C_i$ for the frozen boundary 
  one for its dual $C_i^\vee$, 
    $(x^\vee(t),y^\vee(t))$:
    \begin{equation*}
      x^{\vee}=\frac{y'}{yx'-xy'},\quad
      y^{\vee}=-\frac{x'}{yx'-xy'}.
    \end{equation*}
    and obtain that 
    the dual curve $C_i^{\vee}$
  given in the following parametric form
  \begin{equation}
    C_i^{\vee}=\left\{\left(-\frac{n}{t_i},-\frac{J_i(t_i)}{t_i}\right)\ ;\
    t\in\mathbb{C}\cup\{\infty\}\right\}.
    \label{dual}
  \end{equation}
  from which we can read that its degree is $(m+1)(D_1+1)$ for $i=1$ and $(D_i+1)$ for $2\leq i\leq n$. To show that $C_i^{\vee}$ is
  winding, we need to look at real intersections with straight lines.

  %Recall that the class of a curve is defined to be the degree of its dual curve.
  %By Definition~\ref{df43}, it suffices to show that the dual curve
  %$C^{\vee}=(x,y)$ is winding.
  
  First, from Equation~\eqref{dual}, one sees that the first coordinate $x$ of
  the dual curve $C^\vee$ and the parameter $t$ are linked by the simple
  relation $xt_i=-1$.
  
  Using this relation to eliminate $t$ from the expression of the second
  coordinate, we obtain that the points $(x,t)$ on the dual curve satisfy the
  following implicit equation:
  \begin{equation*}
    y=\frac{x}{n} J_i\left(-\frac{n}{x}\right).
  \end{equation*}

  The points of intersection $(x(t_i),y(t_i))$ of the dual curve with a straight line
  of the form $y=cx+d$ have a parameter $t_i$ satisfying:
  \begin{equation}
    cn-dt_i=J_i(t_i).
    \label{eq:intersect_dual}
  \end{equation}
   the exact same argument as in Lemma~\ref{l34}
  (but with the role of $s$ and $(m+1)$ exchanged)
  shows that the~\eqref{eq:intersect_dual} has at least $(m+1)(D_1+1)-2$, if $i=1$, (resp.\ $D_i-1$, if $2\leq i\leq n$)
  distinct real solutions, yielding $(m+1)(D_1+1)-2$, when $i=1$, (resp.\ $D_i-1$, if $2\leq i\leq n$) points of intersections for the
  dual curve and the line $y=cx+d$.
  Moreover, if $t_0$ doesn't lie in a compact interval containing all the zeros of
  $J_i$, then any non vertical straight line passing through $(t_0,0)$ will have
  $(m+1)(D_1+1)-1$, when $i=1$, (resp.\ $D_i$, when $2\leq i\leq n$) intersections with the graph of $J_i$.  
  This means that $x_0$ in some
  closed interval, there are at least
  $(m+1)(D_1+1)-1$, when $i=1$, (resp.\ $D_i$, when $2\leq i\leq n$) real intersections of the dual curve with a line $y=cx+d$ passing
  through $(x_0,y)$, thus exactly $(m+1)(D_1+1)$, when $i=1$, (resp. $(m+1)(D_i+1)$, when $2\leq i\leq n$) real intersections, since there
  cannot be a single complex one.
  Such points $(x_0,y)$ are candidates to be the center of the dual curve.

  To consider the vertical lines $x=d$, we rewrite the equations in homogeneous
  coordinate $[x:y:z]$ and get that the line $x=dz$ intersects the curve at the
  point $[0:1:0]$ with multiplicity $(m+1)(D_1+1)-1$ when $i=1$ (resp.\ $D_i$ when $2\leq i\leq n$) so again, by the same argument
  as above, $(m+1)(D_1+1)$, when $i=1$, (resp.\ $D_i+1$, when $2\leq i\leq n$,) real intersections. The case of the line $z=0$ is similar.
  
Recall that each point on the dual curve $C_i^{\vee}$ corresponds to a tangent
line of $C_i$.
For $1\leq i\leq n$, let 
\begin{eqnarray*}
U_i&=&(t_i-\beta_{i,0})(t_i-\beta_{i,1})\cdots(t_i-\beta_{i,D_i})\\
V_i&=&(t_i-\gamma_{i,0})(t_i-\gamma_{i,1})\cdots(t_i-\gamma_{i,D_i}).
\end{eqnarray*}
 When $t_i=\infty$, we have
\begin{eqnarray*}
\lim_{t_i\rightarrow\infty}\frac{J_i(t_i)}{t_i}=\lim_{t_i\rightarrow\infty}\frac{V_i}{t_i(U_i-V_i)}.
\end{eqnarray*}
The leading term in $V_i$ is $t_i^{D_i+1}$, while the leading terms for $t_i(U_i-V_i)$ is $\left[\sum_{k=0}^{D_i}(\gamma_{i,k}-\beta_{i,k})\right]t_i^{D_i+1}$, therefore we have $\lim_{t_i\rightarrow\infty}\frac{V_i}{t_i(U_i-V_i)}=1$. Therefore we have $(0,-1)\in C_i^{\vee}$, which corresponds to the tangent line $\kappa=1$ of $C_i^{\vee}$. The unique tangent point is given by $\lim_{t_i\rightarrow\infty}\left(\chi_i(t_i),\kappa_i(t_i)\right)$.

Those $t_i$ such that $J_i(t_i)=\infty$ corresponds to tangent points with the tangent line $\kappa=0$. When $i=1$, the tangent points with the tangent line $\kappa=0$ are solutions of
\begin{eqnarray*}
\left[\Psi_1\left(-\frac{1}{x}\right)-1\right]\prod_{j=1}^{m}\left[\Psi_1\left(-\frac{1}{x}\right)-1\right]=0
\end{eqnarray*}
There are $(m+1)(D_1+1)-1$ such points. When $2\leq i\leq n$, the tangent points with the tangent line $\kappa=0$ are solutions of
\begin{eqnarray*}
\Psi_1\left(-\frac{1}{x}\right)=1
\end{eqnarray*}
and there are $D_i$ such points.
\end{proof}

\begin{lemma}\label{l70}
\begin{enumerate}
\item Assume $x_0>0$ is such that the equation (\ref{p1}) has a pair of complex conjugate roots. Let $s_1(x)$ be a real root of (\ref{p1}). Then
\begin{eqnarray*}
\left.\frac{\partial s_1(x)}{\partial x}\right|_{x=x_0}\geq 0.
\end{eqnarray*}
It is equal to 0 if and only if $s_i(x_0)=1$.
\item Let $i$ be a positive integer satisfying $2\leq i\leq n$. Assume $x_0>0$ is such that the equation (\ref{pi}) has a pair of complex conjugate roots. Let $s_i(x)$ be a real root of (\ref{pi}). Then
\begin{eqnarray*}
\left.\frac{\partial s_i(x)}{\partial x}\right|_{x=x_0}\geq 0.
\end{eqnarray*}
It is equal to 0 if and only if $s_i(x_0)=1$.
\end{enumerate}
\end{lemma}
\begin{proof}We only prove Part (1) here; Part (2) can be proved using exactly the same technique.

The derivative $s_1'(x)$ can be computed explicitly from (\ref{zg}) as follows
\begin{eqnarray*}
s_1'(x)=\frac{\frac{\partial G_1(z,x)}{\partial x}}{1-\frac{\partial G_1(z,x)}{\partial z}}
\end{eqnarray*}
First we claim that $\frac{\partial G_1(z,x)}{\partial x}\leq 0$. Note that 
\begin{eqnarray*}
G_1(z,x)=\frac{\prod_{k=0}^{D_1}\left[x-\left(\frac{\beta_{1,k}}{n(1-\kappa)}-\frac{\kappa z}{n(z-1)(1-\kappa)}-\frac{\kappa(n-1)}{n(1-\kappa)}-\frac{\kappa}{n(1-\kappa)}\sum_{r\in I_2\cap [n]}\frac{y_r x_1z}{1+y_rx_1z}\right)\right]}{\prod_{k=0}^{D_1}\left[x-\left(\frac{\gamma_{1,k}}{n(1-\kappa)}-\frac{\kappa z}{n(z-1)(1-\kappa)}-\frac{\kappa(n-1)}{n(1-\kappa)}-\frac{\kappa}{n(1-\kappa)}\sum_{r\in I_2\cap [n]}\frac{y_rx_1z}{1+y_rx_1z}\right)\right]}
\end{eqnarray*}
where
\begin{eqnarray*}
&&\frac{\beta_{1,k+1}}{n(1-\kappa)}-\frac{\kappa z}{n(z-1)(1-\kappa)}-\frac{\kappa(n-1)}{n(1-\kappa)}-\frac{\kappa}{n(1-\kappa)}\sum_{r\in I_2\cap [n]}\frac{y_r x_1z}{1+y_rx_1z}\\
&&<\frac{\gamma_{1,k+1}}{n(1-\kappa)}-\frac{\kappa z}{n(z-1)(1-\kappa)}-\frac{\kappa(n-1)}{n(1-\kappa)}-\frac{\kappa}{n(1-\kappa)}\sum_{r\in I_2\cap [n]}\frac{y_r x_1z}{1+y_rx_1z}\\
&&<\frac{\beta_{1,k}}{n(1-\kappa)}-\frac{\kappa z}{n(z-1)(1-\kappa)}-\frac{\kappa(n-1)}{n(1-\kappa)}-\frac{\kappa}{n(1-\kappa)}\sum_{r\in I_2\cap [n]}\frac{y_r x_1z}{1+y_rx_1z}.
\end{eqnarray*}
By Lemma \ref{l33}, for each fixed $z\in \RR\setminus\{1\}$, $G_i(z,x)$ is strictly decreasing in $x$ whenever it is defined. Hence $\frac{\partial G_i(z,x)}{\partial x}< 0$ whenever $z\neq 1$, and $\frac{\partial G_i(z,x)}{\partial x}= 0$ if $z=1$.

Now we show that $\frac{\partial G_i(z,x)}{\partial z}> 1$. We may also write $G_1(z,x)$ as follows:
\begin{eqnarray*}
G_1(z,x)=\frac{\prod_{k=0}^{D_1}(t_1(z,\kappa,x)-\beta_{1,k})}{\prod_{k=0}^{D_1}(t_1(z,\kappa,x)-\gamma_{1,k})}
\end{eqnarray*}
Let $m$ be the number of distinct values of $y_r$ for $r\in I_2\cap [n]$. By Lemma \ref{l32}, for each $k\in \{0,1,2,\ldots,D_i\}$, $t(z,\kappa,x)=\beta_{1,k}$ (resp.\ $t(z,\kappa,x)=\gamma_{1,k}$) has $(m+1)$ roots on the Riemann sphere $\CC\cup\{\infty\}$, and all these roots are real and simple. Hence we can write
\begin{eqnarray*}
G_1(z,x)=\frac{\prod_{k=0}^{D_1}\left(A_k\prod_{j=1}^{m+1}(z-u_{k,j})\right)}{\prod_{k=0}^{D_1}\left(B_k \prod_{j=1}^{m+1}(z-v_{k,j})\right)}
\end{eqnarray*}
where $\{u_{k,j}\}_{j=1}^{m+1}$ (resp.\ $\{v_{k,j}\}_{j=1}^{m+1}$) are roots of $t(z,\kappa,x)=\beta_{1,k}$ (resp.\ $t(z,\kappa,x)=\gamma_{i,k}$) satisfying
\begin{eqnarray*}
-\infty<u_{k,2}<\ldots<u_{k,m+1}<\infty\\
-\infty<v_{k,2}<\ldots<v_{k,m+1}<\infty
\end{eqnarray*}
and
\begin{eqnarray*}
u_{k,1}\in [-\infty,u_{k,2})\cup (u_{k,m+1},\infty]\\
v_{k,1}\in [-\infty,u_{k,2})\cup (v_{k,m+1},\infty]
\end{eqnarray*}

and
\begin{eqnarray*}
A_k=\left\{nx(1-\kappa)+\kappa(n-1)-\beta_{1,k}+(|I_2\cap[n]|+1)\kappa\right\}\prod_{r\in I_2\cap[n]}y_rx_1\\
B_k=\left\{nx(1-\kappa)+\kappa(n-1)-\gamma_{1,k}+(|I_2\cap[n]|+1)\kappa\right\}\prod_{r\in I_2\cap[n]}y_rx_1
\end{eqnarray*}

Recall that by (\ref{btik}) and (\ref{cmik}), we have
\begin{eqnarray}
\beta_{1,D_1}<\gamma_{1,D_1}<\beta_{1,D_1-1}<\gamma_{1,D_1-1}<\ldots\beta_{1,0}<\gamma_{1,0}\label{brit}
\end{eqnarray}

Recall that $c_r=\frac{1}{y_rx_1}$ and that $\phi_1<\phi_2<\ldots\phi_m$ are all the distinct numbers in $\{c_r\}_{r\in([n]\cap I_2)}$. The points 
\begin{eqnarray*}
-\phi_m<-\phi_{m-1}<\ldots<-\phi_1<1
\end{eqnarray*}
divided $\RR$ into $m$ bounded intervals and two unbounded intervals.

Remark \ref{rem:interlaceH} gives the interlacing properties of the roots  $\{u_{k,j},v_{k,j}\}_{k\in[D_1]\cup\{0\},j\in[m+1]}$ when (\ref{brit}) holds. More precisely, in the $j$th one (counting from the left) of the $m$ bounded intervals above, there exist $2(D_1+1)$ roots in $\{u_{k,j},v_{k,j}\}_{k\in[D_1]\cup\{0\},j\in[m+1]}$ given by
\begin{eqnarray*}
v_{0,j+1}<u_{0,j+1}<v_{1,j+1}<u_{1-1,j+1}<\ldots<v_{D_1,j+1}<u_{D_1,j+1}.
\end{eqnarray*}
Then $\{u_{k,j},v_{k,j}\}_{k\in[D_1]\cup\{0\},j\in[m+1]}$ divided $\RR$ into $(m+1)(D_1+1)-1$ bounded intervals and two unbounded intervals.

By Lemma \ref{l33}, for each fixed $x\in\RR$, in each one of the $(m+1)(D_1+1)+1$ intervals divided by $\{u_{k,j},v_{k,j}\}_{k\in[D_1]\cup\{0\},j\in[m+1]}$, $G_1(z,x)$ is strictly increasing in $z$. On each one of the $(m+1)(D_1+1)-1$ bounded interval divided by $\{u_{k,j},v_{k,j}\}_{k\in[D_1]\cup\{0\},j\in[m+1]}$, $G_1(z,x)$ takes every value in $(-\infty,\infty)$, hence the equation $z=G_1(z,x)$ has at least one root on each such bounded interval. There are $(m+1)(D_1+1)-1$ such bounded intervals, therefore $z=G_1(z,x)$ has at least $(m+1)(D_1+1)-1$ distinct real roots. Moreover, the roots $z=G_1(z,x)$ are those of a polynomial of degree at most $(m+1)(D_1+1)+1$, therefore when it has a pair of complex conjugate roots, each bounded interval divided by $\{u_{k,j},v_{k,j}\}_{k\in[D_1]\cup\{0\},j\in[m+1]}$ has exactly one real root, counting multiplicities. At the real root we have $\frac{\partial G_1(z,x)}{\partial z}> 1$.
\end{proof}

\subsection{Hexagonal lattice}

When $I_2=\emptyset$, the square-hexagon lattice we constructed is actually a hexagonal lattice. In this case we shall show that when $I_2=\emptyset$, for each $i\in[n]$, if a pair of complex conjugate roots exist for (\ref{fi}), then the root $z_i(x)$ as used to compute the density of the limit counting measure, can not be real. This follows from an adaptation of Lemma 4.5 in \cite{bk}, in which the uniform perfect matching on a square grid is considered.

When $I_2=\emptyset$, for each $i\in[n]$ we can write (\ref{fi}) as follows
\begin{eqnarray}\label{h75}
\begin{cases}
\frac{t_i}{1-\kappa}-\frac{\kappa}{1-\kappa}\frac{z}{z-1}=nx+\frac{\kappa(n-i)}{1-\kappa}.\\
\mathrm{St}_{\bm_i}(t_i)=\log(z).
\end{cases}
\end{eqnarray}
Then we have
\begin{eqnarray}
z=\frac{\prod_{k=0}^{D_i}\left[\frac{\kappa z}{z-1}+nx(1-\kappa)+\kappa(n-i)-\beta_{i,k}\right]}{\prod_{k=0}^{D_i}\left[\frac{\kappa z}{z-1}+nx(1-\kappa)+\kappa(n-i)-\gamma_{i,k}\right]}:=G_i(z,x)\label{zg}
\end{eqnarray}

Let $i\in[n]$. We consider a contracting hexagon lattice with boundary partition given by $\phi^{(i,\si_0)}(N)\in\GT_{\frac{N}{n}}^+$. Let $\kappa\in(0,1)$, and $\bm_i^{\kappa}$ be the limit counting measure for the partitions on the $\left\lfloor\frac{2\kappa N}{n}\right\rfloor$th row, counting from the bottom. Then using the same arguments as before, we obtain that
\begin{eqnarray*}
\mathrm{St}_{\bm_i^{\kappa}}\left(nx+\frac{\kappa(n-i)}{1-\kappa}\right)=\log (z_i^{\kappa}(x))
\end{eqnarray*}
Hence we have
\begin{eqnarray*}
z_i^{\kappa}(x)=\mathrm{exp}\left(\int_{\RR}\frac{\bm_i^{\kappa}[ds]}{nx+\frac{\kappa(n-i)}{1-\kappa}-s}\right);
\end{eqnarray*}
and 
\begin{eqnarray*}
z_i^{\kappa}(x+\mathbf{i}\epsilon)=\mathrm{exp}\left(\int_{\RR}\frac{\left(nx+\frac{\kappa(n-i)}{1-\kappa}-s-\mathbf{i}\epsilon\right)\bm_i^{\kappa}[ds]}{\left(nx+\frac{\kappa(n-i)}{1-\kappa}-s\right)^2+\epsilon^2}\right)
\end{eqnarray*}
Therefore $\Im[z_i^{\kappa}(x+\mathbf{i}\epsilon)]<0$ when $\epsilon$ is a small positive number. However, when complex roots exist for (\ref{pi}), for real root $s_i(x)$, Lemma \ref{l70} implies that $\Im[s_i(x+\mathbf{i}\epsilon)]\geq 0$ when $\epsilon$ is a small positive number. This implies that when complex roots exist for (\ref{pi}), $z_i^{\kappa}(x+\mathbf{i}\epsilon)$ cannot be real.

 Then we have the following theorem

\begin{theorem}\label{l71}Assume $I_2=\emptyset$. For the contracting hexagon lattice, $(\chi,\kappa)$ is in the frozen region if and only if (\ref{fic}) only has real roots for all $i\in[n]$. The frozen boundary consists of $n$ disjoint cloud curve $C_1,\ldots,C_n$, where for $i\in[n]$, $C_i$ is a cloud curve of class $D_i+1$ with an explicit parametrization given by 
  \begin{equation*}
    \chi_i(t_i)=\frac{1}{n}\left[t_i-\frac{J_i(t_i)}{J_i'(t_i)}\right],\quad
    \kappa_i(t_i)=\frac{1}{J_i'(t_i)},
  \end{equation*}
  where
\begin{eqnarray*}  
J_i(t_i)=(n-i+1)+\frac{1}{\Psi_i(t_i)-1};
\end{eqnarray*}  
and $\Psi_i$ is given by (\ref{psi}). Moreover, each $C_i$ is tangent to $\kappa=0$ with $D_i$ tangent points, and is tangent to $\kappa=1$ with a unique tangent point. The curve $C_n$ is tangent to $\chi=0$, and the curve $C_1$ is tangent to $\chi-r_{d_1}+\kappa-1=0$.
 \end{theorem}
 \begin{proof}For each $i\in[n]$, given the parametrization of $C_i$, the fact that $C_i$ is a cloud curve of class $D_i+1$ follows from Proposition \ref{p79}. We need to show that $C_1,\ldots, C_n$ are disjoint. Note that $C_i$ is characterized by the condition that the system (\ref{h75}) of equations have double roots (note that $x=\frac{\chi}{1-\kappa}$). We make a change of variables in (\ref{h75})
 \begin{eqnarray*}
 \begin{cases}
 \tilde{\chi}_i=n\chi+\kappa(n-i)\\
 \tilde{\kappa}_i=\kappa
 \end{cases},
 \end{eqnarray*}
 and let $\tilde{C}_i$ be the corresponding curve in the new coordinate system $\tilde{\chi}_i,\tilde{\kappa}_i$. Then $\tilde{C}_i$ is the frozen boundary of a uniform dimer model on contracting hexagon lattice with boundary condition given by $\bm_i$. For $1\leq j\leq s$, let $r_j=\lim_{N\rightarrow\infty}\frac{\mu_j(N)}{N}$.
 
  By the results in \cite{GP15,bg,bk,BL17}, $\tilde{C_i}$ satisfies the following conditions
\begin{enumerate} 
\item It is tangent to $\tilde{\kappa}_i=0$ with $D_i$ tangent points;
\item It is tangent to $\tilde{\kappa}_i=1$ with a unique tangent point;
\item It is tangent to  $\tilde{\chi}_i=nr_{d_{i+1}-1}+(n-i)$;
\item It is tangent to  $\tilde{\kappa_i}=-\tilde{\chi}_i+nr_{d_i}+n-i+1$;
\item It is in the bounded region bounded by the curves $\tilde{\kappa}_i=0$, $\tilde{\kappa}_i=1$, $\tilde{\chi}_i=0$, $\tilde{\kappa_i}=-\tilde{\chi}_i+nr_{d_i}+n-i+1$.
  \end{enumerate}
Then $C_i$ satisfies the following conditions   
\begin{enumerate} 
\item It is tangent to $\kappa_i=0$ with $D_i$ tangent points;
\item It is tangent to $\kappa_i=1$ with a unique tangent point;
\item It is tangent to  $n(\chi-r_{d_{i+1}-1})+(\kappa-1)(n-i)=0$;
\item It is tangent to  $n(\chi-r_{d_{i}})+(n-i+1)(\kappa-1)=0$;
\item It is in the bounded region $R_i$ bounded by the curves $\kappa_i=0$, $\kappa_i=1$, $n(\chi-r_{d_{i+1}-1})+(\kappa-1)(n-i)=0$, $n(\chi-r_{d_i})+(n-i+1)(\kappa-1)=0$.
  \end{enumerate}
  Under the assumption that $r_1>r_2>\ldots>r_s$, it is straightforward to check that $R_i\cap R_j=\emptyset$. Then the theorem follows.
 \end{proof}

To illustrate Theorem \ref{l71}, let us see the following example.
\begin{example}\label{l712}Consider a contracting hexagon lattice with period $1\times 2$. Let $x_1=1$, and $\frac{x_2}{x_1}\leq e^{-\alpha N}$. Assume $N$ is an integer multiple of $6$.
\begin{eqnarray*}
&&\lambda_1(N)=\lambda_2(N)=\ldots=\lambda_{\frac{N}{4}}(N)=\mu_1(N)\\
&&\lambda_{\frac{N}{4}+1}(N)=\lambda_{\frac{N}{4}+2}(N)=\ldots=\lambda_{\frac{N}{2}}(N)=\mu_2(N)\\
&&\lambda_{\frac{N}{2}+1}(N)=\lambda_{\frac{N}{2}+2}(N)=\ldots=\lambda_{\frac{2N}{3}}(N)=\mu_3(N)\\
&&\lambda_{\frac{2N}{3}+1}(N)=\lambda_{\frac{N}{2}+2}(N)=\ldots=\lambda_{\frac{5N}{6}}(N)=\mu_4(N)\\
&&\lambda_{\frac{5N}{6}+1}(N)=\lambda_{\frac{N}{2}+2}(N)=\ldots=\lambda_{N}(N)=\mu_5(N)=0.
\end{eqnarray*}
For $1\leq j\leq 5$, let
\begin{eqnarray*}
r_j=\lim_{N\rightarrow\infty}\frac{\mu_j(N)}{N}.
\end{eqnarray*}
Note that $r_5=0$. Then we have $\phi^{(1,\si_0)}(N)\in\GT_{\lfloor\frac{N}{2} \rfloor}^+$ is given by
\begin{eqnarray*}
\phi_i^{(1,\si_0)}(N)=\begin{cases}\mu_1(N)+\frac{N}{2},\ \mathrm{if}\ 1\leq i\leq \frac{N}{4}\\ \mu_2(N)+\frac{N}{2},\ \mathrm{if}\ \frac{N}{4}+1\leq i\leq \frac{N}{2} \end{cases}.
\end{eqnarray*}
and $\phi^{(2,\si_0)}(N)\in\GT_{\lfloor\frac{N}{2} \rfloor}^+$ is given by
\begin{eqnarray*}
\phi_i^{(2,\si_0)}(N)=\begin{cases}\mu_3(N),\ \mathrm{if}\ 1\leq i\leq \frac{N}{6}\\ \mu_4(N),\ \mathrm{if}\ \frac{N}{6}+1\leq i\leq \frac{N}{3}\\ 0,\ \mathrm{if}\ \frac{N}{3}+1\leq i\leq \frac{N}{2} \end{cases}.
\end{eqnarray*}
Hence $\bm_1$ is the uniform measure on $\left[2r_2+1,2r_2+\frac{3}{2}\right]\cup \left[2r_1+\frac{3}{2}, 2r_1+2\right]$; and $\bm_2$ is the uniform measure on $\left[0,\frac{1}{3}\right]\cup\left[2r_4+\frac{1}{3},2r_4+\frac{2}{3}\right]\cup\left[2r_3+\frac{2}{3},2r_3+1\right]$. Then we have the following two systems of linear equations
\begin{eqnarray*}
\begin{cases}
t_1-\frac{\kappa z}{z-1}=2\chi+\kappa\\
z=\frac{\left(t_1-2r_1-\frac{3}{2}\right)(t_1-2r_2-1)}{(t_1-2r_1-2)\left(t_1-2r_2-\frac{3}{2}\right)}
\end{cases}
\end{eqnarray*}
and
\begin{eqnarray*}
\begin{cases}
t_2-\frac{\kappa z}{z-1}=2\chi\\
z=\frac{t_2\left(t_2-2r_4-\frac{1}{3}\right)\left(t_2-2r_3-\frac{2}{3}\right)}{(t_2-\frac{1}{3})\left(t_2-2r_4-\frac{2}{3}\right)(t_2-2r_3-1)}
\end{cases}
\end{eqnarray*}
Then
\begin{eqnarray*}
&&\Psi_1(t_1)=\frac{\left(t_1-2r_1-\frac{3}{2}\right)\left(t_1-2r_2-1\right)}{\left(t_1-2r_1-2\right)\left(t_1-2r_2-\frac{3}{2}\right)}; \qquad\Psi_2(t_2)=\frac{t_2\left(t_2-2r_4-\frac{1}{3}\right)\left(t_2-2r_3-\frac{2}{3}\right)}{(t_2-\frac{1}{3})\left(t_2-2r_4-\frac{2}{3}\right)(t_2-2r_3-1)}.\\
&&J_1(t_1)=\frac{1}{\Psi_1(t_1)-1}+2;\qquad J_2(t_2)=\frac{1}{\Psi_2(t_2)-1}+1.
\end{eqnarray*}
The boundary separating the region where the first system has only real roots and the first system has a pair of complex conjugate roots is given by
\begin{eqnarray*}
\begin{cases}
\chi_1(t_1)=\frac{1}{2}\left[t_1-\frac{J_1(t_1)}{J_1'(t_1)}\right]\\
\kappa_1(t_1)=\frac{1}{J_1'(t_1)}
\end{cases}
\end{eqnarray*}
The boundary separating the region where the second system has only real roots and the second system has a pair of complex conjugate roots is given by
\begin{eqnarray*}
\begin{cases}
\chi_2(t_2)=\frac{1}{2}\left[t_2-\frac{J_2(t_2)}{J_2'(t_2)}\right]\\
\kappa_2(t_2)=\frac{1}{J_2'(t_2)}
\end{cases}
\end{eqnarray*}
For $(r_1,r_2,r_3,r_4)=(12,8,5,2)$; see Figure \ref{fig:fb} for a picture of the frozen boundary.

\begin{figure}
  \includegraphics[width=1.3\textwidth]{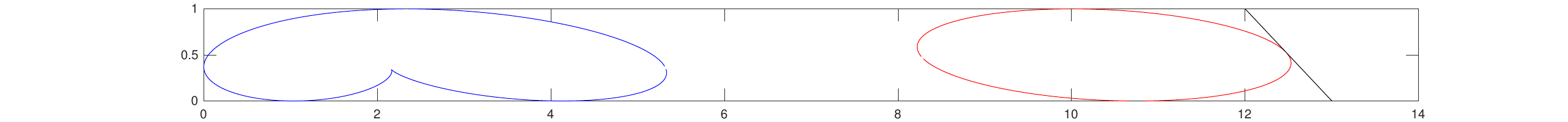}
\caption{Frozen boundary for a contracting hexagonal lattice when $n=2$, $(r_1,r_2,r_3,r_4)=(12,8,5,2)$, represented by the union of the red curve and the blue curve.}
\label{fig:fb}
\end{figure}
\end{example}

\subsection{Square-hexagonal lattice with $|I_2\cap [n]|=1$}

Now we consider the case when
\begin{eqnarray*}
|I_2\cap[n]|=\{r\},
\end{eqnarray*}
where $r$ is a positive integer satisfying $1\leq r\leq n$. Heuristically, there is exactly one row with the structure of a square grid in each period; and all the other rows in the period has the structure of a hexagon lattice.
In this case, for each $2\leq i\leq n$ we can write (\ref{fi}) as follows
\begin{eqnarray*}
\begin{cases}
\frac{t_i}{1-\kappa}-\frac{\kappa}{1-\kappa}\frac{z}{z-1}=nx+\frac{\kappa(n-i)}{1-\kappa}.\\
\mathrm{St}_{\bm_i}(t_i)=\log(z).
\end{cases}
\end{eqnarray*}
When $i=1$, (\ref{fi}) can be written as
\begin{eqnarray*}
\begin{cases}
\frac{t_1}{1-\kappa}-\frac{\kappa}{1-\kappa}\frac{z}{z-1}+\frac{\kappa}{1-\kappa}\frac{z}{z+c_r}=nx+\frac{\kappa(n-1)}{1-\kappa}.\\
\mathrm{St}_{\bm_i}(t_i)=\log(z).
\end{cases}
\end{eqnarray*}
where $c_r=\frac{1}{y_r x_1}$.

Then we have for $2\leq i\leq n$,
\begin{eqnarray}
z=\frac{\prod_{k=0}^{D_i}\left[\frac{\kappa z}{z-1}+nx(1-\kappa)+\kappa(n-i)-\beta_{i,k}\right]}{\prod_{k=0}^{D_i}\left[\frac{\kappa z}{z-1}+nx(1-\kappa)+\kappa(n-i)-\gamma_{i,k}\right]}:=G_i(z,x)\label{zg2}
\end{eqnarray}
and for $i=1$,
\begin{eqnarray}
z=\frac{\prod_{k=0}^{D_i}\left[\frac{\kappa z}{z-1}-\frac{\kappa z}{z+c_r}+nx(1-\kappa)+\kappa(n-1)-\beta_{i,k}\right]}{\prod_{k=0}^{D_i}\left[\frac{\kappa z}{z-1}-\frac{\kappa z}{z+c_r}+nx(1-\kappa)+\kappa(n-i)-\gamma_{i,k}\right]}:=G_i(z,x)\label{zg1}
\end{eqnarray}

\begin{lemma}Assume $x_0>0$ is such that equation (\ref{zg1}) (resp. (\ref{zg2})) has a pair of complex conjugate roots when $i=1$ (resp. $2\leq i\leq n$). Let $s_i(x)$ be a real root of (\ref{zg}). Then
\begin{eqnarray*}
\left.\frac{\partial s_i(x)}{\partial x}\right|_{x=x_0}\geq 0.
\end{eqnarray*}
It is equal to 0 if and only if $s_i(x_0)=1$.
\end{lemma}
\begin{proof}When $2\leq i\leq n$, the lemma follows from lemma \ref{l70}. When $i=1$, (\ref{zg1}) is the same as the equation for rectangular Aztec diamond with period $1\times 1$  and parameter $q=c_r$, boundary condition given by $\bm_i$; see equation (8.6) of \cite{bk}. Then the lemma follows from the same argument as the proof of Lemma 4.5 in \cite{bk}.
\end{proof}

Now let us consider a rectangular Aztec diamond (resp. contracting hexagon lattice) with boundary partition given by $\phi^{(i,\si_0)}(N)\in\GT_{\frac{N}{n}}^+$ when $i=1$ (resp. $2\leq i\leq n$). Let $\kappa\in(0,1)$, and $\bm_i^{\kappa}$ be the limit counting measure for the partitions on the $\left\lfloor\frac{2\kappa N}{n}\right\rfloor$th row, counting from the bottom. Then using the same arguments as before, we obtain that
\begin{eqnarray*}
\mathrm{St}_{\bm_i^{\kappa}}\left(nx+\frac{\kappa(n-i)}{1-\kappa}\right)=\log (z_i^{\kappa}(x))
\end{eqnarray*}
Hence we have
\begin{eqnarray*}
z_i^{\kappa}(x)=\mathrm{exp}\left(\int_{\RR}\frac{\bm_i^{\kappa}[ds]}{nx+\frac{\kappa(n-i)}{1-\kappa}-s}\right);
\end{eqnarray*}
and 
\begin{eqnarray*}
z_i^{\kappa}(x+\mathbf{i}\epsilon)=\mathrm{exp}\left(\int_{\RR}\frac{\left(nx+\frac{\kappa(n-i)}{1-\kappa}-s-\mathbf{i}\epsilon\right)\bm_i^{\kappa}[ds]}{\left(nx+\frac{\kappa(n-i)}{1-\kappa}-s\right)^2+\epsilon^2}\right)
\end{eqnarray*}
Therefore $\Im[z_i^{\kappa}(x+\mathbf{i}\epsilon)]<0$ when $\epsilon$ is a small positive number. However, when complex roots exist for (\ref{zg}), for real root $s_i(x)$, Lemma \ref{l70} implies that $\Im[s_i(x+\mathbf{i}\epsilon)]>0$ when $\epsilon$ is a small positive number. This implies that when complex roots exist for (\ref{zg}), $z_i^{\kappa}(x+\mathbf{i}\epsilon)$ cannot be real. Then we have the following theorem

\begin{theorem}\label{l74}Assume $I_2\cap[n]=\{r\}$, where $r$ is a positive integer satisfying $1\leq r\leq n$. For the contracting square-hexagon lattice, $(\chi,\kappa)$ is in the frozen region if and only if (\ref{fic}) only has real roots for all $1\leq i\leq n$. The frozen boundary consists of $n$ disjoint cloud curve $C_1,\ldots,C_n$, where for $2\leq i\leq n$, $C_i$ is a cloud curve of class $D_i+1$; $C_1$ is a cloud curve of class $2(D_1+1)$. The curve $C_i$ has an explicit parametrization given by 
  \begin{equation*}
    \chi_i(t_i)=\frac{1}{n}\left[t_i-\frac{J_i(t_i)}{J_i'(t_i)}\right],\quad
    \kappa_i(t_i)=\frac{1}{J_i'(t_i)},
  \end{equation*}
  where for $2\leq i\leq n$;
\begin{eqnarray*}  
J_i(t_i)=(n-i+1)+\frac{1}{\Psi_i(t_i)-1};
\end{eqnarray*}
for $i=1$
\begin{eqnarray*}
J_1(t_1)=\frac{1}{\Psi_1(t_1)-1}+n-1+\frac{c_r}{\Psi_1(t_1)+c_r}
\end{eqnarray*}  
and $\Psi_i$ is given by (\ref{psi}). Moreover, for $2\leq i\leq n$, each $C_i$ is tangent to $\kappa=0$ with $D_i$ tangent points, and $C_1$ is tangent to $\kappa=0$ with $2 D_i+1$ points. For $1\leq i\leq n$, $C_i$ is tangent to $\kappa=1$ with a unique tangent point. The curve $C_n$ is tangent to $\chi=0$, and the curve $C_1$ is tangent to $\chi-r_{d_1}+\frac{1}{2}(\kappa-2)=0$.
\end{theorem}

To illustrate Theorem \ref{l74}, let us see the following example.
\begin{example}Consider a contracting square-hexagon lattice with period $1\times 2$. Let $x_1=1$, and $\frac{x_2}{x_1}\leq e^{-\alpha N}$. Assume $N$ is an integer multiple of $6$.
Let $\lambda(N)$, $\mu(N)$, $r_j\ (1\leq j\leq 5)$, $\phi^{i,\si_0}\ (1\leq i\leq 2)$, $\bm_i$ be given as in Example \ref{l712}. 
Then we have the following two systems of linear equations
\begin{eqnarray*}
\begin{cases}
t_1-\frac{\kappa z}{z-1}=2\chi+\kappa\\
z=\frac{\left(t_1-2r_1-\frac{3}{2}\right)(t_1-2r_2-1)}{(t_1-2r_1-2)\left(t_1-2r_2-\frac{3}{2}\right)}
\end{cases}
\end{eqnarray*}
and
\begin{eqnarray*}
\begin{cases}
t_2-\frac{\kappa z}{z-1}+\frac{\kappa z}{z+c_r}=2\chi\\
z=\frac{t_2\left(t_2-2r_4-\frac{1}{3}\right)\left(t_2-2r_3-\frac{2}{3}\right)}{(t_2-\frac{1}{3})\left(t_2-2r_4-\frac{2}{3}\right)(t_2-2r_3-1)}
\end{cases}
\end{eqnarray*}
Then
\begin{eqnarray*}
&&\Psi_1(t_1)=\frac{\left(t_1-2r_1-\frac{3}{2}\right)\left(t_1-2r_2-1\right)}{\left(t_1-2r_1-2\right)\left(t_1-2r_2-\frac{3}{2}\right)}; \qquad\Psi_2(t_2)=\frac{t_2\left(t_2-2r_4-\frac{1}{3}\right)\left(t_2-2r_3-\frac{2}{3}\right)}{(t_2-\frac{1}{3})\left(t_2-2r_4-\frac{2}{3}\right)(t_2-2r_3-1)}.\\
&&J_1(t_1)=\frac{1}{\Psi_1(t_1)-1}+1+\frac{c_r}{\Psi_1(t_1)+c_r};\qquad J_2(t_2)=\frac{1}{\Psi_2(t_2)-1}+1.
\end{eqnarray*}
The boundary separating the region where the first system has only real roots and the first system has a pair of complex conjugate roots is given by
\begin{eqnarray*}
\begin{cases}
\chi_1(t_1)=\frac{1}{2}\left[t_1-\frac{J_1(t_1)}{J_1'(t_1)}\right]\\
\kappa_1(t_1)=\frac{1}{J_1'(t_1)}
\end{cases}
\end{eqnarray*}
The boundary separating the region where the second system has only real roots and the second system has a pair of complex conjugate roots is given by
\begin{eqnarray*}
\begin{cases}
\chi_2(t_2)=\frac{1}{2}\left[t_2-\frac{J_2(t_2)}{J_2'(t_2)}\right]\\
\kappa_2(t_2)=\frac{1}{J_2'(t_2)}
\end{cases}
\end{eqnarray*}
For $(r_1,r_2,r_3,r_4)=(12,8,5,2), c_r=\frac{1}{2}$; see Figure \ref{fig:fbr} for a picture of the frozen boundary.

\begin{figure}
  \includegraphics[width=1.3\textwidth]{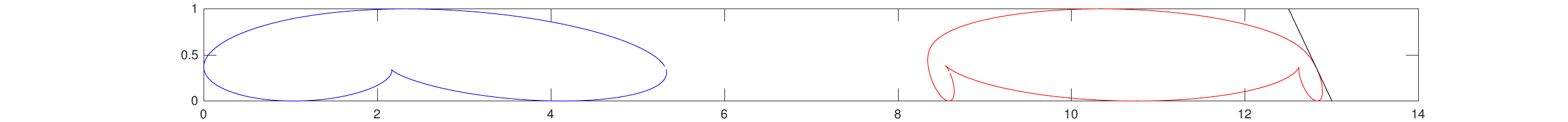}
\caption{Frozen boundary for a contracting square hexagon lattice with $n=2$, $|I_2\cap\{1,2\}|=1$ when $(r_1,r_2,r_3,r_4)=(12,8,5,2),c_r=\frac{1}{2}$, represented by the union of the red curve and the blue curve.}
\label{fig:fbr}
\end{figure}
\end{example}

\section{Appendix}\label{s8}

In this section, we give concrete examples to illustrate the combintorial formula to compute the Schur functions. Example \ref{e422} is to illustrate Theorem \ref{p421}.

\begin{example}\label{e422}Let $N=4$, $\lambda=(3,3,3,1)$, and $X=(x_1,x_2,x_1,x_2)$. Assume $x_1\neq x_2$ Then
\begin{eqnarray}
s_{\lambda}(x_1,x_2,x_1,x_2)=x_1^4x_2^4(3x_1^2+4x_1x_2+3x_2^2)\label{sch}
\end{eqnarray}
Moreover,
\begin{eqnarray*}
|[\Si_4/\Si_4^{X}]^r|=\frac{4!}{2!2!}=6.
\end{eqnarray*}
We find an representative for each right cosets in $[\Si_4/\Si_4^{X}]^r$, as follows:
\begin{eqnarray*}
\sigma_1=\mathrm{id};\qquad\sigma_1=(12);\qquad \sigma_3=(34);\\
\sigma_4=(23);\qquad\sigma_5=(14);\qquad \sigma_6=(12)(34).
\end{eqnarray*}
Then we can compute
\begin{eqnarray*}
(\eta_1^{\si_k},\eta_2^{\si_k},\eta_3^{\si_k},\eta_4^{\si_k})=\left\{\begin{array}{cc}(2,1,1,0).&\mathrm{if}\ k=1,2,3,6\\ (2,2,0,0)&\mathrm{if}\ k=4,5 \end{array}\right.
\end{eqnarray*}
and
\begin{eqnarray*}
\phi^{(1,\sigma_1)}=(5,4),\qquad \phi^{(2,\sigma_1)}=(4,1)\\
\phi^{(1,\sigma_2)}=(4,4),\qquad \phi^{(2,\sigma_2)}=(5,1)\\
\phi^{(1,\sigma_3)}=(5,1),\qquad \phi^{(2,\sigma_3)}=(4,4)\\
\phi^{(1,\sigma_4)}=(5,5),\qquad \phi^{(2,\sigma_4)}=(3,1)\\
\phi^{(1,\sigma_5)}=(3,1),\qquad \phi^{(2,\sigma_5)}=(5,5)\\
\phi^{(1,\sigma_6)}=(4,1),\qquad \phi^{(2,\sigma_6)}=(5,4)
\end{eqnarray*}
Computing the right hand side of (\ref{sws}), we obtain exactly the right hand side of (\ref{sch}).
\end{example}

\bigskip
\bigskip

\noindent\textbf{Acknowledgements.} ZL thanks C\'edric Boutillier for comments. ZL's research is supported by National Science Foundation grant DMS 1608896.

\bibliography{fpmm}
\bibliographystyle{plain}
\end{document}